\documentclass[12pt,a4paper]{amsart}	
 
\let\latexput\put
\usepackage{pictexwd}

\let\put\latexput

\usepackage{enumerate, amsmath, amsthm, amsfonts, amssymb, xy,  mathrsfs, graphicx, array}
\usepackage[usenames, dvipsnames]{color}
\usepackage{paralist}
\usepackage{lscape}
\usepackage{linegoal}
\usepackage{graphicx}
\usepackage{subcaption}				
\usepackage{filecontents}
\usepackage{graphics}
\usepackage{xcolor}
\usepackage{relsize}
\usepackage[margin=1in]{geometry}
\usepackage{tikz-cd}
\usepackage{fancyhdr,color}
\usepackage{tikz,graphicx,multicol}
\usetikzlibrary{shapes.geometric}
\usetikzlibrary{knots}
\usepackage{amssymb,euscript,eufrak, enumitem}
\usepackage{amsfonts,amsmath,amsthm}
\usepackage{dcpic}
\usepackage{pictexwd}
\usepackage{comment}

\usepackage{amscd}
\usepackage{marvosym}
\usepackage{bm, bbm}
\usepackage{braket}
\usepackage[latin2]{inputenc}
\usepackage{t1enc}
\usepackage{indentfirst}
\usepackage{float}
\usepackage{faktor}
\usepackage{mathtools}
\usepackage{caption}			
\usepackage{adjustbox}
\usepackage[outline]{contour}
\usepackage{aligned-overset}
\usepackage{booktabs}
\usepackage{tablefootnote}
\usepackage[many]{tcolorbox}
\usepackage[hidelinks]{hyperref}

\allowdisplaybreaks
\usetikzlibrary{matrix,arrows,decorations.markings,decorations.pathreplacing}
\numberwithin{equation}{section}

\theoremstyle{definition}
\newtheorem{theorem}{Theorem}[section]  

\newtheorem{proposition}[theorem]{Proposition}
\newtheorem{lemma}[theorem]{Lemma}
\newtheorem{corollary}[theorem]{Corollary}

\newtheorem{problem}[theorem]{Problem}
\newtheorem{example}[theorem]{Example}

\newtheorem{oproblem}{Problem}

\DeclareMathOperator{\Hom}{Hom}
\DeclareMathOperator{\End}{End}
\DeclareMathOperator{\diag}{diag}
\DeclareMathOperator{\co}{co}

\newcommand{\CC}{\mathbb{C}}

\newcommand{\ZZ}{\mathbb{Z}}

\newcommand{\RR}{\mathbb{R}}
\newcommand{\C}{\mathcal{C}}

\newtheorem{remark}[theorem]{Remark}

\DeclareMathOperator{\ev}{ev}
\DeclareMathOperator{\coev}{coev}
\DeclareMathOperator{\id}{id}
\DeclareMathOperator{\Irr}{Irr}
\DeclareMathOperator{\tr}{tr}
\DeclareMathOperator{\spn}{span}

\newcommand{\KP}[1]{
  \begin{tikzpicture}[baseline=-\dimexpr\fontdimen22\textfont2\relax]
  #1
  \end{tikzpicture}
}

\newcommand{\loopy}{
  \KP{\filldraw[color=black, fill=none, thick] circle (0.3);}
}

\newcommand{\vpt}{
\KP{
\begin{scope}[scale=.4]
\begin{knot}[consider self intersections=no splits,draft mode=off,flip crossing=1, clip width=6]
\strand[thick,->] (0,2) .. controls (0,-1) and (2,-1) .. (2,0)..controls (2,1) and (0,1)..(0,-2) node[left] at (0,-1){$i$};
\end{knot}
\end{scope}
    }
}

\newcommand{\vnt}{
\KP{
\begin{scope}[scale=.4]
\begin{knot}[consider self intersections=no splits,draft mode=off, clip width=6]
\strand[thick,->] (0,2) .. controls (0,-1) and (2,-1) .. (2,0)..controls (2,1) and (0,1)..(0,-2) node[left] at (0,-1){$i$};
\end{knot}
\end{scope}
    }
}

\newcommand{\vptt}{
\KP{
\begin{scope}[scale=.4]
\begin{knot}[consider self intersections=no splits,draft mode=off,flip crossing=1, rotate=180, clip width=6]
\strand[thick,<-] (0,2) .. controls (0,-1) and (2,-1) .. (2,0)..controls (2,1) and (0,1)..(0,-2) node[right] at (0,1){$i$};
\end{knot}
\end{scope}
    }
}

\newcommand{\vntt}{
\KP{
\begin{scope}[scale=.4]
\begin{knot}[consider self intersections=no splits,draft mode=off, rotate=180, clip width=6]
\strand[thick,<-] (0,2) .. controls (0,-1) and (2,-1) .. (2,0)..controls (2,1) and (0,1)..(0,-2) node[right] at (0,1){$i$};
\end{knot}
\end{scope}
    }
}

\newcommand{\px}{
  \KP{
    \draw[color=black,thick] (-0.3,-0.3) -- (0.3,0.3);
    \draw[color=black,thick] (-0.3,0.3) -- (-0.1,0.1);
    \draw[color=black,thick] (0.1,-0.1) -- (0.3,-0.3);
  }
}

\newcommand{\nx}{
  \KP{
    \draw[color=black,thick] (-0.3,0.3) -- (0.3,-0.3);
    \draw[color=black,thick] (-0.3,-0.3) -- (-0.1,-0.1);
    \draw[color=black,thick] (0.1,0.1) -- (0.3,0.3);
  }
}

\newcommand{\opx}[2]{
  \KP{
    \draw[color=black,thick,<-] (-0.3,-0.3) -- (0.3,0.3) node[left] at (-0.3,0.3){$#1$};
    \draw[color=black,thick] (-0.3,0.3) -- (-0.1,0.1);
    \draw[color=black,thick,->] (0.1,-0.1) -- (0.3,-0.3) node[right] at (0.3,0.3){$#2$};
  }
}

\newcommand{\ccm}{
  \KP{
    \draw[color=black,thick] (-0.3,0.3) .. controls (0,-0.01) .. (0.3,0.3);
    \draw[color=black,thick] (-0.3,-0.3) .. controls (0,0.01) .. (0.3,-0.3);
  }
}

\newcommand{\idm}{
  \KP{
    \draw[color=black,thick] (-0.3,-0.3) .. controls (0.01,0) .. (-0.3,0.3);
    \draw[color=black,thick] (0.3,-0.3) .. controls (-0.01,0) .. (0.3,0.3);
  }
}

\newcommand{\jack}[1]{
  \KP{
    \draw[color=black,thick] (-0.3,-0.3) -- (0.0,-0.2);
    \draw[color=black,thick] (0.3,-0.3) -- (0.0,-0.2);
    \draw[color=black,thick] (-0.3,0.3) -- (0.0,0.2);
    \draw[color=black,thick] (0.3,0.3) -- (0.0,0.2);
    \draw[color=black,thick] (0.0,-0.2) -- node[right]{$#1$} (0.0,0.2);
  }
}

\newcommand{\jj}[1]{
\mathcal{J}_{#1}
}

\newcommand{\jjd}[1]{
\mathcal{J}'_{#1}
}

\newcommand{\ojack}[1]{
  \KP{
    \draw[color=black,thick] (-0.3,-0.3) -- (0.0,-0.2);
    \draw[color=black,thick] (0.3,-0.3) -- (0.0,-0.2);
    \draw[color=black,thick] (-0.3,0.3) -- (0.0,0.2);
    \draw[color=black,thick] (0.3,0.3) -- (0.0,0.2);
    \draw[<-,color=black,thick] (0.0,-0.05) -- node[right,yshift=-1mm]{$#1$} (0.0,0.2);
    \draw[color=black,thick] (0.0,0.0)--(0.0,-0.2);
  }
}

\newcommand{\bone}[1]{
  \KP{
    \draw[color=black,thick] (-0.3,-0.3) -- (-0.2,-0.0);
    \draw[color=black,thick] (0.3,-0.3) -- (0.2,0.0);
    \draw[color=black,thick] (-0.3,0.3) -- (-0.2,0.0);
    \draw[color=black,thick] (0.3,0.3) -- (0.2,0.0);
    \draw[color=black,thick] (-0.2,0.0) -- node[above]{$#1$} (0.2,0.0);
  }
}

\newcommand{\Loopy}{
  \KP{\filldraw[color=black, fill=none, thick] circle (0.5);}
}

\newcommand{\Px}{
  \KP{
    \draw[color=black,thick] (-0.4,-0.4) -- (0.4,0.4);
    \draw[color=black,thick] (-0.4,0.4) -- (-0.1,0.1);
    \draw[color=black,thick] (0.1,-0.1) -- (0.4,-0.4);
  }
}

\newcommand{\OPx}[2]{
  \KP{
    \draw[color=black,thick,<-] (-0.4,-0.4) -- (0.4,0.4) node[left] at (-0.4,0.4){$#1$};
    \draw[color=black,thick] (-0.4,0.4) -- (-0.1,0.1);
    \draw[color=black,thick,->] (0.1,-0.1) -- (0.4,-0.4) node[right] at (0.4,0.4){$#2$};
  }
}

\newcommand{\ONx}{
  \KP{
    \draw[color=black,thick,->] (-0.4,0.4) -- (0.4,-0.4);
    \draw[color=black,thick, <-] (-0.4,-0.4) -- (-0.1,-0.1);
    \draw[color=black,thick] (0.1,0.1) -- (0.4,0.4);
  }
}

\newcommand{\Nx}{
  \KP{
    \draw[color=black,thick] (-0.4,0.4) -- (0.4,-0.4);
    \draw[color=black,thick] (-0.4,-0.4) -- (-0.1,-0.1);
    \draw[color=black,thick] (0.1,0.1) -- (0.4,0.4);
  }
}

\newcommand{\Ccm}{
  \KP{
    \draw[color=black,thick] (-0.4,0.4) .. controls (0,-0.01) .. (0.4,0.4);
    \draw[color=black,thick] (-0.4,-0.4) .. controls (0,0.01) .. (0.4,-0.4);
  }
}

\newcommand{\Idm}{
  \KP{
    \draw[color=black,thick] (-0.4,-0.4) .. controls (0.01,0) .. (-0.4,0.4);
    \draw[color=black,thick] (0.4,-0.4) .. controls (-0.01,0) .. (0.4,0.4);
  }
}

\newcommand{\OIdm}[2]{
  \KP{
    \draw[color=black,thick,<-] (-0.4,-0.4) .. controls (0.01,0) .. (-0.4,0.4) node[left]{$#1$};
    \draw[color=black,thick,<-] (0.4,-0.4) .. controls (-0.01,0) .. (0.4,0.4) node[right]{$#2$};
  }
}

\newcommand{\Jack}[1]{
  \KP{
    \draw[color=black,thick] (-0.4,-0.5) -- (0.0,-0.3);
    \draw[color=black,thick] (0.4,-0.5) -- (0.0,-0.3);
    \draw[color=black,thick] (-0.4,0.5) -- (0.0,0.3);
    \draw[color=black,thick] (0.4,0.5) -- (0.0,0.3);
    \draw[color=black,thick] (0.0,-0.3) -- node[right]{$#1$} (0.0,0.3);
  }
}

\newcommand{\OLJack}[5]{
  \KP{
    \draw[color=black,thick] (-0.4,-0.5) -- (0.0,-0.3) node[left] at (-0.4,-0.5){$#1$};
    \draw[thick,->](0,-0.3)--(-0.2,-0.4);
    \draw[color=black,thick] (0.4,-0.5) -- (0.0,-0.3)node[right] at (0.4,-0.5){$#2$};
    \draw[thick,->](0,-0.3)--(0.2,-0.4);
    \draw[color=black,thick] (-0.4,0.5) -- (0.0,0.3)node[left] at (-0.4,0.5) {$#4$};
    \draw[thick,->](-0.4,0.5)--(-0.2,0.4);
    \draw[color=black,thick] (0.4,0.5) -- (0.0,0.3)node[right] at (0.4,0.5) {$#5$};
    \draw[thick,->](0.4,0.5)--(0.2,0.4);
    \draw[color=black,thick] (0.0,-0.3) -- node[right]{$#3$} (0.0,0.3);
    \draw[thick,->](0,0.3)--(0,-0.05);
  }
}

\newcommand{\OJack}[1]{
  \KP{
   \draw[color=black,thick] (-0.4,-0.5) -- (0.0,-0.3);
    \draw[color=black,thick] (0.4,-0.5) -- (0.0,-0.3);
    \draw[color=black,thick] (-0.4,0.5) -- (0.0,0.3);
    \draw[color=black,thick] (0.4,0.5) -- (0.0,0.3);
    \draw[<-,color=black,thick] (0.0,-0.05) -- (0.0,0.3)node[right] at (0.0,-0.05){$#1$};
    \draw[color=black,thick] (0.0,-0.3) -- (0.0,0.0);
  }
}

\newcommand{\Bone}[1]{
  \KP{
    \draw[color=black,thick] (-0.5,-0.4) -- (-0.3,-0.0);
    \draw[color=black,thick] (0.5,-0.4) -- (0.3,0.0);
    \draw[color=black,thick] (-0.5,0.4) -- (-0.3,0.0);
    \draw[color=black,thick] (0.5,0.4) -- (0.3,0.0);
    \draw[color=black,thick] (-0.3,0.0) -- node[above]{$#1$} (0.3,0.0);
  }
}

\newcommand{\OBone}[1]{
  \KP{
    \draw[color=black,thick] (-0.5,-0.4) -- (-0.3,-0.0);
    \draw[color=black,thick] (0.5,-0.4) -- (0.3,0.0);
    \draw[color=black,thick] (-0.5,0.4) -- (-0.3,0.0);
    \draw[color=black,thick] (0.5,0.4) -- (0.3,0.0);
    \draw[->,color=black,thick] (-0.3,0.0) --(0.05,0.0) node[above]{$x$};
    \draw[color=black,thick] (0.05,0.0) -- (0.3,0.0);
  }
}

\newcommand{\CctwistA}{\hspace{-10mm}\raisebox{-7mm}{\begin{tikzpicture}  
\draw[color=black,thick] (-0.3,0.3) .. controls (0,-0.01) .. (0.3,0.3);
\scalebox{0.8}{
\begin{knot}[clip width=6, xshift=-0.35cm, yshift=-0.75cm,consider self intersections,draft mode=off]
\strand[thick] (0,0) .. controls +(1.75,1) and +(-1.75,1) .. (0.7,0);
\end{knot}
}
\end{tikzpicture}}\hspace{-10mm}}

\newcommand{\CctwistB}{\hspace{-10mm}\raisebox{-10mm}{\begin{tikzpicture} 
\begin{scope}[xscale=-1,yscale=1]
\rotatebox{180}{\scalebox{0.8}{
\begin{knot}[clip width=6, xshift=-0.35cm, yshift=-0.75cm,consider self intersections,draft mode=off]
\strand[thick] (0,0) .. controls +(1.75,1) and +(-1.75,1) .. (0.7,0);
\end{knot}
}}
\end{scope}
\draw[color=black,thick] (-0.3,-0.3) .. controls (0,0.01) .. (0.3,-0.3);
\end{tikzpicture}}\hspace{-10mm}}

\newcommand{\rtwoA}{ 
  \raisebox{-4mm}{
  \begin{tikzpicture}
  \coordinate (O) at (0.1,0,0); \coordinate (A) at (0.1,1,0); 
  \coordinate (B) at (0,0.25,0); \coordinate (C) at (0,0.75,0); 
  \coordinate (D) at (-0.3,0,0); \coordinate (E) at (-0.15,0.1,0);
  \coordinate (F) at (-0.3,1,0); \coordinate (G) at (-0.15,0.9,0);
  
  \draw[thick](O) to [bend left=70] (A); 
  \draw[thick](B) to [bend right=20] (C); 
  \draw[thick](D) to [bend right=10] (E); 
  \draw[thick](F) to [bend left=10] (G); 
  \end{tikzpicture}
  }
}

\newcommand{\rtwoB}{ 
  \raisebox{6mm}{\rotatebox{180}{ 
  \begin{tikzpicture}
  \coordinate (O) at (0.1,-0,0); \coordinate (A) at (0.1,1,0); 
  \coordinate (B) at (0,0.25,0); \coordinate (C) at (0,0.75,0); 
  \coordinate (D) at (-0.3,0,0); \coordinate (E) at (-0.15,0.1,0);
  \coordinate (F) at (-0.3,1,0); \coordinate (G) at (-0.15,0.9,0);
  
  \draw[thick](O) to [bend left=70] (A); 
  \draw[thick](B) to [bend right=20] (C); 
  \draw[thick](D) to [bend right=10] (E);
  \draw[thick](F) to [bend left=10] (G); 
  \end{tikzpicture}
  }}
}

\newcommand{\mtrivalent}{
\KP{
\begin{scope}[scale=.5, rotate=180]
\draw[thick](2,-1)--(3,0);\draw[thick,->](2,-1)--(5/2,-.5)node[right] at (2,-1){$X_j$};
\draw[thick](4,-1)--(3,0);\draw[thick,->](4,-1)--(7/2,-.5)node[left] at (4,-1) {$X_i$};
\draw[thick,->](3,0)--(3,.5)node[right] at (3,.5){$X_k$} node[left] at (3,0.1){$\mu$};
\draw[thick](3,.5)--(3,1);
\end{scope}
}
}

\newcommand{\trivalent}{
\KP{
\begin{scope}[scale=.5, rotate=180]
\draw[thick](2,-1)--(3,0);\draw[thick,->](2,-1)--(5/2,-.5)node[right] at (2,-1){$X_j$};
\draw[thick](4,-1)--(3,0);\draw[thick,->](4,-1)--(7/2,-.5)node[left] at (4,-1) {$X_i$};
\draw[thick,->](3,0)--(3,.5)node[right] at (3,.5){$X_k$} node[left] at (3,0) {$\mu$};
\draw[thick](3,.5)--(3,1);
\end{scope}
}
}

\newcommand{\adtrivalent}{
\KP{
\begin{scope}[scale=.5]
\draw[thick](2,-1)--(3,0);\draw[thick,<-](5/2,-.5)--(3,0) node[left] at (2,-1){$X_i$};
\draw[thick](4,-1)--(3,0);\draw[thick,<-](7/2,-.5)--(3,0) node[right] at (4,-1) {$X_j$};
\draw[thick,<-](3,.5)--(3,1)node[right] at (3,.5){$X_k$}node[left] at (3,0) {$\mu$};
\draw[thick](3,.5)--(3,0);
\end{scope}
}
}

\newcommand{\braideigen}{ 
\KP{
\begin{scope}[scale=.8,shift={(0,.75)}]
\draw[thick] (0.4,-0.4) to [out=315,in=45] (0,-1);
\draw[thick](-0.4,-0.4) to [out=225,in=135] (0,-1);
\draw[thick,->] (0,-1) -- (0,-1.25) node[left] at (0,-1.25){$k$};\draw[thick](0,-1.25)--(0,-1.5);
\draw[color=black,thick,<-] (-0.4,-0.4) -- (0.4,0.4) node[left] at (-0.4,0.4){$i$};
\draw[color=black,thick] (-0.4,0.4) -- (-0.1,0.1);
\draw[color=black,thick,->] (0.1,-0.1) -- (0.4,-0.4) node[right] at (0.4,0.4){$j$};
\end{scope}
}
}

\newcommand{\unbraideigen}{ 
\KP{
\begin{scope}[scale=.8,shift={(0,.75)}]
\draw[thick,->] (0,-.5) -- (0,-1.25) node[left] at (0,-1.25){$k$};\draw[thick](0,-1.25)--(0,-1.5);
\draw[color=black,thick] (0,-.5)--(-.4,.4) node[left] at (-.2,-.05) {$i$};
\draw[color=black,thick,->](-.4,.4)-- (-.2,-.05);
\draw[color=black,thick] (0,-.5)--(.4,.4) node[right] at (.2,-.05){$j$};
\draw[color=black,thick,->](.4,.4)-- (.2,-.05);
\end{scope}
}
}

\newcommand{\Ststrand}{ 
\KP{
\begin{scope}[scale=.6]
\draw[thick,->] (0,2)-- (0,0);
\draw[thick](0,0)--(0,-2) node[right] at (0,-1){$X_k$};
\end{scope}
}
}

\newcommand{\ststrand}[1]{ 
\KP{
\begin{scope}[scale=.4]
\draw[thick,->] (0,2)-- (0,0);
\draw[thick](0,0)--(0,-2) node[right] at (0,-1){$#1$};
\end{scope}
}
}

\newcommand{\fststrand}{ 
\KP{
\begin{scope}[scale=.4,rotate=180]
\draw[thick,->] (0,2)-- (0,0);
\draw[thick](0,0)--(0,-2) node[right] at (0,1){$X_k$};
\end{scope}
}
}

\tikzset{
    arrowat/.style={
        postaction={decorate,decoration={
                markings,
                mark=at position #1 with {\arrow[xshift=2pt]{<}}}}
    }
}

\newcommand{\hopf}[2]{
\begin{tikzpicture}[baseline={([yshift=-.5ex]current bounding box.center)}]
\begin{knot}[flip crossing=2, clip width=10]
\strand[thick, arrowat=1] (0.4,0) circle[radius=0.8cm] node[midway,right] {$\ \ \ #1$};
\strand[thick, arrowat=1] (-0.4,0) circle[radius=0.8cm] node[midway,right] {$\quad \quad \ \ \ #2$};
\end{knot}
\end{tikzpicture}
}

\newcommand{\trivalentxyz}{
\KP{
\begin{scope}[scale=.5, rotate=180]
\draw[thick](2,-1)--(3,0);\draw[thick,->](2,-1)--(5/2,-.5)node[right] at (2,-1){$y$};
\draw[thick](4,-1)--(3,0);\draw[thick,->](4,-1)--(7/2,-.5)node[left] at (4,-1) {$x$};
\draw[thick,->](3,0)--(3,.5)node[right] at (3,.5){$z$};
\draw[thick](3,.5)--(3,1);
\end{scope}
}
}

\newcommand{\trivalentdoozy}{
\KP{
\begin{scope}[scale=.5, rotate=180]
\draw[thick](2,-1)--(3,0);\draw[thick,->](2,-1)--(5/2,-.5)node[right] at (2,-1){$x^*$};
\draw[thick](4,-1)--(3,0);\draw[thick,->](4,-1)--(7/2,-.5)node[left] at (4,-1) {$z$};
\draw[thick,->](3,0)--(3,.5)node[right] at (3,.5){$y$};
\draw[thick](3,.5)--(3,1);
\end{scope}
}
}

\newcommand{\Ststrandz}{ 
\KP{
\begin{scope}[scale=.4]
\draw[thick,->] (0,2)-- (0,0);
\draw[thick](0,0)--(0,-2) node[right] at (0,-1){$z$};
\end{scope}
}
}

\title{Skein-Theoretic Methods for Unitary Fusion Categories}
\author[Anup Poudel and Sachin J. Valera]{Anup Poudel$^{\mathsection}$ and Sachin J. Valera$^{\dagger}$}
\thanks{$^{\mathsection}$\textit{Department of Mathematics, The University of Iowa, Iowa City, USA}}
\thanks{$^{\dagger}$\textit{Selmer Center, Department of Informatics, University of Bergen, Norway}}

\begin{document}

\begin{abstract}
Unitary fusion categories (UFCs) have gained increased attention due to emerging connections with quantum physics. We consider a fusion rule of the form $q\otimes q \cong \bm{1}\oplus\bigoplus^k_{i=1}x_{i}$ in a UFC $\C$, and extract information using the graphical calculus. For instance, we classify all associated skein relations when $k=1,2$ and $\C$ is ribbon. In particular, we also consider the instances where $q$ is antisymmetrically self-dual. Our main results follow from considering the action of a rotation operator on a ``canonical basis''. Assuming self-duality of the summands $x_{i}$, some general observations are made e.g. the real-symmetricity of the $F$-matrix $F^{qqq}_q$. We then find explicit formulae for $F^{qqq}_q$ when $k=2$ and $\C$ is ribbon, and see that the spectrum of the rotation operator distinguishes between the Kauffman and Dubrovnik polynomials.
\end{abstract}


\maketitle

\vspace{-4mm}
\section{Introduction}
\label{intro}
Fusion categories have played an important role in understanding structures arising from quantum physics, and lie at the heart of quantum algebra and quantum topology. Some fusion categories can be extended to ribbon fusion categories (RFCs): these gadgets are rich in structure, and carry a lot of information. Since ribbon categories are endowed with the topological properties of ribbon graphs, they naturally lend themselves to investigation from a skein-theoretic perspective. For instance, it is known that one can fashion link (in fact, $3$-manifold) invariants from RFCs: seminal work in this direction was carried out by Reshetikhin and Turaev \cite{RT}, followed by Kuperberg who used a skein-theoretic method to obtain new link invariants associated to quantum groups coming from Lie algebras of type $A_2, B_2, C_2$ and $G_2$ \cite{Kup1,Kup2}. In a similar vein, an important class of RFCs known as Temperley-Lieb-Jones categories can be understood using Kauffman and Lins' planar algebra of Jones-Wenzl idempotents at roots of unity \cite{kauflin, quik3}.\\

Understanding unitary fusion categories (i.e.\ fusion categories with a positive dagger structure) is crucial to developing an algebraic framework for describing topological phases of matter (TPMs). Indeed, unitary modular tensor categories (MTCs) have proved to be useful in the program for classifying (bosonic) TPMs and $(2+1)$-dimensional topological quantum field theories (TQFTs) \cite{RSW, rk5clas,dgreen}. The connection between link invariants and TQFTs was first observed by Witten when he gave an interpretation of the Jones polynomial in the context of Chern-Simons QFTs \cite{bigboiwitten}. \\

Although the classification of fusion categories is beyond our current capabilities, weaker variants of this problem can be studied by structural embellishment (e.g. imposing pivotality, braiding, (pre)modularity); but even with these modifications, the problem remains out of reach. It has been shown that there are finitely many braided fusion categories of any given rank \cite{gxrkfin}, whence there are finitely many commutative fusion algebras (of a given rank) that admit categorification. The categorifications admitted by a (commutative) fusion algebra can be explicitly calculated by solving the pentagon (and hexagon) equations: doing so recovers all of the information contained in the categories. However, solving these equations quickly becomes intractable as the rank grows. This motivates the idea of determining additional general relations between unknowns, in an attempt to reduce the size of the parameter space. In this spirit, much of our exposition revolves around starting with a fusion rule of the form
\begin{equation}q\otimes q \cong \bm{1} \oplus \bigoplus_{i=1}^{k}x_{i}\label{thesource}\end{equation}
and applying skein-theoretic methods to deduce some properties of the underlying category $\mathcal{C}$ and the associated quantum invariants. Our work is inspired by \cite[Theorems 3.1 \& 3.2]{MSP1}: using a rotation operator on $\End(q^{\otimes2})$ for $\C$ ribbon and $q$ symmetrically self-dual\footnote{For a reminder of the definition of (anti)symmetrically self-dual objects, we direct the reader to Appendix \ref{pivcoffappx}.}, the authors discuss the link invariants coming from $q$ invertible and (\ref{thesource}) for $k=1,2$. In each case, they also give some relations between the eigenvalues of the $R$-matrix $R^{qq}$. \\

\noindent We systematically recover and extend the results of \mbox{\cite[Theorems 3.1 \& 3.2]{MSP1}} in \mbox{Section \ref{skeinrev}.} Our main contributions are contained in Section \ref{mainresults}, where we uncover a relationship between the rotation operator and the $F$-matrix $F^{qqq}_{q}$ (under certain assumptions), thereby allowing us to deduce some  properties of said matrix. We note that understanding $F$-matrices is particularly important for many physical applications (e.g.in the study of TPMs, topological quantum computation, quantum tensor networks).

\subsection{Outline of the paper}\hspace{2mm} 
\vspace{2mm}

In Section \ref{prelims}, we detail the relevant mathematical background. In  Section \ref{normpatsec}, we introduce a canonical orthonormal basis of ``jumping jacks'' on $\End(q^{\otimes2})$ (for \ref{thesource}) that features throughout our main exposition. In Section \ref{relcase}, we define some conventions that are followed in the main discourse. The \mbox{rotation} operator (one of the tools most central to this paper) is introduced in \mbox{Section \ref{rotopsec}}, and a supplementary discussion is provided in Appendix \ref{rotappx}.\\

In Section \ref{skeinrev}, we consider the action of the rotation operator on crossings in $\End(q^{\otimes2})$ so as to ascertain the link invariants associated to the fusion rule (\ref{thesource}) for $k=1,2$. The Jones, Kauffman and Dubrovnik polynomials are recovered, and we find three additional skein relations coming from the antisymmetrically self-dual cases. In Appendix \ref{reptheory}, the narrative of the Section \ref{skeinrev} is reframed in terms of braid group representations. In particular, observing that braid representations associated to fusion rules of the form $q^{\otimes2}\cong\bm{1}\oplus y$ factor through the Iwahori-Hecke and Temperley-Lieb algebras, we derive a skein relation for the framed HOMFLY-PT polynomial (from which we recover the quantum invariants associated to (\ref{thesource}) for $k=1$). In \mbox{Appendix \ref{addendum}}, we give a few insights into invariants coming from antisymmetrically self-dual objects. \\

In Section \ref{mainresults}, we apply the rotation operator to the canonical basis of jumping jacks on $\End(q^{\otimes2})$ for a unitary spherical fusion category $\C$. In doing so, we require that the summands in (\ref{thesource}) are self-dual (Remark \ref{dualityremk}). Theorem \ref{bojackthm} determines the components of a ``bone'' morphism (a rotated jumping jack) in the canonical basis: we use this to prove some ``bubble-popping'' identities (Corollary \ref{popping}) and to make some general observations, a highlight of which is the real-symmetricity of $F^{qqq}_{q}$ (Corollary \ref{realsymm}). As a simple application, we deduce the form of $F^{qqq}_{q}$ when $k=1$ in (\ref{thesource}). \\
We proceed to apply the results of Section \ref{skeinrev} in order to derive explicit formulae for $F^{qqq}_{q}$ when $\C$ is also ribbon and $k=2$ (Theorem \ref{cubicfsymbols}), and deduce that $q$ cannot be antisymmetrically self-dual in this instance (Corollary \ref{noasd}). It is also observed that the spectrum of the rotation operator distinguishes between the Kauffman and Dubrovnik invariants. In Section \ref{newbies}, we investigate some properties of bases for $\End(q^{\otimes2})$ whose elements are permuted (up to a sign) under the action of the rotation operator. 
We apply our results to construct such bases when $k=2$; the diagonalisation of the rotation operator follows as an immediate consequence.\\
\vspace{-1mm}

\noindent In Section \ref{outlook}, we review the contents of our work with an eye to future extensions. 

\subsection{Acknowledgements}Both authors would like to thank Corey Jones and Dave Penneys for organising the 2019 OSU Quantum Symmetries Summer Reseach Program (with grant support from David Penneys' NSF grant DMS 1654159) where they met. Sachin Valera also wishes to thank Daniel R. Copeland for helpful discussions, and for acquainting him with the rotation operator and the results of \cite{MSP1} (at aforementioned research program) without which this work would not have been possible.

\section{Preliminaries}
\label{prelims}
We provide an overview of various concepts that are used throughout this work. For further details on some of these topics, we refer the reader to \cite{EGNO, Selinger, Tur, Kitaev}.
\subsection{Tensor categories} 
Recall that a \textit{tensor category} $\C$ is a $\Bbbk$-linear, rigid monoidal category with $\End(\mathbf{1})\cong \Bbbk$ (where $\mathbf{1}$ denotes the unit object). We henceforth let $\Bbbk=\CC$. By a \textit{simple} object $X\in\C$, we mean an object $X$ such that every nonzero $f\in \End(X)$ is an isomorphism. For any object $X\in \C$, its left and right dual objects are respectively denoted by $X^*$ and $^*X$. Every object $X\in\C$ comes with the $(\co)\ev_X$ and $(\co)\ev'_X$ morphisms, which are the left and right \textit{(co)evaluations} respectively.\footnote{The compatibility of these morphisms with the monoidal structure is ensured by the \textit{rigidity axioms}.}
\begin{align*}
    \ev_X:X^*\otimes X \to \mathbf{1} \hspace{1cm} \coev_X: \mathbf{1}\to X \otimes X^*\\
    \ev'_X:X \otimes {}^*X\to \mathbf{1} \hspace{1cm} \coev'_X: \mathbf{1}\to {}^*X \otimes X
\end{align*}
Dual objects are unique up to unique isomorphism \cite[Proposition 2.10.5]{EGNO}. 
Throughout the rest of this paper, we identify left and right duals, and denote the dual of $X$ by $X^{*}$.

\subsection{Pivotality, sphericality and quantum trace}\label{pivsphetrace}
A \textit{pivotal tensor category} $\C$ is a tensor category with a collection of isomorphisms (called a \textit{pivotal structure}) \mbox{$a_X: X \xrightarrow{\sim} X^{**}$} natural in $X$ and satisfying $a_{X\otimes Y}=a_X \otimes a_Y$ for all objects $X,Y\in\C$. For any $X \in \C$ and any morphism $f\in \End(X)$, the \textit{left} and \textit{right quantum traces} of $f$ are defined as
\begin{subequations}
\begin{equation}
    \widetilde{Tr}_l(f) := \ev_{X}\circ (\id_{X^*} \otimes f) \circ \coev'_X \in \End(\mathbf{1})
\end{equation}
\begin{equation}
    \widetilde{Tr}_r(f) :=\ev'_{X}\circ (f \otimes \id_{X^*})\circ \coev_{X} \in \End(\mathbf{1}) 
\end{equation}
\end{subequations}
If the left and right quantum traces coincide, then $\C$ is called a \textit{spherical tensor category}. Hence, for a spherical tensor category, we can unambiguously define the \textit{quantum trace} for any object $X\in \C$ and any morphism $f\in \End(X)$. That is,
\begin{equation}
     \widetilde{Tr}(f):=\widetilde{Tr}_l(f)=\widetilde{Tr}_r(f)
\end{equation}


\noindent and the \textit{quantum dimension} $d_X$ of an object $X$ is is given by (\ref{yungpepsi}), noting that $d_X=d_{X^*}$.
\begin{align}\label{yungpepsi}
    d_X:=\widetilde{Tr}(\id_X)
\end{align}

\subsection{Fusion categories and trivalent vertices} \label{fuscattriver}
A \textit{fusion category} $\C$ is a semisimple tensor category with only finitely many simple objects up to isomorphism.
\begin{remark}[Skeleton of $\C$]
Let $\Irr(\C)$ denote a set of representatives of isomorphism classes of simple objects in $\C$. Let $X_i\in \Irr(\C)$, where $i\in I$ for some index set $I\subseteq \ZZ_{\geq 0}$ and $X_0:=\mathbf{1}$. We also let $i^*$ denote $j\in I$ such that $X_j=X_i^*$. The cardinality of $\Irr(\C)$ is called the \textit{rank} of $\C$. When we restrict to working with objects in $\Irr(\C)$, it is understood that we are working in the \textit{skeleton} of $\C$: this is the full subcategory of $\C$ on the subset of objects $\Irr(\C)$, and is equivalent \mbox{to $\C$.} A category is called \textit{skeletal} if it contains one object in each isomorphism class. See also \mbox{Remarks \ref{skelremk1} and \ref{skelremk2}.}
\end{remark}

\noindent The so-called \textit{fusion rules} for $\C$ are encoded by the \textit{fusion coefficients} $N_k^{ij}\in \ZZ_{\geq 0}$ where 
\begin{align}\label{fusionrule}
    X_i \otimes X_j = \bigoplus_{k\in I} N_k^{ij}X_k, \hspace{1cm} i,j\in I.
\end{align}
We also have (where $\delta_{ij}$ denotes the Kronecker delta)
\begin{align}
    N_j^{i0}=N_j^{0i}=\delta_{ij}
\end{align}

\noindent There is a graphical calculus associated with morphisms for any tensor category $\C$. We adopt the \textit{pessimistic} convention i.e. our diagrams are viewed as morphisms going from \textit{top-to-bottom}. 
Any edge is oriented and labelled by an object $X\in \C$; and for $\bm{1}\in\C$, the edge is either invisible or emphasised by a dotted line. Diagrams representing morphisms in the \textit{skeleton} of a fusion category $\C$ have edges labelled by objects $X_i\in \Irr(\C)$. A \textit{trivalent vertex} represents a projection from a twofold tensor product onto a summand 
(or conversely, an inclusion of a summand into such a product). E.g.\ we have the projections
\begin{equation} \label{trivalent}
      \spn_{\CC}\Bigg\{\mtrivalent\Bigg\}^{N^{ij}_k}_{\mu=1} = \Hom(X_i\otimes X_j, X_k)
\end{equation}

\noindent where the left-hand side constitutes a \textit{basis} for $\Hom(X_i\otimes X_j, X_k)$. Similarly, flipping the trivalent vertices upside-down  in (\ref{trivalent}), we obtain a basis of inclusion morphisms for $\Hom(X_{k},X_{i}\otimes X_{j})$. $\End(X_i)\cong \CC$, whence diagrammatically, we have 
\vspace{1.5mm}
\begin{align}
    \def\svgwidth{3cm}
\begingroup%
  \makeatletter%
  \providecommand\color[2][]{%
    \errmessage{(Inkscape) Color is used for the text in Inkscape, but the package 'color.sty' is not loaded}%
    \renewcommand\color[2][]{}%
  }%
  \providecommand\transparent[1]{%
    \errmessage{(Inkscape) Transparency is used (non-zero) for the text in Inkscape, but the package 'transparent.sty' is not loaded}%
    \renewcommand\transparent[1]{}%
  }%
  \providecommand\rotatebox[2]{#2}%
  \newcommand*\fsize{\dimexpr\f@size pt\relax}%
  \newcommand*\lineheight[1]{\fontsize{\fsize}{#1\fsize}\selectfont}%
  \ifx\svgwidth\undefined%
    \setlength{\unitlength}{101.77079088bp}%
    \ifx\svgscale\undefined%
      \relax%
    \else%
      \setlength{\unitlength}{\unitlength * \real{\svgscale}}%
    \fi%
  \else%
    \setlength{\unitlength}{\svgwidth}%
  \fi%
  \global\let\svgwidth\undefined%
  \global\let\svgscale\undefined%
  \makeatother%
  \begin{picture}(1,0.8974646)%
    \lineheight{1}%
    \setlength\tabcolsep{0pt}%
    \put(0.43422201,0.42552155){\color[rgb]{0,0,0}\makebox(0,0)[lt]{\lineheight{1.25}\smash{\begin{tabular}[t]{l}$= \lambda$\end{tabular}}}}%
    \put(0,0){\includegraphics[width=\unitlength,page=1]{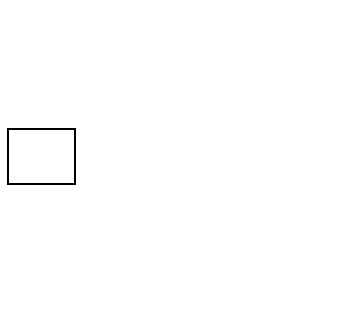}}%
    \put(0.06555778,0.43083299){\color[rgb]{0,0,0}\makebox(0,0)[lt]{\lineheight{1.25}\smash{\begin{tabular}[t]{l}$f$\end{tabular}}}}%
    \put(0,0){\includegraphics[width=\unitlength,page=2]{Schurslemma.pdf}}%
    \put(0.15335653,0.67175677){\color[rgb]{0,0,0}\makebox(0,0)[lt]{\lineheight{1.25}\smash{\begin{tabular}[t]{l}$X_i$\end{tabular}}}}%
    \put(0.1537951,0.20393795){\color[rgb]{0,0,0}\makebox(0,0)[lt]{\lineheight{1.25}\smash{\begin{tabular}[t]{l}$X_i$\end{tabular}}}}%
    \put(0,0){\includegraphics[width=\unitlength,page=3]{Schurslemma.pdf}}%
    \put(0.81802115,0.6710769){\color[rgb]{0,0,0}\makebox(0,0)[lt]{\lineheight{1.25}\smash{\begin{tabular}[t]{l}$X_i$\end{tabular}}}}%
  \end{picture}%
\endgroup%
 
\end{align}
where $f\in\End(X_{i})$ and $\lambda\in\CC$.

\vspace{2.5mm}
\subsection{Dagger structure, inner product and unitarity}\label{daginnuni}
Let $\C$ be a fusion category. Then $\C$ is called a \textit{dagger} fusion category if it is equipped with an involutive, contravariant functor $\dagger: \C \to \C$ such that it acts as the identity on objects, and satisfies (\ref{germanalgebra1})-(\ref{germanalgebra4}) where for any morphisms $f: X \to Y$, we have $\dagger(f)=f^{\dagger}:Y \to X$ where $f^\dagger$ is called the \textit{adjoint} of $f$. For morphisms $f,g\in\C$ and scalars $\lambda_1, \lambda_2\in \CC$, the $\dagger$-functor satisfies
\begin{subequations}
\begin{align}
    (\id_X)^\dagger &= \id_X \label{germanalgebra1}\\
    (g\circ f)^\dagger &= f^\dagger \circ g^\dagger\label{compositionabc}\\
    (f \otimes g)^\dagger &= f^\dagger \otimes g^\dagger\\
    (\lambda_1\cdot f + \lambda_2\cdot g)^\dagger &= \lambda_1^*\cdot f^\dagger + \lambda_2^*\cdot g^\dagger \label{germanalgebra4}
\end{align}
\end{subequations}
where for (\ref{compositionabc}) we have $f: X \to Y$ and $g: Y \to Z$ for some objects $X,Y,Z\in \C$. Note that $\lambda^*$ denotes the complex conjugate of $\lambda \in \CC$. Considering the skeleton of $\C$, we have
\vspace{-1mm}
\begin{equation}\begin{split}
    \Hom(X_i\otimes X_j, X_k) &\overset{\dagger}{\xrightarrow{\sim}} \Hom(X_k, X_i \otimes X_j) \\
     \trivalent &\overset{\dagger}{\longmapsto} \adtrivalent
\label{dualspace}
\end{split}\end{equation}
\vspace{-1mm}
\noindent We can define a sesquilinear form 
\begin{align}\label{sesquilinear}
    \langle g,f \rangle=\tr(fg^\dagger)
\end{align}
where $f,g\in \Hom(Y,X)$ and $fg^\dagger\in \End(X)$ for any $X,Y\in \C$. Further note that
\begin{align}
\langle f,g \rangle=\langle g,f \rangle^*
\end{align}
whence, (\ref{sesquilinear}) actually defines a Hermitian form.\footnote{Let $\eta:\Hom(Y,X)\overset{\sim}{\to}\Hom(YX^{*},\bm{1})$. We may equivalently write $\langle g,f \rangle=\eta(f)(\eta(g))^{\dagger}\in\End(\bm{1})\cong\CC$.}\\

\noindent Consider two elements $\bm{e}_{\mu}$ and $\bm{e}_{\nu}$ of the basis in (\ref{trivalent}). Then
\[ \bm{e}_{\nu}\bm{e}^{\dagger}_{\mu}= \raisebox{-11mm}{\includegraphics[width=0.11\textwidth]{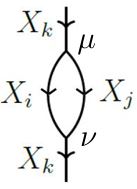}} = \lambda \ \Ststrand \ \ , \ \ \lambda\in\CC \]
Note that
\begin{equation}
\lambda=\tr(\bm{e}_{\nu}\bm{e}^{\dagger}_{\mu})=\tr(\bm{e}^{\dagger}_{\mu}\bm{e}_{\nu})=\tr\left(\raisebox{-8mm}{\includegraphics[width=0.11\textwidth]{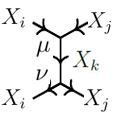}}\right)
\end{equation}
whence $\lambda$ vanishes for $\mu\neq\nu$. It follows that (\ref{trivalent}) defines an orthogonal basis with respect to the Hermitian form. We may thus write
\begin{align}\label{innerproddef}
    \raisebox{-11mm}{\includegraphics[width=0.11\textwidth]{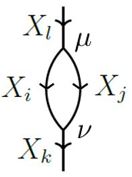}} = \lambda_{lijk} \cdot \delta_{lk}\delta_{\mu\nu} \ \Ststrand \ \ , \ \ \lambda_{lijk} \in \CC
\end{align}
where the factor of $\delta_{lk}$ follows from Schur's lemma.

\begin{proposition}$\lambda_{kijk}$ is real.
\begin{proof}
Taking the adjoint of (\ref{innerproddef}), the result is immediate.
\end{proof}
\end{proposition}

\begin{remark}[\textbf{Basis for Hom-space}]\label{poppingmonkey}
Consider the space $\Hom(X,Y)$ in the skeleton of $\C$ (where at least one of $X$ or $Y$ is not simple). This space is isomorphic to a direct sum of $\Hom$-spaces of the form $\Hom(\bigotimes^m_{k=1}X_{i_k},\bigotimes^n_{l=1}X_{j_l})$ where $X_{i_k},X_{j_l}\in \Irr(\C)$. We thus consider spaces of the form
\begin{align}\label{decomposition1}
    \Hom\left(\bigotimes^m_{k=1}X_{i_k},\bigotimes^n_{l=1}X_{j_l}\right) \cong \bigoplus_{b\in \Irr(\C)} \Hom\left(\bigotimes^m_{k=1}X_{i_k},b\right)\otimes \Hom\left(b,\bigotimes^n_{l=1}X_{j_l}\right)
\end{align}
Writing $V^X_Y:=\Hom(X,Y)$, further note that
\begin{subequations}
\begin{align}
    V^{X_{i_1}\cdots X_{i_m}}_b &\cong\bigoplus_{e_1,\ldots,e_{m-2}\in \Irr(\C)}V^{X_{i_1}X_{i_2}}_{e_1}\otimes V^{e_1X_{i_3}}_{e_2}\otimes \cdots \otimes V^{e_{m-3}X_{i_{m-1}}}_{e_{m-2}}\otimes V^{e_{m-2}X_{i_m}}_b \label{bottomgun}\\
    V^b_{X_{j_1}\cdots X_{j_n}} &\cong\bigoplus_{f_1,\ldots,f_{n-2}\in \Irr(\C)}V^b_{f_{n-2}X_{j_n}}\otimes V^{f_{n-2}}_{f_{n-3}X_{j_{n-1}}}\otimes \cdots \otimes V^{f_2}_{f_1 X_{j_3}} \otimes V_{X_{j_1}X_{j_2}}^{f_1}\label{topgun}
\end{align}
\end{subequations}

\noindent The decompositions in (\ref{bottomgun}) and (\ref{topgun}) correspond to a choice of \textit{fusion basis} on the respective Hom-spaces.

\begin{figure}[H]\centering \includegraphics[width=0.575\textwidth]{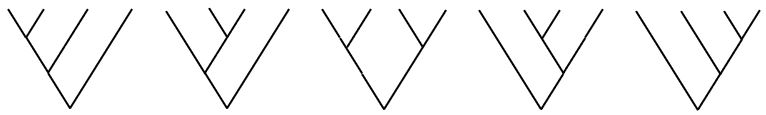} \caption{By a ``fusion basis'', we mean a parenthesisation of $\bigotimes_{k}X_{i_{k}}$. Diagrammatically, a fusion basis corresponds to a full rooted binary tree on a space of the form in (\ref{bottomgun}) or (\ref{topgun}). The above trees illustrate the fusion bases for a space of the form $V^{abcd}_{e}$. The number of distinct fusion bases for an $n$-fold product is given by the $(n-1)^{th}$ Catalan number.} \label{hoipoloi}\end{figure}

\noindent Using the basis from (\ref{trivalent}), and fixing fusion bases as in (\ref{bottomgun}) and (\ref{topgun}), we obtain the following basis\footnote{Note that inpermissible values of the indices do not contribute to the basis. Labels $\mu_{i}$ and $\nu_{j}$ respectively denote the multiplicities of trivalent vertices associated to $e_{i}$ and $f_{j}$ (they are not annotated on the basis in (\ref{orangepill}) so as not to clutter the diagram).}:
\begin{align}\label{orangepill}
    \Hom\left(\bigotimes^m_{k=1}X_{i_k},\bigotimes^n_{l=1}X_{j_l}\right) = \spn_{\CC}\left\{
   \raisebox{-35mm}{\includegraphics[width=0.35\textwidth]{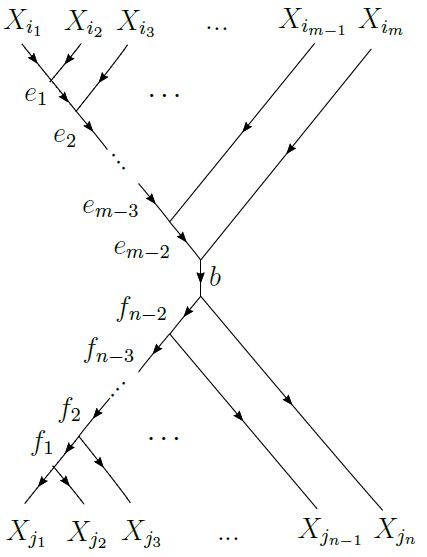}}
    \right\}_{\substack{b,e_1,\cdots,e_{m-2},f_1,\cdots,f_{n-2} \\ \phantom{b,}\mu_1,\cdots,\mu_{m-2},\nu_1,\cdots,\nu_{n-2}}} 
\end{align}
\end{remark}
\vspace{.75cm}

\vspace{-4mm}
\noindent Let $\{\textbf{e}_i\}_i$ denote elements of the basis in (\ref{orangepill}). Then we can write
\begin{align}
    g=\sum_i g_i \textbf{e}_i \quad \text{and}\quad f=\sum_i f_i \textbf{e}_i
\end{align}
where $f_i,g_i\in \CC$. Let $X':=\bigotimes^m_{i=1}X_i$. We write $\textbf{e}_i\textbf{e}_j^\dagger=:K_{ij}\tilde{\textbf{e}}_{ij}$ where the value of $K_{ij}$ is determined by the following three cases: 
\begin{enumerate}
    \item $\textbf{e}_i\textbf{e}_j^\dagger$ vanishes, in which case $K_{ij}:=0$.
    \item \label{tpop2} $\textbf{e}_i\textbf{e}_j^\dagger\in \End(X')$ and contains no loops, in which case $K_{ij}:=1$.
    \item \label{tpop3}$\textbf{e}_i\textbf{e}_j^\dagger\in \End(X')$ and contains loops, in which case $K_{ij}$ is a product of some scalars \mbox{${\lambda_{abca}\in \RR}$} coming from loops of the form (\ref{innerproddef}).
\end{enumerate}
where in cases (\ref{tpop2}) and (\ref{tpop3}), $\tilde{\textbf{e}}_{ij}$ is a basis element of the form (\ref{orangepill}) in $\End(X')$. Then 
\vspace{-0.5mm}
\begin{align*}
    \langle g,f \rangle = \tr\left(\sum_{i,j}f_ig_j^*\textbf{e}_i\textbf{e}_j^\dagger\right)=\tr\left(\sum_{i,j}f_ig_j^*K_{ij}\tilde{\textbf{e}}_{ij}\right)=\sum_{i} f_ig_i^*K_{ii}
\end{align*}
\begin{remark}[\textbf{Positive dagger structure}]
Note that
\begin{align}
    \langle f,f \rangle = \sum_i|f_i|^2K_{ii}
\end{align}
Hence, given $K_{ii}>0$, our Hermitian form defines a Hermitian inner product. This is ensured by setting  $\lambda_{kijk}>0$ in (\ref{innerproddef}). Under this constraint, our category is said to have a \textit{positive dagger structure}. 
Furthermore, this means that $\C$ is a \textit{unitary} fusion category (see also Remarks \ref{skelremk1} and \ref{skelremk2}). Also note that basis in (\ref{orangepill}) is orthogonal with respect to this inner product. If $\C$ is also spherical, viewing the quantum dimension of an object as an inner product immediately shows that it must be positive. Throughout this paper, we assume that any category we work with possesses a positive dagger structure.
\end{remark}


\subsection{Frobenius-Schur indicator}\label{frobschsec}
\noindent Let $\C$ be a unitary pivotal fusion category. Following \cite[Proposition 3.9]{Pen}, we identify zig-zag \mbox{morphisms} with the pivotal structure:
\vspace{-2mm}
\begin{align}
    \def\svgwidth{5.5cm}
\begingroup%
  \makeatletter%
  \providecommand\color[2][]{%
    \errmessage{(Inkscape) Color is used for the text in Inkscape, but the package 'color.sty' is not loaded}%
    \renewcommand\color[2][]{}%
  }%
  \providecommand\transparent[1]{%
    \errmessage{(Inkscape) Transparency is used (non-zero) for the text in Inkscape, but the package 'transparent.sty' is not loaded}%
    \renewcommand\transparent[1]{}%
  }%
  \providecommand\rotatebox[2]{#2}%
  \newcommand*\fsize{\dimexpr\f@size pt\relax}%
  \newcommand*\lineheight[1]{\fontsize{\fsize}{#1\fsize}\selectfont}%
  \ifx\svgwidth\undefined%
    \setlength{\unitlength}{238.2957062bp}%
    \ifx\svgscale\undefined%
      \relax%
    \else%
      \setlength{\unitlength}{\unitlength * \real{\svgscale}}%
    \fi%
  \else%
    \setlength{\unitlength}{\svgwidth}%
  \fi%
  \global\let\svgwidth\undefined%
  \global\let\svgscale\undefined%
  \makeatother%
  \begin{picture}(1,0.49049296)%
    \lineheight{1}%
    \setlength\tabcolsep{0pt}%
    \put(0,0){\includegraphics[width=\unitlength,page=1]{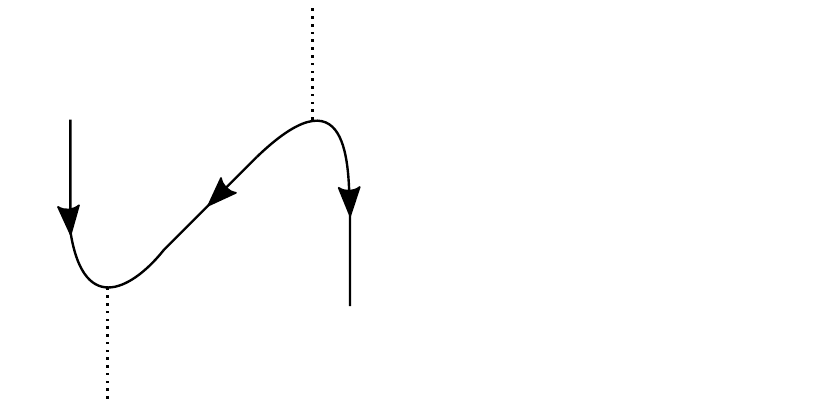}}%
    \put(0.00577545,0.19067157){\color[rgb]{0,0,0}\makebox(0,0)[lt]{\lineheight{1.25}\smash{\begin{tabular}[t]{l}$X$\end{tabular}}}}%
    \put(0.2606178,0.19022578){\color[rgb]{0,0,0}\makebox(0,0)[lt]{\lineheight{1.25}\smash{\begin{tabular}[t]{l}$X^*$\end{tabular}}}}%
    \put(0.44540784,0.18994988){\color[rgb]{0,0,0}\makebox(0,0)[lt]{\lineheight{1.25}\smash{\begin{tabular}[t]{l}$X^{**}$\end{tabular}}}}%
    \put(0.61842799,0.23457687){\color[rgb]{0,0,0}\makebox(0,0)[lt]{\lineheight{1.25}\smash{\begin{tabular}[t]{l}$=$\end{tabular}}}}%
    \put(0,0){\includegraphics[width=\unitlength,page=2]{pivotal.pdf}}%
    \put(0.76678518,0.22743551){\color[rgb]{0,0,0}\makebox(0,0)[lt]{\lineheight{1.25}\smash{\begin{tabular}[t]{l}$a_X$\end{tabular}}}}%
    \put(0,0){\includegraphics[width=\unitlength,page=3]{pivotal.pdf}}%
    \put(0.82568037,0.34838564){\color[rgb]{0,0,0}\makebox(0,0)[lt]{\lineheight{1.25}\smash{\begin{tabular}[t]{l}$X$\end{tabular}}}}%
    \put(0.82589564,0.09982756){\color[rgb]{0,0,0}\makebox(0,0)[lt]{\lineheight{1.25}\smash{\begin{tabular}[t]{l}$X^{**}$\end{tabular}}}}%
    \put(0.05609498,0.11091424){\color[rgb]{0,0,0}\makebox(0,0)[lt]{\lineheight{1.25}\smash{\begin{tabular}[t]{l}$\mu$\end{tabular}}}}%
    \put(0.30170569,0.35948829){\color[rgb]{0,0,0}\makebox(0,0)[lt]{\lineheight{1.25}\smash{\begin{tabular}[t]{l}$\mu'$\end{tabular}}}}%
  \end{picture}%
\endgroup%
\label{FSdef}
\end{align}
\noindent Thus, passing to the skeleton yields
\begin{align}
    \def\svgwidth{4.75cm}
\begingroup%
  \makeatletter%
  \providecommand\color[2][]{%
    \errmessage{(Inkscape) Color is used for the text in Inkscape, but the package 'color.sty' is not loaded}%
    \renewcommand\color[2][]{}%
  }%
  \providecommand\transparent[1]{%
    \errmessage{(Inkscape) Transparency is used (non-zero) for the text in Inkscape, but the package 'transparent.sty' is not loaded}%
    \renewcommand\transparent[1]{}%
  }%
  \providecommand\rotatebox[2]{#2}%
  \newcommand*\fsize{\dimexpr\f@size pt\relax}%
  \newcommand*\lineheight[1]{\fontsize{\fsize}{#1\fsize}\selectfont}%
  \ifx\svgwidth\undefined%
    \setlength{\unitlength}{181.97038065bp}%
    \ifx\svgscale\undefined%
      \relax%
    \else%
      \setlength{\unitlength}{\unitlength * \real{\svgscale}}%
    \fi%
  \else%
    \setlength{\unitlength}{\svgwidth}%
  \fi%
  \global\let\svgwidth\undefined%
  \global\let\svgscale\undefined%
  \makeatother%
  \begin{picture}(1,0.61474622)%
    \lineheight{1}%
    \setlength\tabcolsep{0pt}%
    \put(0,0){\includegraphics[width=\unitlength,page=1]{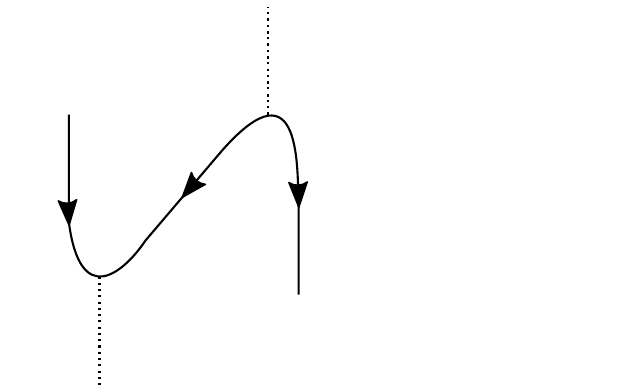}}%
    \put(0.01012206,0.23188513){\color[rgb]{0,0,0}\makebox(0,0)[lt]{\lineheight{1.25}\smash{\begin{tabular}[t]{l}$X_i$\end{tabular}}}}%
    \put(0.49679116,0.23638866){\color[rgb]{0,0,0}\makebox(0,0)[lt]{\lineheight{1.25}\smash{\begin{tabular}[t]{l}$X_i$\end{tabular}}}}%
    \put(0.64166972,0.28868318){\color[rgb]{0,0,0}\makebox(0,0)[lt]{\lineheight{1.25}\smash{\begin{tabular}[t]{l}$=$\end{tabular}}}}%
    \put(0.90133809,0.23263538){\color[rgb]{0,0,0}\makebox(0,0)[lt]{\lineheight{1.25}\smash{\begin{tabular}[t]{l}$X_i$\end{tabular}}}}%
    \put(0,0){\includegraphics[width=\unitlength,page=2]{FS.pdf}}%
    \put(0.78289969,0.29301976){\color[rgb]{0,0,0}\makebox(0,0)[lt]{\lineheight{1.25}\smash{\begin{tabular}[t]{l}$t_i$\end{tabular}}}}%
    \put(0.28023185,0.23638866){\color[rgb]{0,0,0}\makebox(0,0)[lt]{\lineheight{1.25}\smash{\begin{tabular}[t]{l}$X^*_i$\end{tabular}}}}%
    \put(0.07691071,0.14167563){\color[rgb]{0,0,0}\makebox(0,0)[lt]{\lineheight{1.25}\smash{\begin{tabular}[t]{l}$\mu$\end{tabular}}}}%
    \put(0.33810594,0.45847795){\color[rgb]{0,0,0}\makebox(0,0)[lt]{\lineheight{1.25}\smash{\begin{tabular}[t]{l}$\mu'$\end{tabular}}}}%
  \end{picture}%
\endgroup%
\label{FSdef2}
\end{align}
where $t_i \in \CC^\times$ is called a \textit{pivotal coefficient}. It can be shown \cite[Lemma E.3]{Kitaev} that (\ref{FSdef2}) implies (\ref{nobuo}), whence the indices on the trivalent vertices in (\ref{FSdef}) and (\ref{FSdef2}) \mbox{can be dropped.}
\begin{align}
    N^{ij}_0=N^{ji}_0=\delta_{ij^*} \label{nobuo}
\end{align}
It can also be shown (Proposition \ref{pivpropappx}) that\vspace{1.5mm}
\begin{equation} |t_{i}|=1 \quad , \quad t_{i^{*}}^{}\hspace{-1mm}=t_{i}^{*} \label{pivcoprops1}\end{equation}

\begin{itemize}
\vspace{2.5mm}
\item If $X_{i}$ is non self-dual, we will assume that $t_{i}=1$. This choice is always possible through a unitary ("gauge") transformation of the trivalent vertices in (\ref{FSdef2}).
\item If $X_{i}$ is self-dual, then $t_{i}$ is called the \textit{Frobenius-Schur} indicator and is written $\varkappa_{i}$; this quantity is invariant under any unitary transformations of trivalent vertices, and is therefore a fixed property of $X_{i}$. Furthermore, (\ref{pivcoprops1}) tells us that $\varkappa_{i}=\pm1$. The object $X_{i}$ is said to be \textit{(anti)symmetrically} self-dual when  $\varkappa_{i}$ is ($-1$ or) $+1$.
\end{itemize}

\vspace{3mm}
\noindent Further details are given in Appendix \ref{pivcoffappx}. Following (\ref{FSdef2}), we can make the identification
\begin{equation}
    \ststrand{X^*_k} = \hspace{.55cm}\fststrand
\end{equation}
which allows us to slide arrows around cups and caps. 

\vspace{8mm}
\subsection{Normalisation and partial trace}\label{normpatsec}
\begin{remark}
A unitary fusion category admits a unique \mbox{spherical (and corresponding pivotal) structure \cite[Prop.\ 8.23]{ENO}.}
\end{remark}
 Let $\C$ be a unitary spherical fusion category. We will henceforth use labels $i\in I$ to denote objects $X_i\in \Irr(\C)$.
\begin{remark}\textbf{(Multiplicity-free)}
Since the results of this paper pertain to fusion rules without multiplicity, we shall henceforth assume our fusion categories to be \textit{multiplicity-free} i.e. $N_k^{ij}\in \{0,1\}$ for all $i,j,k$ (unless stated otherwise). This obviates the need to index trivalent vertices (e.g. $\mu$ can be omitted in (\ref{trivalent}) and (\ref{dualspace}) in this instance). 
\end{remark}
\noindent We adopt a normalisation convention where trivalent vertices as in (\ref{trivalent}) are normalised through a scaling of factor $\sqrt[4]{\frac{d_{k}}{d_{i}d_{j}}}$. Further details are provided in Appendix \ref{normaddappx}. Under this normalisation, observe that
\begin{equation}\lambda_{kijk}=\sqrt{\dfrac{d_{i} d_{j}}{d_{k}}}\label{mangta}\end{equation}
in (\ref{innerproddef}). Following Remark \ref{poppingmonkey}, we have a \textit{canonical (orthonormal) basis} 
\begin{equation}\label{basisdecomp}
   \Hom(i\otimes j,l\otimes m) =\spn_{\CC}\Bigg\{\left(\sqrt[4]{\dfrac{d^2_{k}}{d_{i}d_{j}d_ld_m}}\right)
   \OLJack{l}{m}{k}{i}{j}\Bigg\}_{k\in I: N^{ij}_k N^{lm}_k\neq 0}
\end{equation}
where we call the graphical components of the basis diagrams \textit{jumping jacks} or \textit{jack morphisms}. Using the canonical basis for $\End(i\otimes j)$, we have the decomposition
\begin{align}\label{orientedid}
    \id_{i\otimes j}=\OIdm{i}{j} = \sum_{k\in I: N^{ij}_k\neq0}\sqrt{\frac{d_k}{d_i d_j}} \OLJack{i}{j}{k}{i}{j}  
\end{align}

\noindent For any morphism $f\in \Hom(i_1\otimes i_2 \cdots \otimes i_n, j_1\otimes j_2 \cdots \otimes j_n)$, one can define a \textit{right partial trace} if $i_n=j_n$, and a \textit{left partial trace} if $i_1=j_1$. 
\begin{figure}[H]
    \centering
\def\svgwidth{5cm}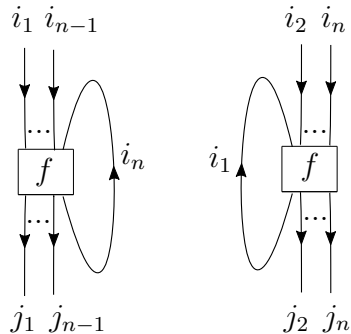\label{partialtrace}
    \caption{The right and left partial traces of $f$.}
    \label{partialtraces}
\end{figure}





\noindent Now suppose that $\mathcal{C}$ is also spherical. We define the \textit{phi-net}
\begin{equation}\Phi(i,j,k):=\widetilde{Tr}\left(\raisebox{-8mm}{\includegraphics[width=0.065\textwidth]{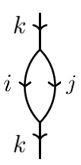}}\right)
=\raisebox{-9mm}{\includegraphics[width=0.11\textwidth]{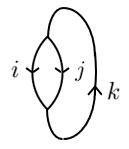}}
=\raisebox{-4.5mm}{\includegraphics[width=0.18\textwidth]{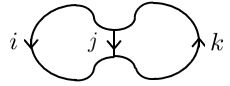}} \end{equation}
where the final diagram corresponds to the left and right partial trace of a basis jack in $\End(i^{*}\otimes k)$. Following (\ref{mangta}), we know that
\begin{equation}\Phi(i,j,k)=\sqrt{\dfrac{d_{i}d_{j}}{d_{k}}}\cdot\widetilde{Tr}(\id_{k})=\sqrt{d_{i}d_{j}d_{k}}\end{equation}
Given $a,b,c$ self-dual, we define the \textit{theta-net}
\vspace{-2mm}
\begin{equation}\Theta(a,b,c):=\raisebox{-8mm}{\includegraphics[width=0.085\textwidth]{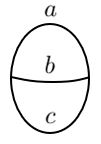}} \end{equation}
where we have left the edges unoriented (since the labels are self-dual). Note that $\Theta(a,b,c)=\Phi(a,b,c)$. Applying the left and right partial traces to (\ref{orientedid}), we get
\begin{equation}d_{i}d_{j}=\sum_{k\in I : N^{ij}_{k}\neq0}\sqrt{\dfrac{d_{k}}{d_{i}d_{j}}}\Phi(i^{*},k,j)=\sum_{k}N^{ij}_{k}d_{k}\label{gina}\end{equation}

\subsection{$F$-matrices} Recall that a monoidal category $\mathcal{C}$ has associativity isomorphisms \mbox{$\alpha_{X,Y,Z}: (X\otimes Y)\otimes Z \xrightarrow{\sim} X\otimes(Y\otimes Z)$} for any objects $X,Y,Z\in\C$. These isomorphisms satisfy compatibility conditions given by the pentagon and triangle axioms.\\

\noindent For a skeletal fusion category, after making a choice of basis for each $\Hom(i\otimes j,k)$ where $i,j,k \in I$, we obtain a block-diagonal matrix $A^{abc}$ corresponding to each associativity isomorphism $\alpha_{a,b,c}$ where $a,b,c\in I$. Each block in $A^{abc}$ is called an $F$-\textit{matrix}, and is written $F^{abc}_d$ (where $d$ indexes each block). As a map, $F^{abc}_d$ represents the isomorphism (\ref{fmatrix}) and can be interpreted as a change of (fusion) basis on $\Hom(a\otimes b \otimes c,d)$.

\begin{align}\label{fmatrix}
    F^{abc}_d:\bigoplus_e \Hom(a\otimes b,e)\otimes \Hom(e\otimes c,d) \xrightarrow{\sim} \bigoplus_f \Hom(a\otimes f,d) \otimes \Hom(b\otimes c,f)
\end{align}

\noindent where $A^{abc} = \bigoplus_d F^{abc}_d$. In the graphical calculus,
\vspace{2mm}
\begin{equation}\centering\def\svgwidth{8cm}
\begingroup%
  \makeatletter%
  \providecommand\color[2][]{%
    \errmessage{(Inkscape) Color is used for the text in Inkscape, but the package 'color.sty' is not loaded}%
    \renewcommand\color[2][]{}%
  }%
  \providecommand\transparent[1]{%
    \errmessage{(Inkscape) Transparency is used (non-zero) for the text in Inkscape, but the package 'transparent.sty' is not loaded}%
    \renewcommand\transparent[1]{}%
  }%
  \providecommand\rotatebox[2]{#2}%
  \newcommand*\fsize{\dimexpr\f@size pt\relax}%
  \newcommand*\lineheight[1]{\fontsize{\fsize}{#1\fsize}\selectfont}%
  \ifx\svgwidth\undefined%
    \setlength{\unitlength}{569.64522234bp}%
    \ifx\svgscale\undefined%
      \relax%
    \else%
      \setlength{\unitlength}{\unitlength * \real{\svgscale}}%
    \fi%
  \else%
    \setlength{\unitlength}{\svgwidth}%
  \fi%
  \global\let\svgwidth\undefined%
  \global\let\svgscale\undefined%
  \makeatother%
  \begin{picture}(1,0.26264718)%
    \lineheight{1}%
    \setlength\tabcolsep{0pt}%
    \put(0.43926282,0.11931157){\color[rgb]{0,0,0}\rotatebox{-0.82568067}{\makebox(0,0)[lt]{\lineheight{1.25}\smash{\begin{tabular}[t]{l}=$\sum_f [F^{abc}_d]_{fe}$\end{tabular}}}}}%
    \put(0.59814623,0.22387554){\color[rgb]{0,0,0}\rotatebox{-0.82568067}{\makebox(0,0)[lt]{\lineheight{1.25}\smash{\begin{tabular}[t]{l}$a$\end{tabular}}}}}%
    \put(0.74922339,0.22336594){\color[rgb]{0,0,0}\rotatebox{-0.82568067}{\makebox(0,0)[lt]{\lineheight{1.25}\smash{\begin{tabular}[t]{l}$b$\end{tabular}}}}}%
    \put(0.93408176,0.22046338){\color[rgb]{0,0,0}\rotatebox{-0.82568067}{\makebox(0,0)[lt]{\lineheight{1.25}\smash{\begin{tabular}[t]{l}$c$\end{tabular}}}}}%
    \put(0.80097533,0.02019462){\color[rgb]{0,0,0}\rotatebox{-0.82568067}{\makebox(0,0)[lt]{\lineheight{1.25}\smash{\begin{tabular}[t]{l}$d$\end{tabular}}}}}%
    \put(0.82730394,0.11162154){\color[rgb]{0,0,0}\rotatebox{-0.82568067}{\makebox(0,0)[lt]{\lineheight{1.25}\smash{\begin{tabular}[t]{l}$f$\end{tabular}}}}}%
    \put(0,0){\includegraphics[width=\unitlength,page=1]{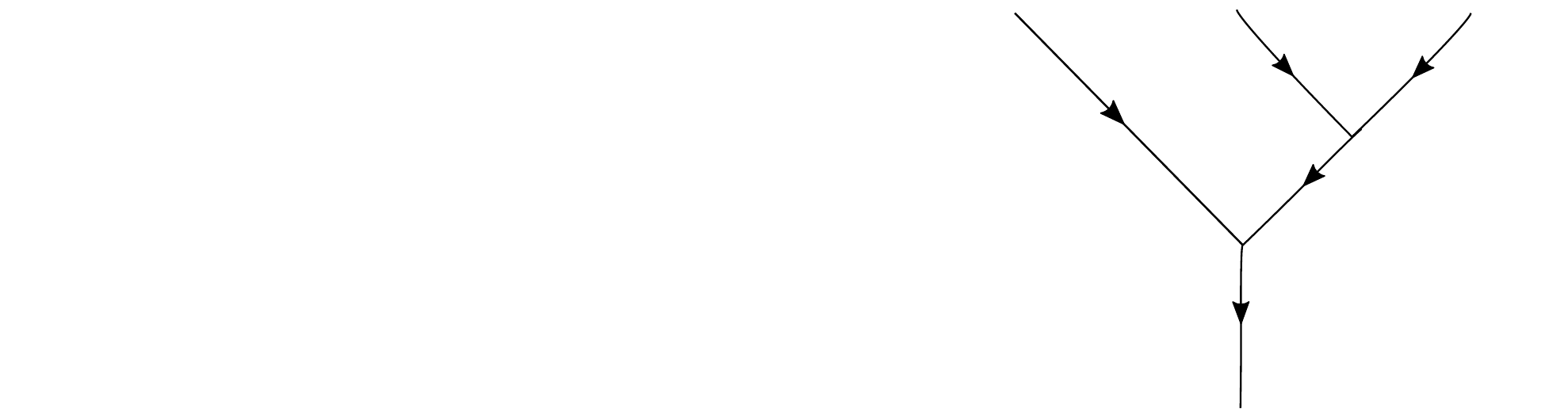}}%
    \put(0.0106328,0.22146846){\color[rgb]{0,0,0}\rotatebox{-0.82568067}{\makebox(0,0)[lt]{\lineheight{1.25}\smash{\begin{tabular}[t]{l}$a$\end{tabular}}}}}%
    \put(0,0){\includegraphics[width=\unitlength,page=2]{6j2.pdf}}%
    \put(0.21974927,0.22123098){\color[rgb]{0,0,0}\rotatebox{-0.82568067}{\makebox(0,0)[lt]{\lineheight{1.25}\smash{\begin{tabular}[t]{l}$b$\end{tabular}}}}}%
    \put(0.34947747,0.22088958){\color[rgb]{0,0,0}\rotatebox{-0.82568067}{\makebox(0,0)[lt]{\lineheight{1.25}\smash{\begin{tabular}[t]{l}$c$\end{tabular}}}}}%
    \put(0.23099825,0.02352304){\color[rgb]{0,0,0}\rotatebox{-0.82568067}{\makebox(0,0)[lt]{\lineheight{1.25}\smash{\begin{tabular}[t]{l}$d$\end{tabular}}}}}%
    \put(0.19008244,0.14490624){\color[rgb]{0,0,0}\rotatebox{-0.82568067}{\makebox(0,0)[lt]{\lineheight{1.25}\smash{\begin{tabular}[t]{l}$e$\end{tabular}}}}}%
    \put(0,0){\includegraphics[width=\unitlength,page=3]{6j2.pdf}}%
  \end{picture}%
\endgroup%
\label{6jdef}\end{equation}
\vspace{1mm}
\noindent The entries of an $F$-matrix are called $F$-\textit{symbols} (or $6j$-symbols).
In terms of the fusion coefficients, associativity is expressed as
\vspace{1.5mm}
\begin{align}
    \sum_e N^{ab}_e N^{ec}_d = \sum_f N^{af}_d N^{bc}_f
\end{align}
 
\begin{remark}[\textbf{Skeletal data I}]\label{skelremk1}
Given a fusion category $\C$, its \textit{skeletal data} is given by the set of all fusion coefficients and $F$-symbols; this data completely characterises $\C$. The $F$-symbols satisfy the pentagon equation coming from the pentagon axiom. If $\C$ has a positive dagger structure, it is easy to see that all associated $F$-matrices will be unitary (and so $\C$ is called unitary).
\end{remark}

\subsection{Braided tensor categories} Recall that for any two objects $X$ and $Y$ in a \textit{braided tensor category} $\C$, a \textit{braiding} is a natural isomorphism $c_{X,Y}: X\otimes Y \xrightarrow{\sim} Y \otimes X$ which is compatible with the associativity isomorphisms: this is ensured by the hexagon axioms, and the braidings consequently satisfy the \textit{Yang-Baxter equation}
\begin{equation}
    (c_{Y,Z}\otimes \id_X)\circ (\id_Y \otimes c_{X,Z} )\circ (c_{X,Y}\otimes \id_Z) = (\id_Z\otimes c_{X,Y})\circ (c_{X,Z}\otimes \id_Y)\circ (\id_X \otimes c_{Y,Z}) 
\end{equation}
\noindent for any $X,Y,Z\in \C$. This affords us braid isotopy in the graphical calculus.

\begin{figure}[H]
    \centering
    \begin{tabular}{c c}
        \text{(a) } \raisebox{-1cm}{\def\svgwidth{2.5cm}
\begingroup%
  \makeatletter%
  \providecommand\color[2][]{%
    \errmessage{(Inkscape) Color is used for the text in Inkscape, but the package 'color.sty' is not loaded}%
    \renewcommand\color[2][]{}%
  }%
  \providecommand\transparent[1]{%
    \errmessage{(Inkscape) Transparency is used (non-zero) for the text in Inkscape, but the package 'transparent.sty' is not loaded}%
    \renewcommand\transparent[1]{}%
  }%
  \providecommand\rotatebox[2]{#2}%
  \newcommand*\fsize{\dimexpr\f@size pt\relax}%
  \newcommand*\lineheight[1]{\fontsize{\fsize}{#1\fsize}\selectfont}%
  \ifx\svgwidth\undefined%
    \setlength{\unitlength}{160.23176779bp}%
    \ifx\svgscale\undefined%
      \relax%
    \else%
      \setlength{\unitlength}{\unitlength * \real{\svgscale}}%
    \fi%
  \else%
    \setlength{\unitlength}{\svgwidth}%
  \fi%
  \global\let\svgwidth\undefined%
  \global\let\svgscale\undefined%
  \makeatother%
  \begin{picture}(1,0.62753194)%
    \lineheight{1}%
    \setlength\tabcolsep{0pt}%
    \put(0,0){\includegraphics[width=\unitlength,page=1]{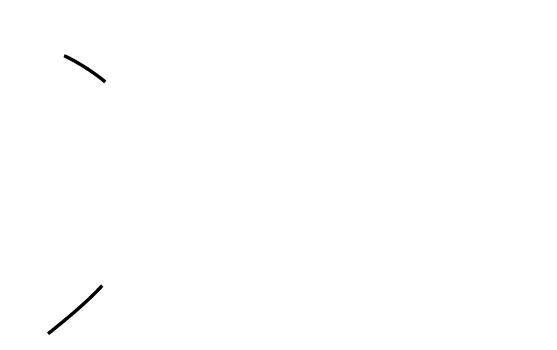}}%
    \put(0.03374778,0.56011769){\color[rgb]{0,0,0}\makebox(0,0)[lt]{\lineheight{1.25}\smash{\begin{tabular}[t]{l}$X$\end{tabular}}}}%
    \put(0.25740203,0.56373993){\color[rgb]{0,0,0}\makebox(0,0)[lt]{\lineheight{1.25}\smash{\begin{tabular}[t]{l}$Y$\end{tabular}}}}%
    \put(0.6189212,0.55196307){\color[rgb]{0,0,0}\makebox(0,0)[lt]{\lineheight{1.25}\smash{\begin{tabular}[t]{l}$X$\end{tabular}}}}%
    \put(0.85627093,0.54991217){\color[rgb]{0,0,0}\makebox(0,0)[lt]{\lineheight{1.25}\smash{\begin{tabular}[t]{l}$Y$\end{tabular}}}}%
    \put(0,0){\includegraphics[width=\unitlength,page=2]{R2.pdf}}%
    \put(0.39931362,0.28707862){\color[rgb]{0,0,0}\makebox(0,0)[lt]{\lineheight{1.25}\smash{\begin{tabular}[t]{l}$=$\end{tabular}}}}%
    \put(0,0){\includegraphics[width=\unitlength,page=3]{R2.pdf}}%
  \end{picture}%
\endgroup%
}  &  \quad \text{(b) } \raisebox{-1cm}{\def\svgwidth{4.5cm}
\begingroup%
  \makeatletter%
  \providecommand\color[2][]{%
    \errmessage{(Inkscape) Color is used for the text in Inkscape, but the package 'color.sty' is not loaded}%
    \renewcommand\color[2][]{}%
  }%
  \providecommand\transparent[1]{%
    \errmessage{(Inkscape) Transparency is used (non-zero) for the text in Inkscape, but the package 'transparent.sty' is not loaded}%
    \renewcommand\transparent[1]{}%
  }%
  \providecommand\rotatebox[2]{#2}%
  \newcommand*\fsize{\dimexpr\f@size pt\relax}%
  \newcommand*\lineheight[1]{\fontsize{\fsize}{#1\fsize}\selectfont}%
  \ifx\svgwidth\undefined%
    \setlength{\unitlength}{162.80853752bp}%
    \ifx\svgscale\undefined%
      \relax%
    \else%
      \setlength{\unitlength}{\unitlength * \real{\svgscale}}%
    \fi%
  \else%
    \setlength{\unitlength}{\svgwidth}%
  \fi%
  \global\let\svgwidth\undefined%
  \global\let\svgscale\undefined%
  \makeatother%
  \begin{picture}(1,0.40907186)%
    \lineheight{1}%
    \setlength\tabcolsep{0pt}%
    \put(0,0){\includegraphics[width=\unitlength,page=1]{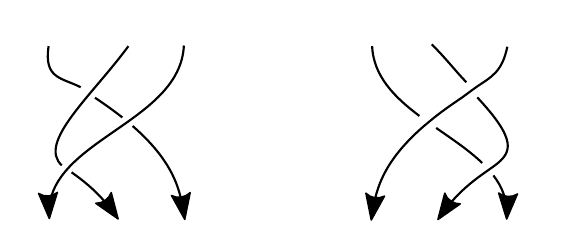}}%
    \put(0.03065849,0.36005549){\color[rgb]{0,0,0}\makebox(0,0)[lt]{\lineheight{1.25}\smash{\begin{tabular}[t]{l}$X$\end{tabular}}}}%
    \put(0.18539077,0.35986421){\color[rgb]{0,0,0}\makebox(0,0)[lt]{\lineheight{1.25}\smash{\begin{tabular}[t]{l}$Y$\end{tabular}}}}%
    \put(0.2987623,0.35737117){\color[rgb]{0,0,0}\makebox(0,0)[lt]{\lineheight{1.25}\smash{\begin{tabular}[t]{l}$Z$\end{tabular}}}}%
    \put(0.60118397,0.35256999){\color[rgb]{0,0,0}\makebox(0,0)[lt]{\lineheight{1.25}\smash{\begin{tabular}[t]{l}$X$\end{tabular}}}}%
    \put(0.73325943,0.34899727){\color[rgb]{0,0,0}\makebox(0,0)[lt]{\lineheight{1.25}\smash{\begin{tabular}[t]{l}$Y$\end{tabular}}}}%
    \put(0.86928779,0.34758178){\color[rgb]{0,0,0}\makebox(0,0)[lt]{\lineheight{1.25}\smash{\begin{tabular}[t]{l}$Z$\end{tabular}}}}%
    \put(0.44279422,0.17082656){\color[rgb]{0,0,0}\makebox(0,0)[lt]{\lineheight{1.25}\smash{\begin{tabular}[t]{l}$=$\end{tabular}}}}%
  \end{picture}%
\endgroup%
}
    \end{tabular}
    \caption{(a) $(c_{X,Y})^{-1}\circ c_{X,Y}=\id_{X\otimes Y}$, (b) Yang-Baxter equation.}
    \label{braidisotopy2}
\end{figure}
    

\noindent For a skeletal braided fusion category, after making a choice of basis for each $\Hom(i\otimes j,k)$ where $i,j,k \in I$, we obtain a block-diagonal matrix $R^{ij}$ corresponding to each braiding isomorphism $c_{i,j}$. In the multiplicity-free case, each block is a $1 \times 1$ matrix denoted by $R^{ij}_k$ (where $k$ indexes each block) whose entry is called an $R$-\textit{symbol}. By abuse of notation, we will use $R^{ij}_k$ to denote the $1 \times 1$ matrix and the $R$-symbol interchangeably. In the graphical calculus, the $R$-symbols are given by
\begin{equation}\label{braideigen}
    \braideigen := R^{ij}_k \unbraideigen\\
\end{equation}
\noindent whence in the graphical calculus, the $R$-\textit{matrix} is given by
\begin{align}\label{crossingdecomp}
    R^{ij}:=\OPx{i}{j} \overset{(\ref{basisdecomp})}{=} \sum_{{k\in I: N^{ij}_k\neq0}} R^{ij}_k \sqrt{\frac{d_k}{d_i d_j}} \OLJack{j}{i}{k}{i}{j}
\end{align}
Thus, the $R$-matrix is diagonal; specifically, we have
\begin{align}
    R^{ij}=\bigoplus_{k\in I: N^{ij}_k\neq 0} R^{ij}_k
\end{align}
In the presence of a braiding, all fusion coefficients clearly satisfy
\begin{align}
    N^{ij}_k=N^{ji}_k
\end{align}

\begin{remark}[\textbf{Skeletal data II}]\label{skelremk2}
Given a braided fusion category $\C$, its \textit{skeletal data} is given by the set of all fusion coefficients, $F$-symbols and $R$-symbols; this data completely characterises $\C$.\footnote{The skeletal data of a ribbon fusion category or modular tensor category is also given by this set.} The $F$-symbols and $R$-symbols satisfy the hexagon equations coming from the hexagon axioms. If $\C$ has a positive dagger structure, then the category is called unitary: we know that all associated $F$-matrices will be unitary; furthermore, all associated $R$-matrices must also be unitary, since every admissible braiding on a unitary fusion category must also be unitary \cite[Theorem 3.2]{Gal}.
\end{remark}


\subsection{Ribbon structure}
A spherical braided fusion category $\C$ is called a \textit{ribbon fusion} (or \textit{premodular}) \textit{category}. This is a braided fusion category with a \textit{ribbon structure}, which is given by a natural isomorphism $\theta_X: X \overset{\sim}{\to} X$ called the \textit{twist} that satisfies
\vspace{-2mm}
\begin{subequations}
\begin{align}
    \theta_{X\otimes Y} &= c_{Y,X}\circ c_{X,Y} \circ (\theta_X \otimes \theta_Y) \label{argo1}\\
    (\theta_X)^* &= \theta_{X^*} \label{argo2}
\end{align}
\end{subequations}
for all $X,Y\in \C$, and where $^*$ denotes the dual functor on the left-hand side of (\ref{argo2}). Graphically, the twist is defined as follows for a skeletal ribbon category:
\begin{equation}
\label{ribgradef}
    \text{(a)}\ \ \ststrand{i} \overset{\theta_i}{\longmapsto} \vptt= \vartheta_i \ststrand{i} \quad , \quad
    \text{(b)} \ \ \ststrand{i} \overset{\theta^{-1}_{i}}{\longmapsto} \vnt = \vartheta_{i}^{-1} \ststrand{i} 
\end{equation}    
where $\vartheta_{i}\in\mathbb{C}^{\times}$. Note that (\ref{ribgradef}b) follows from (\ref{ribgradef}a), since by braid isotopy (and pivotality),
\begin{equation}
\def\svgwidth{4cm}
\begingroup%
  \makeatletter%
  \providecommand\color[2][]{%
    \errmessage{(Inkscape) Color is used for the text in Inkscape, but the package 'color.sty' is not loaded}%
    \renewcommand\color[2][]{}%
  }%
  \providecommand\transparent[1]{%
    \errmessage{(Inkscape) Transparency is used (non-zero) for the text in Inkscape, but the package 'transparent.sty' is not loaded}%
    \renewcommand\transparent[1]{}%
  }%
  \providecommand\rotatebox[2]{#2}%
  \newcommand*\fsize{\dimexpr\f@size pt\relax}%
  \newcommand*\lineheight[1]{\fontsize{\fsize}{#1\fsize}\selectfont}%
  \ifx\svgwidth\undefined%
    \setlength{\unitlength}{175.64263411bp}%
    \ifx\svgscale\undefined%
      \relax%
    \else%
      \setlength{\unitlength}{\unitlength * \real{\svgscale}}%
    \fi%
  \else%
    \setlength{\unitlength}{\svgwidth}%
  \fi%
  \global\let\svgwidth\undefined%
  \global\let\svgscale\undefined%
  \makeatother%
  \begin{picture}(1,0.50892525)%
    \lineheight{1}%
    \setlength\tabcolsep{0pt}%
    \put(0,0){\includegraphics[width=\unitlength,page=1]{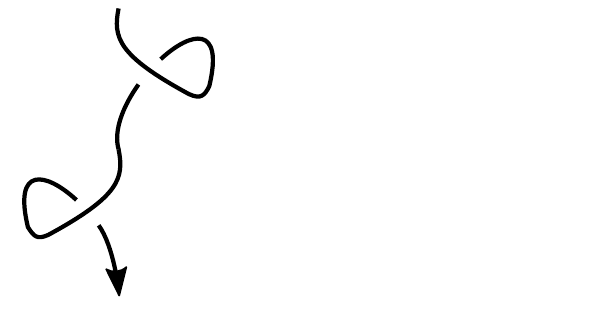}}%
    \put(0.4495475,0.22653062){\color[rgb]{0,0,0}\makebox(0,0)[lt]{\lineheight{1.25}\smash{\begin{tabular}[t]{l}$=$\end{tabular}}}}%
    \put(0,0){\includegraphics[width=\unitlength,page=2]{inversetwist.pdf}}%
    \put(0.18933197,0.11362662){\color[rgb]{0,0,0}\makebox(0,0)[lt]{\lineheight{1.25}\smash{\begin{tabular}[t]{l}$i$\end{tabular}}}}%
    \put(0.8275178,0.11956438){\color[rgb]{0,0,0}\makebox(0,0)[lt]{\lineheight{1.25}\smash{\begin{tabular}[t]{l}$i$\end{tabular}}}}%
  \end{picture}%
\endgroup%

\end{equation}
Further note that
\begin{equation}
\def\svgwidth{15cm}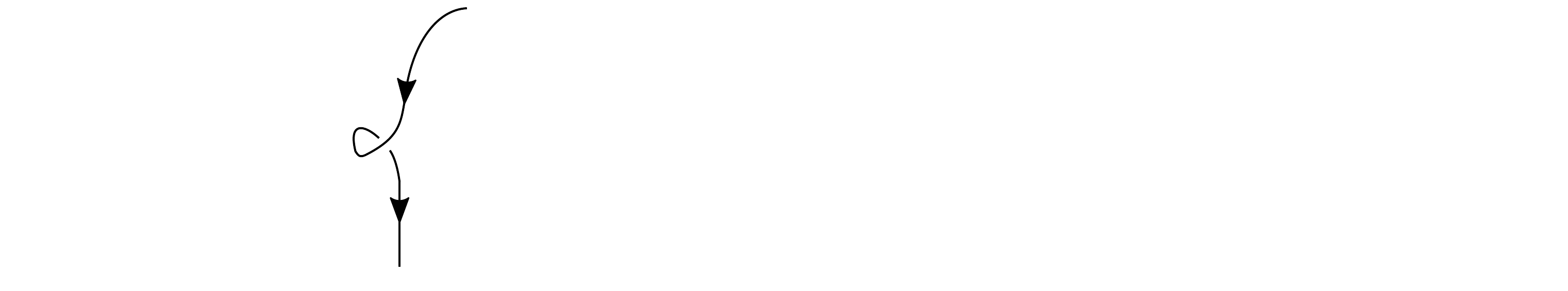
\label{surgeles}
\end{equation}
whence we obtain (\ref{baku}a). Equation (\ref{baku}b) follows similarly. 
\begin{equation}\label{baku}
\text{(a)} \ \ \scalebox{0.82}{\vptt=\vpt}  \quad , \quad \text{(b)} \ \ \scalebox{0.82}{\vnt \ = \ \vntt}
\end{equation}
From (\ref{surgeles}), the skeletal form of (\ref{argo2}) is also made apparent:
\begin{equation}\vartheta_{i}=\vartheta_{i^*}\label{yunalesca}\end{equation}
\noindent Taking the left and right partial traces for the crossing $\opx{i}{i}$, note that
\vspace{-2mm}
\begin{equation}\raisebox{-4mm}{\includegraphics[width=0.11\textwidth]{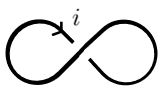}}=\raisebox{-6mm}{\includegraphics[width=0.09\textwidth]{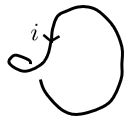}}=\raisebox{-6mm}{\includegraphics[width=0.09\textwidth]{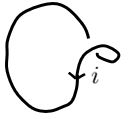}}=\vartheta_{i}d_{i}\label{agnusdei}\end{equation}
Resolving the crossing in the first diagram of (\ref{agnusdei}) using (\ref{crossingdecomp}), it easy to check that
\begin{equation}\vartheta_{i}=\dfrac{1}{d_{i}}\sum_{k}R^{ii}_{k}d_{k}\label{floobendooben}\end{equation}
It can also be shown (Appendix \ref{pivcoffappx}) that for $i$ self-dual,
\begin{equation}\vartheta_{i}=\varkappa_{i}\left(R^{ii}_{0}\right)^{-1}\label{fstwisteq}\end{equation}
In the graphical calculus for ribbon categories, edges may be promoted from lines to ribbons, and twists are $2\pi$ clockwise self-rotations of a ribbon. A labelled edge is assumed to be oriented from top-to-bottom. For instance, (\ref{yunalesca}) can be observed from
\begin{equation}\raisebox{-12mm}{\includegraphics[width=0.11\textwidth]{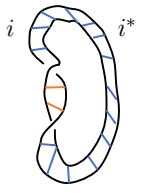}}=\raisebox{-11mm}{\includegraphics[width=0.11\textwidth]{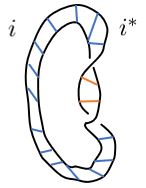}}\end{equation}
where the twist is pushed around the closed ribbon. For any $x,y,z\in\Irr(\C)$, (\ref{argo1}) may be illustrated via the action of the monodromy on a basis element of $\Hom(x\otimes y,z)$:
\vspace{-1mm}
\begin{equation}\left(R^{yx}\circ R^{xy}\right)\raisebox{-12mm}{\includegraphics[width=0.06\textwidth]{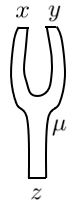}}
=\raisebox{-13mm}{\includegraphics[width=0.075\textwidth]{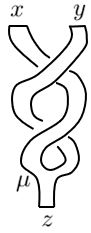}}
=\raisebox{-13mm}{\includegraphics[width=0.082\textwidth]{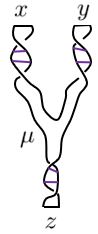}}
=\dfrac{\vartheta_{z}}{\vartheta_{x}\vartheta_{y}}\raisebox{-12mm}{\includegraphics[width=0.06\textwidth]{monoskel1}}\end{equation}
\vspace{-1mm}
where we have relaxed the multiplicity-free assumption. Thus,
\begin{equation}\sum_{\lambda}\left[R^{yx}_{z}\right]_{\mu\lambda}\left[R^{xy}_{z}\right]_{\lambda\nu}=\dfrac{\vartheta_{z}}{\vartheta_{x}\vartheta_{y}}\delta_{\mu\nu}\end{equation}
Graphically, ribbon structure affords our diagrams equivalence under braid isotopy on the $2$-sphere. We henceforth refer to braid isotopy on the $2$-sphere as \textit{framed isotopy}.
\begin{remark}\hspace{2mm} 
\begin{enumerate}[label=(\roman*)]
\item The Anderson-Moore-Vafa theorem \cite{andymoore, vafa} tells us that the twist factor $\vartheta_{i}$ is a root of unity for all $i\in I$. For a proof, we refer the reader to \cite[Theorem E.10]{Kitaev}.
\item A unitary braided fusion category admits a unique unitary ribbon structure \cite{Gal}. 
\end{enumerate}
\end{remark}

\begin{proposition}
\label{dimboi}
Let $\C$ be a unitary ribbon fusion category. Then for any $x\in\Irr(\C)$, 
\begin{equation}d_{x}\in\{1\}\cup[\sqrt{2},\infty)\end{equation}
\begin{proof}
We know that $d_{x}>0$ by unitarity. Using (\ref{gina}), we have $d_{x}d_{\bm{1}}=d_{x}$, whence $d_{\bm{1}}=1$. It will be useful to classify $x$ according to whether it satisfies the property\footnote{In a field-theoretic context, \ref{buchla} characterises the \textit{abelianity} of a quasiparticle.}
\begin{enumerate}[label=(P\arabic*),start=0]
\item $\sum_{z}N^{xy}_{z}=1$ for all $y\in\Irr(\C)$ \label{buchla}
\end{enumerate}
\textit{Claim:} $x$ satisfies \ref{buchla} if and only if $x$ is invertible (i.e.\ $x\otimes x^{*}=\bm{1}$). \\
If $x$ satisfies \ref{buchla}, then $x$ is clearly invertible. If $x$ is invertible, then $x^{*}\otimes x\otimes y=y$ for all $y\in\Irr(\C)$. Thus, $\sum_{z}N^{x^{*}z}_{y}N^{xy}_{z}=1$, whence $\sum_{z}N^{xy}_{z}=1$ for all $y$. This shows the claim. It immediately follows that if $x$ satisfies \ref{buchla}, then so does $x^{*}$. Now,
\begin{enumerate}[label=(\roman*)]
\item If $x$ satisfies \ref{buchla}, then $d_{x}d_{x^{*}}=d_{x}^{2}=d_{\bm{1}}=1$, whence $d_{x}=1$.
\item If $x$ does not satisfy \ref{buchla}, then $d_{x}d_{x^{*}}=d_{x}^{2}=d_{\bm{1}}+\sum_{y\neq\bm{1}}N^{xx^{*}}_{y}d_{y}>1$, whence $d_{x}>1$. The lower bound is attained when $x\otimes x^{*}=\bm{1}\oplus y$ for some $y$ satisfying \ref{buchla}. Thus, $d_{x}\geq\sqrt{2}$.
\end{enumerate}
\end{proof}
\end{proposition}

\subsection{Modularity}\label{mtcsec}

Let $\mathcal{C}$ be a braided fusion category. An object $X$ in $\mathcal{C}$ such that
\begin{equation}c_{Y,X}\circ c_{X,Y}=\id_{X\otimes Y}\end{equation}
for all objects $Y$ in $\mathcal{C}$ is called \textit{transparent}. If all transparent objects in $\mathcal{C}$ are isomorphic to $\bm{1}$, then the braiding is called \textit{non-degenerate}.

\noindent Further assume that $\mathcal{C}$ is ribbon. We define\footnote{We caution the reader that conventions for the orientation of (\ref{unnormsmat}) vary in the literature.} the matrix $\tilde{S}$ where
\begin{equation}[\tilde{S}]_{xy}:=\hopf{x}{y} \ \ , \ \ x,y\in\Irr(\mathcal{C})\label{unnormsmat}\end{equation}
i.e.\ the left and right partial trace of $R^{y^{*}x}\circ R^{xy^{*}}$. The \textit{S-matrix} is $S:=\dfrac{1}{\mathcal{D}}\tilde{S}$ where 
\begin{equation}\mathcal{D}:=\sum_{x\in\Irr(\mathcal{C})}\sqrt{d_{x}^{2}}\end{equation}
is called the \textit{total quantum dimension} of $\mathcal{C}$. A ribbon fusion category $\mathcal{C}$ is called a \textit{modular tensor category} (MTC) if it has a non-degenerate braiding (or equivalently, if the associated $S$-matrix is invertible).

\subsection{Additional conventions}\label{relcase}
Throughout much of this paper, we consider a fusion category $\C$ containing a fusion rule of the form
\begin{align}
    q\otimes q = \bm{1}\oplus\bigoplus_{i=1}^{k}x_{i}\label{ourguysec2}
\end{align}
where $q, x_i\in \Irr(\C)$ and objects $x_{i}$ are distinct. In this context, we fix some conventions:
\begin{itemize}
    \item Unlabelled, unoriented edges are understood to represent edges labelled by the self-dual object $q$.
    \item Greek indices (e.g.\ $\lambda$) will be used to denote elements in $I$ for which $N^{qq}_\lambda \neq0$. Latin indices (e.g. $i$) will be used to denote elements in $I\setminus \{0\}$ for which $N^{qq}_i \neq 0$.
\end{itemize}
For instance, (\ref{crossingdecomp}) may be written as follows for $i=j=q$:
\begin{align}\label{ourguydecomp}
    \Px=\sum_{\lambda}R^{qq}_{\lambda}\frac{\sqrt{d_{\lambda}}}{d_{q}}\OJack{\lambda}= R^{qq}_0\frac{1}{d_q}\Ccm + \sum_{i}R^{qq}_{i}\frac{\sqrt{d_{i}}}{d_{q}}\OJack{i}
\end{align}
where we call the jumping jacks on the right-hand side of (\ref{ourguydecomp}) \textit{$i$-jacks}.


\subsection{Rotation operator}\label{rotopsec}
We follow the conventions of Section \ref{relcase}, and further assume that $\C$ is \textit{pivotal}. 
Let $f\in \End(q^{\otimes 2})$ and $f':=\id_q \otimes f \otimes \id_q\in \End(q^{\otimes 4})$. We define the \textit{rotation operator} 
\begin{align}\label{rotaction}
    \varphi: f &\mapsto f'\mapsto \underbrace{(\id_q \otimes \id_q \otimes \ev_q)}_{\in\text{ } \Hom(q^{\otimes 4},q^{\otimes 2}\otimes \mathbf{1})}\circ f' \circ \underbrace{(\coev_q \otimes \id_q \otimes \id_q)}_{\in \text{ } \Hom(\mathbf{1}\otimes q^{\otimes 2},q^{\otimes 4})}
\end{align}
Hence, $\varphi(f) \in \Hom(\mathbf{1}\otimes q^{\otimes 2},q^{\otimes 2}\otimes \mathbf{1})= \End(q^{\otimes 2})$. Graphically, $\varphi$ acts as an anticlockwise $\frac{\pi}{2}$-rotation on a morphism in $\End(q^{\otimes 2})$:
\begin{equation}\centering\def\svgwidth{7cm}
\begingroup%
  \makeatletter%
  \providecommand\color[2][]{%
    \errmessage{(Inkscape) Color is used for the text in Inkscape, but the package 'color.sty' is not loaded}%
    \renewcommand\color[2][]{}%
  }%
  \providecommand\transparent[1]{%
    \errmessage{(Inkscape) Transparency is used (non-zero) for the text in Inkscape, but the package 'transparent.sty' is not loaded}%
    \renewcommand\transparent[1]{}%
  }%
  \providecommand\rotatebox[2]{#2}%
  \newcommand*\fsize{\dimexpr\f@size pt\relax}%
  \newcommand*\lineheight[1]{\fontsize{\fsize}{#1\fsize}\selectfont}%
  \ifx\svgwidth\undefined%
    \setlength{\unitlength}{405.26453694bp}%
    \ifx\svgscale\undefined%
      \relax%
    \else%
      \setlength{\unitlength}{\unitlength * \real{\svgscale}}%
    \fi%
  \else%
    \setlength{\unitlength}{\svgwidth}%
  \fi%
  \global\let\svgwidth\undefined%
  \global\let\svgscale\undefined%
  \makeatother%
  \begin{picture}(1,0.42062255)%
    \lineheight{1}%
    \setlength\tabcolsep{0pt}%
    \put(0,0){\includegraphics[width=\unitlength,page=1]{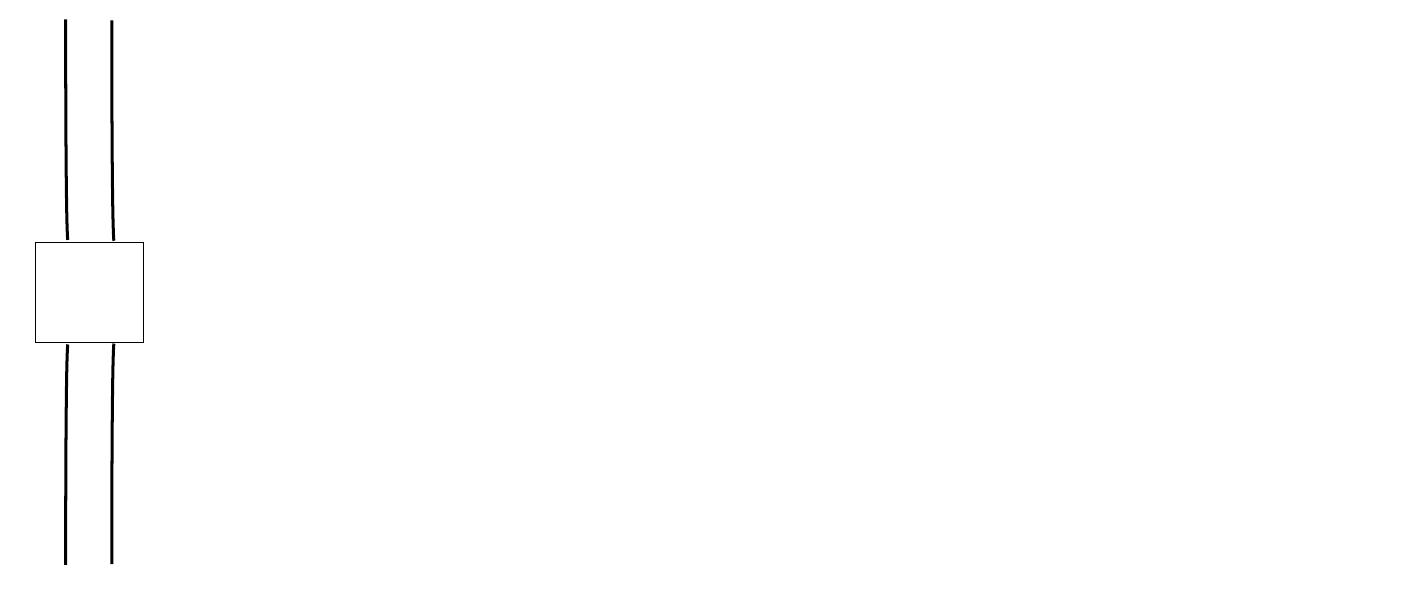}}%
    \put(0.04478872,0.19690946){\color[rgb]{0,0,0}\makebox(0,0)[lt]{\lineheight{1.25}\smash{\begin{tabular}[t]{l}$f$\end{tabular}}}}%
    \put(0,0){\includegraphics[width=\unitlength,page=2]{rotdef.pdf}}%
    \put(0.17834671,0.20726311){\color[rgb]{0,0,0}\makebox(0,0)[lt]{\lineheight{1.25}\smash{\begin{tabular}[t]{l}$\overset{\varphi}{\longmapsto}$\end{tabular}}}}%
    \put(0,0){\includegraphics[width=\unitlength,page=3]{rotdef.pdf}}%
    \put(0.8538355,0.19384292){\color[rgb]{0,0,0}\rotatebox{90}{\makebox(0,0)[lt]{\lineheight{1.25}\smash{\begin{tabular}[t]{l}$f$\end{tabular}}}}}%
    \put(0,0){\includegraphics[width=\unitlength,page=4]{rotdef.pdf}}%
    \put(0.57492705,0.19452072){\color[rgb]{0,0,0}\makebox(0,0)[lt]{\lineheight{1.25}\smash{\begin{tabular}[t]{l}$=$\end{tabular}}}}%
    \put(0.39919801,0.20057056){\color[rgb]{0,0,0}\makebox(0,0)[lt]{\lineheight{1.25}\smash{\begin{tabular}[t]{l}$f$\end{tabular}}}}%
  \end{picture}%
\endgroup%
\label{rotationdef}\end{equation}

\noindent By $\CC$-linearity of $\C$ and bilinearity of the bifunctor $``\otimes"$ on morphisms, note that $\varphi$ is a $\CC$-linear operator. Further note that $\varphi^4(f)=f$ (as demonstrated in (\ref{rotationoperator}), where the final diagram can be straightened to the first diagram).
\begin{equation}\centering\def\svgwidth{17cm}
\begingroup%
  \makeatletter%
  \providecommand\color[2][]{%
    \errmessage{(Inkscape) Color is used for the text in Inkscape, but the package 'color.sty' is not loaded}%
    \renewcommand\color[2][]{}%
  }%
  \providecommand\transparent[1]{%
    \errmessage{(Inkscape) Transparency is used (non-zero) for the text in Inkscape, but the package 'transparent.sty' is not loaded}%
    \renewcommand\transparent[1]{}%
  }%
  \providecommand\rotatebox[2]{#2}%
  \newcommand*\fsize{\dimexpr\f@size pt\relax}%
  \newcommand*\lineheight[1]{\fontsize{\fsize}{#1\fsize}\selectfont}%
  \ifx\svgwidth\undefined%
    \setlength{\unitlength}{855.6141963bp}%
    \ifx\svgscale\undefined%
      \relax%
    \else%
      \setlength{\unitlength}{\unitlength * \real{\svgscale}}%
    \fi%
  \else%
    \setlength{\unitlength}{\svgwidth}%
  \fi%
  \global\let\svgwidth\undefined%
  \global\let\svgscale\undefined%
  \makeatother%
  \begin{picture}(1,0.18240309)%
    \lineheight{1}%
    \setlength\tabcolsep{0pt}%
    \put(0,0){\includegraphics[width=\unitlength,page=1]{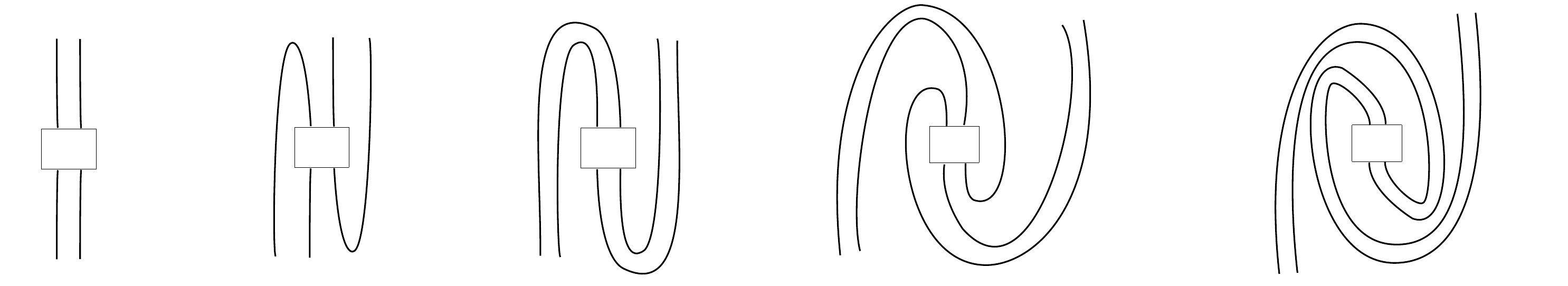}}%
    \put(0.09530062,0.0837714){\color[rgb]{0,0,0}\makebox(0,0)[lt]{\lineheight{1.25}\smash{\begin{tabular}[t]{l}$\overset{\varphi}{\longmapsto}$\end{tabular}}}}%
    \put(0.27253211,0.08447428){\color[rgb]{0,0,0}\makebox(0,0)[lt]{\lineheight{1.25}\smash{\begin{tabular}[t]{l}$\overset{\varphi}{\longmapsto}$\end{tabular}}}}%
    \put(0.46552024,0.08447428){\color[rgb]{0,0,0}\makebox(0,0)[lt]{\lineheight{1.25}\smash{\begin{tabular}[t]{l}$\overset{\varphi}{\longmapsto}$\end{tabular}}}}%
    \put(0.74187922,0.08524672){\color[rgb]{0,0,0}\makebox(0,0)[lt]{\lineheight{1.25}\smash{\begin{tabular}[t]{l}$\overset{\varphi}{\longmapsto}$\end{tabular}}}}%
    \put(0.03431008,0.080714){\color[rgb]{0,0,0}\makebox(0,0)[lt]{\lineheight{1.25}\smash{\begin{tabular}[t]{l}$f$\end{tabular}}}}%
    \put(0.19719625,0.08234704){\color[rgb]{0,0,0}\makebox(0,0)[lt]{\lineheight{1.25}\smash{\begin{tabular}[t]{l}$f$\end{tabular}}}}%
    \put(0.37851661,0.08242064){\color[rgb]{0,0,0}\makebox(0,0)[lt]{\lineheight{1.25}\smash{\begin{tabular}[t]{l}$f$\end{tabular}}}}%
    \put(0.59896283,0.08318257){\color[rgb]{0,0,0}\makebox(0,0)[lt]{\lineheight{1.25}\smash{\begin{tabular}[t]{l}$f$\end{tabular}}}}%
    \put(0.86696369,0.08401728){\color[rgb]{0,0,0}\makebox(0,0)[lt]{\lineheight{1.25}\smash{\begin{tabular}[t]{l}$f$\end{tabular}}}}%
  \end{picture}%
\endgroup%
\label{rotationoperator}\end{equation}
\noindent To the knowledge of the authors, the first instance where the rotation operator was used in a categorical context was in \cite{MSP1}. See Appendix \ref{rotappx} for a supplementary excursion on the rotation of morphisms.

\section{Some Framed Invariants from Ribbon Categories}
\label{skeinrev}
Let $\mathcal{C}$ be a unitary ribbon fusion category containing a fusion rule of the form
\begin{equation}q\otimes q=\bm{1}\oplus\bigoplus_{i=1}^{k}x_{i}\label{ourguysec3}\end{equation}
where $q,x_{i}\in\Irr(\mathcal{C})$ and objects $x_{i}$ are distinct. Framed, oriented links whose components are labelled by elements of $\Irr(\mathcal{C})$ can be thought of as morphisms in $\End(\bm{1})$; the value in $\mathbb{C}$ to which any such link evaluates is invariant under framed isotopy. Restated, given an oriented\footnote{In the instance where all labels are self-dual, $D$ is an unoriented diagram.} link diagram $D$ whose components are labelled as such, there is a complex-valued function whose value is constant on the framed isotopy class of $D$. Such a function coincides with the notion of a \textit{framed link invariant} (when none of the labels are antisymmetrically self-dual). \\
Let $\Lambda_{\mathcal{C},q}$ denote the framed link invariant for oriented links with all components labelled by $q\in\Irr(\mathcal{C})$ symmetrically self-dual. For $q$ antisymmetrically self-dual, $\Lambda_{\mathcal{C},q}$ denotes the polynomial-valued function obtained from applying the associated skein relation to a link diagram.\footnote{This is further explored in Appendix \ref{addendum}.} Our goal is to extract information pertaining to $\Lambda_{\mathcal{C},q}$ when $q$ satisfies (\ref{ourguysec3}); in particular, we use the rotation operator $\varphi$ to find relations amongst $d_{q}$ and the eigenvalues of $R^{qq}$. We do this for the trivial case (i.e.\ $q^{\otimes2}=\bm{1}$) and then for cases $k=1,2$. When $q$ is symmetrically self-dual (i.e. $\varkappa_q=1$), it is easy to see that 
\begin{equation}\tilde{\Lambda}_{\mathcal{C},q}(L)=\vartheta_{q}^{-w(D)}\Lambda_{\mathcal{C},q}(D)\label{woofwatch}\end{equation}
is an oriented link invariant (where $L$ is the oriented link for which $D$ is a diagram, and $w$ is the writhe). \\

\noindent Much of the exposition in this section is already well-established and has been presented in \cite{MSP1} (see Theorems 3.1 \& 3.2) where a broader discussion may be found. However, we choose to include this material for its relevance to our main results in Section \ref{mainresults}. Furthermore, our approach differs slightly to that taken in \cite{MSP1} and we also treat the instances where $q$ is antisymmetrically self-dual (i.e. $\varkappa_q=-1$), which leads to an extended discussion in Appendix \ref{addendum}. In Appendix \ref{reptheory}, the narrative of this section is approached from the perspective of braid group representations. We follow the conventions fixed in Section \ref{relcase}.\\

\noindent $R^{qq}$ is diagonalisable in the canonical basis, so we may resolve crossings as follows:
\vspace{2mm}
\begin{equation}\Px=\sum_{\lambda}R^{qq}_{\lambda}\frac{\sqrt{d_{\lambda}}}{d_{q}}\OJack{\lambda} \quad \text{and} \quad
\Nx=\sum_{\lambda}\left(R^{qq}_{\lambda}\right)^{-1}\frac{\sqrt{d_{\lambda}}}{d_{q}}\OJack{\lambda}\label{resx}\end{equation}
\noindent Also,
\begin{equation}\varphi\left(\ \Px\ \right)=\raisebox{-6mm}{\includegraphics[width=0.085\textwidth]{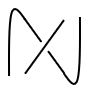}}=\varkappa_{q}\ \Nx \label{rotx}\end{equation}

\vspace{1mm}
\subsection{Trivial case}
\label{trivcase}
Here, the fusion rule is given by
\begin{equation}q\otimes q = \bm{1}\end{equation}
Let $\alpha:=R^{qq}_0$. Then from (\ref{resx}), 
\begin{align}\Px = \frac{\alpha}{d_{q}} \ \Ccm \quad \text{and} \quad \Nx = \frac{\alpha^{-1}}{d_{q}}\ \Ccm\end{align}
whence $\varphi\left(\px\right)=\frac{\alpha}{d_{q}}\idm$. Using (\ref{rotx}) and comparing coefficients, we get
\[ \alpha=\varkappa_q\alpha^{-1} \implies \alpha^2=\varkappa_q.\]
and so we have the following two possibilities:\\[2mm]
\fcolorbox{black}[HTML]{FFFFFF}{\parbox{\textwidth}{
\begin{enumerate}
    \item For $\varkappa_{q}=1$, $R_0^{qq}=\pm1$ with skein relation 
    \begin{equation}\Px=\Nx\text{ \ \ and \ \ } \Loopy =1\label{trivskein1}\end{equation}
    \item For $\varkappa_{q}=-1$, $R_0^{qq}=\pm i$ with skein relation 
    \begin{equation}\Px=-\Nx\text{\ \ and \ \ } \Loopy =1\label{trivskein2}\end{equation}
\end{enumerate}}}

\subsection{k=1}
\label{quadcase} 
Now our fusion rule is of the form
\begin{equation}q\otimes q = \bm{1}\oplus x\end{equation}
Let $\alpha:=R^{qq}_{0}$ and $\beta:=R^{qq}_{x}$. Using (\ref{resx}) and (\ref{orientedid}), we resolve the crossings in the basis $\left\{\idm, \ccm\right\}$ to get 
\begin{subequations}\begin{align}
\Px&=\frac{1}{d_q}(\alpha-\beta)\Ccm+\beta\Idm \label{trotsky1}\\
\Nx&=\frac{1}{d_q}(\alpha^{-1}-\beta^{-1})\Ccm+\beta^{-1}\Idm \label{trotsky2}
\end{align}\end{subequations}
whence
\begin{equation}\varphi\left(\ \Px \ \right)=\beta\Ccm+\frac{1}{d_q}(\alpha-\beta)\Idm\end{equation}
Then using (\ref{rotx}) and comparing coefficients with (\ref{trotsky2}), we have
\begin{subequations}\begin{align}
\Idm&: \quad \varkappa_{q}\beta^{-1}=\frac{1}{d_{q}}(\alpha-\beta)\implies\alpha=\varkappa_{q}d_{q}\beta^{-1}+\beta \label{borzov1}\\
\Ccm&: \quad \varkappa_{q}\beta=\frac{1}{d_{q}}(\alpha^{-1}-\beta^{-1})\implies\alpha^{-1}=\varkappa_{q}d_{q}\beta+\beta^{-1} \label{borzov2}
\end{align}\end{subequations}
which can be solved to get
\begin{equation}d_{q}=-\varkappa_{q}(\beta^{2}+\beta^{-2})\quad , \quad \alpha=-\beta^{-3}\label{quadkeys}\end{equation}
We thus have the following two possibilities:\\[2mm]
\fcolorbox{black}[HTML]{FFFFFF}{\parbox{\textwidth}{
\begin{enumerate}
    \item For $\varkappa_q=1$, $R^{qq}=\diag(-\beta^{-3},\beta)$ with skein relation 
    \begin{equation}\Px=\beta\Idm+\beta^{-1}\Ccm \text{\ \ and \ \ } \Loopy=-(\beta^2+\beta^{-2})\label{kbskein}\end{equation}
    i.e. the \textit{Kauffman bracket}.\vspace{4mm}
    \item For $\varkappa_q=-1$, $R^{qq}=\diag(-\beta^{-3},\beta)$ with skein relation 
    \begin{equation}\Px=\beta\Idm-\beta^{-1}\Ccm \text{ \ \ and \ \ } \Loopy=\beta^2+\beta^{-2}\label{kbtwinskein}\end{equation}
\end{enumerate}}}

\vspace{1mm}
\subsection{k=2}
\label{cubicase}
Our fusion rule is of the form
\begin{equation}q\otimes q = \bm{1}\oplus x\oplus y\end{equation}
Let $\alpha:=R^{qq}_{0}, \beta:=R^{qq}_{x}$ and $\gamma:=R^{qq}_{y}$. Using (\ref{resx}) and (\ref{orientedid}), we resolve the crossings in the basis $\left\{\idm, \ccm, \ojack{x}\right \}$ to get 
\begin{subequations}\begin{align}
\Px&=A \Ccm+B\OJack{x}+C\Idm \label{flageolet1}\\
\Nx&=A' \Ccm+B'\OJack{x}+C^{-1}\Idm \label{flageolet2}
\end{align}\end{subequations}
where 
\[A:=\frac{1}{d_q}(\alpha-\gamma),\text{ } B:=\frac{\sqrt{d_x}}{d_q}(\beta-\gamma),\text{ } C:= \gamma\]
\[A':=\frac{1}{d_q}(\alpha^{-1}-\gamma^{-1}),\text{ } B':=\frac{\sqrt{d_x}}{d_q}(\beta^{-1}-\gamma^{-1})\]

\begin{remark}We henceforth set $\beta \neq \gamma$. The case $B,B'=0$ is treated in Section \ref{specialcubics}.\end{remark}
\noindent Eliminating the $x$-jacks and rearranging yields
\begin{subequations}\begin{align}
    \Px&=\left(A-\frac{A'B}{B'}\right)\Ccm+\left(C-\frac{C^{-1}B}{B'}\right)\Idm+\frac{B}{B'}\Nx \label{bluku1} \\
   \implies \varphi\left(\ \Px\ \right)&=\left(A-\frac{A'B}{B'}\right)\Idm+\left(C-\frac{C^{-1}B}{B'}\right)\Ccm+\varkappa_q\frac{B}{B'}\Px
\end{align}\end{subequations}
\noindent Using (\ref{flageolet1}) to express $\varphi\left(\px\right)$ in our chosen basis,
\begin{align*}
    \varphi\left(\ \Px\ \right)=\left(A-\frac{A'B}{B'}+\varkappa_q\frac{CB}{B'}\right)\Idm+\left(C-\frac{C^{-1}B}{B'}+\varkappa_q\frac{AB}{B'}\right)\Ccm+\varkappa_q\frac{B^2}{B'}\OJack{x}
\end{align*}
\vspace{2mm}
\noindent Applying (\ref{rotx}) and comparing coefficients with (\ref{flageolet2}), we have
\vspace{1.5mm}
\begin{subequations}\begin{align}
\ojack{x}:& \quad B'=\pm B \\
\Ccm:& \quad
 	A'=\varkappa_{q}C+\frac{B}{B'}(A-\varkappa_{q}C^{-1}) \\
\Idm:& \quad
    A'=\varkappa_qC+\frac{B'}{B}(A-\varkappa_qC^{-1})
\end{align}\end{subequations}

\noindent Thus,
\begin{equation}B'=\pm B \quad , \quad A'=\varkappa_{q}(C\mp C^{-1})\pm A\label{thisguy}\end{equation}

\vspace{2.5mm}
\begin{remark}[\textbf{Caveat}]\label{kinkhorz}
When $\varkappa_{q}=-1$, there is a difference in sign betweeen (vertical) twists and their ``horizontal'' counterparts (i.e. a $\frac{\pi}{2}$-rotated version). This is taken into account when solving for Cases $3$ and $4$ below. For instance,
\begin{align*}
\raisebox{-5mm}{\includegraphics[width=0.075\textwidth]{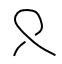}}=u\ \raisebox{-2mm}{\includegraphics[width=0.06\textwidth]{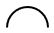}} &\implies \raisebox{-7mm}{\includegraphics[width=0.06\textwidth]{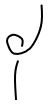}} = u \ \raisebox{-5.5mm}{\includegraphics[width=0.06\textwidth]{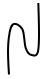}} \\
 &\implies u=\varkappa_{q}\vartheta_{q}
\end{align*}
\end{remark}
\vspace{5mm}
\noindent There are now four cases to examine (for $B'=\pm B$ and $\varkappa_{q}=\pm1$). \\[4mm]
\noindent\underline{Case 1:} Let $B'=B$ and $\varkappa_{q}=1$. Then (\ref{bluku1}) becomes
\begin{equation}
    \Px-\Nx=D\left(\ \Idm-\Ccm\ \right) \label{oo1} 
\end{equation}    
where $D:=C-C^{-1}$. Stacking $\px$ on top of $\nx$ and using (\ref{oo1}) ,
\begin{align*}
\rtwoA&=D^{2}\left(\ \Idm-\Ccm\ \right)\left(\ \Ccm-\Idm\ \right)+D\Big(\ \Px-\CctwistA+\CctwistB-\Nx\ \Big)+\rtwoB \\[-3.5mm]
         &=D^{2}\left(\ \Ccm-\Idm-d_{q}\Ccm+\Ccm\ \right)+D(\vartheta_{q}-\vartheta_{q}^{-1})\Ccm+D\left(\ \Px-\Nx \ \right)+\rtwoB \\
         \overset{(\ref{oo1})}&{=}D\left[(2-d_{q})D+\vartheta_{q}-\vartheta_{q}^{-1}\right]\Ccm-D^{2}\Idm+D^{2}\left(\ \Idm-\Ccm \ \right)+\rtwoB \\
         &=D\left[(1-d_{q})D+\vartheta_{q}-\vartheta_{q}^{-1}\right]\Ccm+\rtwoB
\end{align*}whence by Reidemeister-II,
\begin{center}\begin{tabular}{ccccccc}
(i) $D=\cfrac{\vartheta_{q}-\vartheta_{q}^{-1}}{d_{q}-1}$ & & or & & (ii) $D=0$
\end{tabular}\end{center}
For (i), note that $D$ is well-defined since $d_{q}>1$. Then
\begin{equation}d_{q}=\cfrac{\alpha^{-1}-\alpha}{\gamma-\gamma^{-1}}+1=\cfrac{\alpha^{-1}-\alpha}{\beta-\beta^{-1}}+1\label{greenjam}\end{equation}
where the second equality\footnote{Heuristically, this equality also follows by the symmetry of the construction in $\beta$ and $\gamma$. That is, we could have alternatively used the $y$-jack in our chosen basis.} follows from $B'=B$. Since $D\neq0$, we have $\beta,\gamma\neq\pm1$ and so (\ref{greenjam}) is also well-defined.\\

\vspace{2.5mm}
\noindent\underline{Case 2:} Let $B'=-B$ and $\varkappa_q=1$. Then (\ref{bluku1}) becomes
\begin{align}\Px+\Nx=K\left(\ \Idm+\Ccm\ \right)\end{align}
where $K:=C+C^{-1}$. Similarly, we get
\begin{center}\begin{tabular}{ccccccc}
(i) $K=\cfrac{\vartheta_{q}+\vartheta_{q}^{-1}}{d_{q}+1}$ & & or & & (ii) $K=0$
\end{tabular}\end{center}
where for (i), 
\begin{equation}d_{q}=\cfrac{\alpha^{-1}+\alpha}{\gamma+\gamma^{-1}}-1=\cfrac{\alpha^{-1}+\alpha}{\beta+\beta^{-1}}-1\label{bluejam}\end{equation}

\vspace{7.5mm}
\noindent\underline{Case 3:} Let $B'=B$ and $\varkappa_q=-1$. Then (\ref{bluku1}) becomes
\begin{align}\Px-\Nx=D\left(\ \Idm+\Ccm\ \right)\end{align}
Similarly, we get
\begin{center}\begin{tabular}{ccccccc}
(i) $D=\cfrac{\vartheta_{q}-\vartheta_{q}^{-1}}{d_{q}+1}$ & & or & & (ii) $D=0$
\end{tabular}\end{center}
where for (i), 
\begin{equation}d_{q}=\cfrac{\alpha-\alpha^{-1}}{\gamma-\gamma^{-1}}-1=\cfrac{\alpha-\alpha^{-1}}{\beta-\beta^{-1}}-1\label{redjam}\end{equation}

\vspace{7.5mm}
\noindent\underline{Case 4:} Let $B'=-B$ and $\varkappa_q=-1$. Then (\ref{bluku1}) becomes
\begin{align}\Px+\Nx=K\left(\ \Idm-\Ccm\ \right)\end{align}
Similarly, we get
\begin{center}\begin{tabular}{ccccccc}
(i) $K=\cfrac{\vartheta_{q}+\vartheta_{q}^{-1}}{d_{q}-1}$ & & or & & (ii) $K=0$
\end{tabular}\end{center}
where for (i), 
\begin{equation}d_{q}=-\cfrac{\alpha^{-1}+\alpha}{\gamma+\gamma^{-1}}+1=-\cfrac{\alpha^{-1}+\alpha}{\beta+\beta^{-1}}+1\label{bluejam}\end{equation}

\vspace{4mm}
\begin{remark}Cases $D=0$ and $K=0$ for $\varkappa_{q}=\pm1$ are covered in Section \ref{specialcubics}.\end{remark}
\vspace{2mm}
\noindent For $B'=B$ we have $\beta-\beta^{-1}=\gamma-\gamma^{-1}=D$ whence $\sin(\arg\beta)=\sin(\arg\gamma)$. Since $\beta\neq\gamma$, we have $\arg\beta+\arg\gamma=\pi$. Thus, 
\begin{equation}\beta\gamma=-1 \ \ \text{and} \ \ D=\beta+\gamma \quad , \quad B'=B\end{equation}
For $B'=-B$ we have $\beta+\beta^{-1}=\gamma+\gamma^{-1}=K$ whence $\cos(\arg\beta)=\cos(\arg\gamma)$. Since $\beta\neq\gamma$, we have $\arg\beta+\arg\gamma=2\pi$. Thus, 
\begin{equation}\beta\gamma=+1 \ \ \text{and} \ \ K=\beta+\gamma \quad , \quad B'=-B\end{equation}
Following the notation in \cite{MSP1}, we let $z:=\beta+\gamma$ and $a:=\vartheta_{q}$. Summarising these four cases (where $\beta\neq\gamma$ and $z\neq0$),\\[2mm]
\fcolorbox{black}[HTML]{FFFFFF}{\parbox{\textwidth}{
\begin{enumerate}
    \item For $B'=B$ and $\varkappa_q=1$, $R^{qq}=\diag(\alpha,\beta,-\beta^{-1})$ with skein relation 
    \begin{equation}\Px-\Nx=z\left(\ \Idm-\Ccm\ \right) \text{ \ \ and \ \ } \Loopy=\cfrac{a-a^{-1}}{z}+1\label{dubskein}\end{equation}
    i.e. the \textit{framed Dubrovnik polynomial}.\vspace{4mm}
    \item For $B'=-B$ and $\varkappa_q=1$, $R^{qq}=\diag(\alpha,\beta,\beta^{-1})$ with skein relation 
    \begin{equation}\Px+\Nx=z\left(\ \Idm + \Ccm \ \right) \text{ \ \ and \ \ } \Loopy=\cfrac{a+a^{-1}}{z}-1\label{kauffskein}\end{equation}
    i.e.\ the \textit{framed Kauffman polynomial}.\vspace{4mm}
     \item For $B'=B$ and $\varkappa_q=-1$, $R^{qq}=\diag(\alpha,\beta,-\beta^{-1})$ with skein relation 
    \begin{equation}\Px-\Nx=z\left(\ \Idm+\Ccm\ \right) \text{ \ \ and \ \ } \Loopy=\cfrac{a-a^{-1}}{z}-1\label{dubtwinskein}\end{equation}
     \item For $B'=-B$ and $\varkappa_q=-1$, $R^{qq}=\diag(\alpha,\beta,\beta^{-1})$ with skein relation 
    \begin{equation}\Px+\Nx=z\left(\ \Idm-\Ccm\ \right) \text{ \ \ and \ \ } \Loopy=\cfrac{a+a^{-1}}{z}+1\label{kaufftwinskein}\end{equation}   
\end{enumerate}}}
\vspace{2mm}

\begin{remark}Let $L$ denote a link, $D$ a corresponding diagram, and $w(D)$ the writhe of $D$ (given some choice of orientation on $D$). Let $\Lambda^{(1)}_{(a,z)}$ and $\Lambda^{(2)}_{(a,z)}$ respectively denote the framed Dubrovnik (\ref{dubskein}) and framed Kauffman (\ref{kauffskein}) polynomial. Then
\begin{equation}\Lambda^{(1)}_{(a,z)}(D)=i^{-w(D)}(-1)^{c(L)}\Lambda^{(2)}_{(ia,-iz)}(D)\label{licktrick}\end{equation}
noting that the writhe does not depend on the choice of orientation modulo $4$. The relation (\ref{licktrick}) was proved by Lickorish in \cite{lickorish}.\footnote{There is a sign error in the statement of (\ref{licktrick}) in \cite{lickorish} which has been corrected here. The same correction appears in \cite{MSP1}.}\end{remark}

\vspace{2mm}
\subsubsection{Special cases}
\label{specialcubics}

\hspace{3mm}

\vspace{2mm}
\begin{enumerate}

\item Suppose $B,B'=0$ (i.e.\ $\beta=\gamma$). Then (\ref{flageolet1}) and (\ref{flageolet2}) become (\ref{trotsky1}) and (\ref{trotsky2}). That is, the crossings lie in the subspace $\spn\left\{\idm,\ccm\right\}$. Thus,
\begin{center}\begin{tabular}{ccc}
 & $\varkappa_{q}=1$ & $\varkappa_{q}=-1$ \\
\hline\addlinespace[1.5pt]
Skein relation & (\ref{kbskein}) & (\ref{kbtwinskein})  \\[1pt]
$R^{qq}$ & $\diag(-\beta^{-3},\beta,\beta)$ & $\diag(-\beta^{-3},\beta,\beta)$
\end{tabular}\end{center}

\vspace{2mm}
\item Suppose $D=0$ with $\varkappa_{q}=\pm1$. Then 
\[(A',B',C^{-1})=(A,B,C) \implies \alpha,\beta,\gamma\in\{\pm1\}\]
and using (\ref{flageolet1}),(\ref{flageolet2}) and $\beta\neq\gamma$ we get
\begin{equation}R^{qq}=(\alpha,\pm1,\mp1) \ \ \text{and} \ \ \Px=\Nx \quad ,\quad \alpha\in\{\pm1\} \end{equation}

\vspace{2mm}
\item Suppose $K=0$ with $\varkappa_{q}=\pm1$. Then 
\[(A',B',C^{-1})=(-A,-B,-C) \implies \alpha,\beta,\gamma\in\{\pm i\}\]
and using (\ref{flageolet1}),(\ref{flageolet2}) and $\beta\neq\gamma$ we get
\begin{equation}R^{qq}=(\alpha,\pm i,\mp i) \ \ \text{and} \ \ \Px=-\Nx \quad ,\quad \alpha\in\{\pm i\} \end{equation}
\end{enumerate}

\vspace{2mm}

\section{Main Results}
\label{mainresults}
\noindent Let $\mathcal{C}$ be a unitary spherical fusion category containing a fusion rule of the form
\begin{equation}q\otimes q= \bm{1}\oplus\bigoplus_{i}x_{i}\label{mainrule}\end{equation}
where $q,x_{i}\in\Irr(\mathcal{C})$, objects $x_{i}$ are distinct and where $N:=\dim\left(\End(q^{\otimes2})\right)\geq2$.\\
We write $f_{\mu\nu}:=\left[F^{qqq}_{q}\right]_{\mu\nu}$. For any simple object $x$ in the decomposition of $q^{\otimes2}$, the symmetries of the fusion coefficients give  $N^{qq}_{x}=N^{xq}_{q}=N^{qx}_{q}=N^{qq}_{x^{*}}$. Firstly, this tells us that the indices of $f_{\mu\nu}$ run over  $\bm{1}$ and $\{x_{i}\}_{i}$. Secondly, this tells us that the set $\{x_{i}\}_{i}$ is closed under taking duals: this allows us to define a \textit{(charge) conjugation matrix} $\mathscr{C}:=\delta_{\mu\mu^{*}}$ where $\mu$ indexes  $\bm{1}$ and $\{x_{i}\}_{i}$. We follow the conventions from Section \ref{relcase} and let $\sigma(A)$ denote the spectrum of a linear operator $A$.

\vspace{2mm}
\subsection{Rotation operator in the canonical basis}
\label{rotcanonsec}

\vspace{2mm}
\begin{lemma}\hspace{2mm}
\vspace{2mm}
\label{treecapping}
\begin{enumerate}[label=(\roman*)]
\item $f_{0\lambda}=\varkappa_{q}\frac{\sqrt{d_{\lambda}}}{d_{q}}$
\item $\delta_{\lambda0}=\varkappa_{q}\sum_{\rho}\frac{\sqrt{d_{\rho}}}{d_{q}}f_{\rho\lambda}$
\item \hspace{2mm}\begin{minipage}[t]{\linewidth}
          \raggedright
          \adjustbox{valign=t}{\centering\def\svgwidth{5cm}
\begingroup%
  \makeatletter%
  \providecommand\color[2][]{%
    \errmessage{(Inkscape) Color is used for the text in Inkscape, but the package 'color.sty' is not loaded}%
    \renewcommand\color[2][]{}%
  }%
  \providecommand\transparent[1]{%
    \errmessage{(Inkscape) Transparency is used (non-zero) for the text in Inkscape, but the package 'transparent.sty' is not loaded}%
    \renewcommand\transparent[1]{}%
  }%
  \providecommand\rotatebox[2]{#2}%
  \newcommand*\fsize{\dimexpr\f@size pt\relax}%
  \newcommand*\lineheight[1]{\fontsize{\fsize}{#1\fsize}\selectfont}%
  \ifx\svgwidth\undefined%
    \setlength{\unitlength}{311.4412515bp}%
    \ifx\svgscale\undefined%
      \relax%
    \else%
      \setlength{\unitlength}{\unitlength * \real{\svgscale}}%
    \fi%
  \else%
    \setlength{\unitlength}{\svgwidth}%
  \fi%
  \global\let\svgwidth\undefined%
  \global\let\svgscale\undefined%
  \makeatother%
  \begin{picture}(1,0.51551043)%
    \lineheight{1}%
    \setlength\tabcolsep{0pt}%
    \put(0,0){\includegraphics[width=\unitlength,page=1]{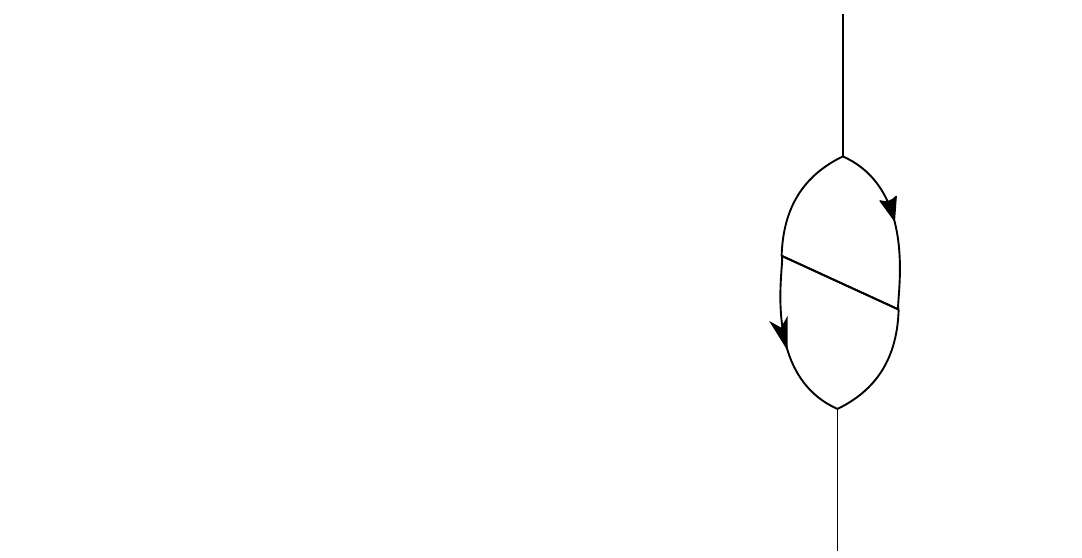}}%
    \put(0.66827567,0.34022088){\color[rgb]{0,0,0}\makebox(0,0)[lt]{\lineheight{1.25}\smash{\begin{tabular}[t]{l}$q$\end{tabular}}}}%
    \put(0.84454319,0.32303264){\color[rgb]{0,0,0}\makebox(0,0)[lt]{\lineheight{1.25}\smash{\begin{tabular}[t]{l}$\mu$\end{tabular}}}}%
    \put(0.64114654,0.17314193){\color[rgb]{0,0,0}\makebox(0,0)[lt]{\lineheight{1.25}\smash{\begin{tabular}[t]{l}$\lambda$\end{tabular}}}}%
    \put(0.84532579,0.16245874){\color[rgb]{0,0,0}\makebox(0,0)[lt]{\lineheight{1.25}\smash{\begin{tabular}[t]{l}$q$\end{tabular}}}}%
    \put(0,0){\includegraphics[width=\unitlength,page=2]{thetastump1.pdf}}%
    \put(0.3646254,0.02597784){\color[rgb]{0,0,0}\makebox(0,0)[lt]{\lineheight{1.25}\smash{\begin{tabular}[t]{l}$q$\end{tabular}}}}%
    \put(0.46892609,0.22685489){\color[rgb]{0,0,0}\makebox(0,0)[lt]{\lineheight{1.25}\smash{\begin{tabular}[t]{l}=\end{tabular}}}}%
    \put(0.01514449,0.24347368){\color[rgb]{0,0,0}\makebox(0,0)[lt]{\lineheight{1.25}\smash{\begin{tabular}[t]{l}$\varkappa_{q}d_qf_{\mu\lambda}$\end{tabular}}}}%
  \end{picture}%
\endgroup%
}\medskip
          \end{minipage}
\item \begin{minipage}[t]{\linewidth}
          \raggedright
          \adjustbox{valign=t}{\centering\def\svgwidth{6.6cm}
\begingroup%
  \makeatletter%
  \providecommand\color[2][]{%
    \errmessage{(Inkscape) Color is used for the text in Inkscape, but the package 'color.sty' is not loaded}%
    \renewcommand\color[2][]{}%
  }%
  \providecommand\transparent[1]{%
    \errmessage{(Inkscape) Transparency is used (non-zero) for the text in Inkscape, but the package 'transparent.sty' is not loaded}%
    \renewcommand\transparent[1]{}%
  }%
  \providecommand\rotatebox[2]{#2}%
  \newcommand*\fsize{\dimexpr\f@size pt\relax}%
  \newcommand*\lineheight[1]{\fontsize{\fsize}{#1\fsize}\selectfont}%
  \ifx\svgwidth\undefined%
    \setlength{\unitlength}{344.63479831bp}%
    \ifx\svgscale\undefined%
      \relax%
    \else%
      \setlength{\unitlength}{\unitlength * \real{\svgscale}}%
    \fi%
  \else%
    \setlength{\unitlength}{\svgwidth}%
  \fi%
  \global\let\svgwidth\undefined%
  \global\let\svgscale\undefined%
  \makeatother%
  \begin{picture}(1,0.42708106)%
    \lineheight{1}%
    \setlength\tabcolsep{0pt}%
    \put(0.61330095,0.28486519){\color[rgb]{0,0,0}\makebox(0,0)[lt]{\lineheight{1.25}\smash{\begin{tabular}[t]{l}$\mu$\end{tabular}}}}%
    \put(0.76432644,0.27280344){\color[rgb]{0,0,0}\makebox(0,0)[lt]{\lineheight{1.25}\smash{\begin{tabular}[t]{l}$q$\end{tabular}}}}%
    \put(0.61109522,0.14650122){\color[rgb]{0,0,0}\makebox(0,0)[lt]{\lineheight{1.25}\smash{\begin{tabular}[t]{l}$q$\end{tabular}}}}%
    \put(0.76406567,0.13103911){\color[rgb]{0,0,0}\makebox(0,0)[lt]{\lineheight{1.25}\smash{\begin{tabular}[t]{l}$\rho$\end{tabular}}}}%
    \put(0,0){\includegraphics[width=\unitlength,page=1]{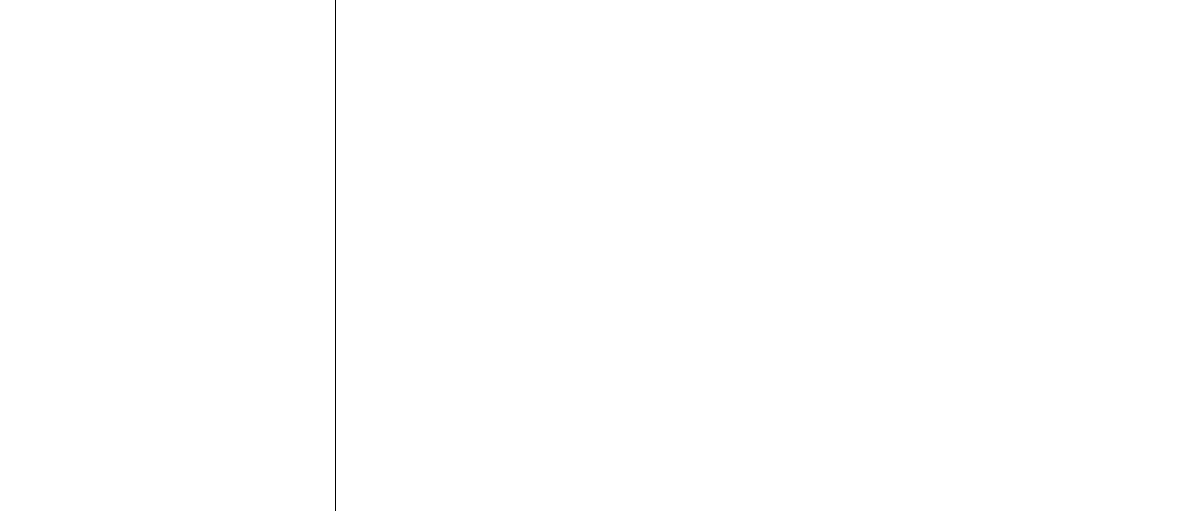}}%
    \put(0.29372403,0.01810026){\color[rgb]{0,0,0}\makebox(0,0)[lt]{\lineheight{1.25}\smash{\begin{tabular}[t]{l}$q$\end{tabular}}}}%
    \put(0.34235226,0.18689202){\color[rgb]{0,0,0}\makebox(0,0)[lt]{\lineheight{1.25}\smash{\begin{tabular}[t]{l}=\end{tabular}}}}%
    \put(-0.00295285,0.20100206){\color[rgb]{0,0,0}\makebox(0,0)[lt]{\lineheight{1.25}\smash{\begin{tabular}[t]{l}$\varkappa_{q}d_q\delta_{\mu\lambda}$\end{tabular}}}}%
    \put(0,0){\includegraphics[width=\unitlength,page=2]{thetastump2.pdf}}%
    \put(0.39800606,0.19878425){\color[rgb]{0,0,0}\makebox(0,0)[lt]{\lineheight{1.25}\smash{\begin{tabular}[t]{l}$\sum_\rho f_{\rho\lambda}$\end{tabular}}}}%
  \end{picture}%
\endgroup%
}\medskip
          \end{minipage}
\end{enumerate}
\end{lemma}
\begin{proof}
\begin{equation}\centering\def\svgwidth{7cm}
\begingroup%
  \makeatletter%
  \providecommand\color[2][]{%
    \errmessage{(Inkscape) Color is used for the text in Inkscape, but the package 'color.sty' is not loaded}%
    \renewcommand\color[2][]{}%
  }%
  \providecommand\transparent[1]{%
    \errmessage{(Inkscape) Transparency is used (non-zero) for the text in Inkscape, but the package 'transparent.sty' is not loaded}%
    \renewcommand\transparent[1]{}%
  }%
  \providecommand\rotatebox[2]{#2}%
  \newcommand*\fsize{\dimexpr\f@size pt\relax}%
  \newcommand*\lineheight[1]{\fontsize{\fsize}{#1\fsize}\selectfont}%
  \ifx\svgwidth\undefined%
    \setlength{\unitlength}{310.00965509bp}%
    \ifx\svgscale\undefined%
      \relax%
    \else%
      \setlength{\unitlength}{\unitlength * \real{\svgscale}}%
    \fi%
  \else%
    \setlength{\unitlength}{\svgwidth}%
  \fi%
  \global\let\svgwidth\undefined%
  \global\let\svgscale\undefined%
  \makeatother%
  \begin{picture}(1,0.26680534)%
    \lineheight{1}%
    \setlength\tabcolsep{0pt}%
    \put(0,0){\includegraphics[width=\unitlength,page=1]{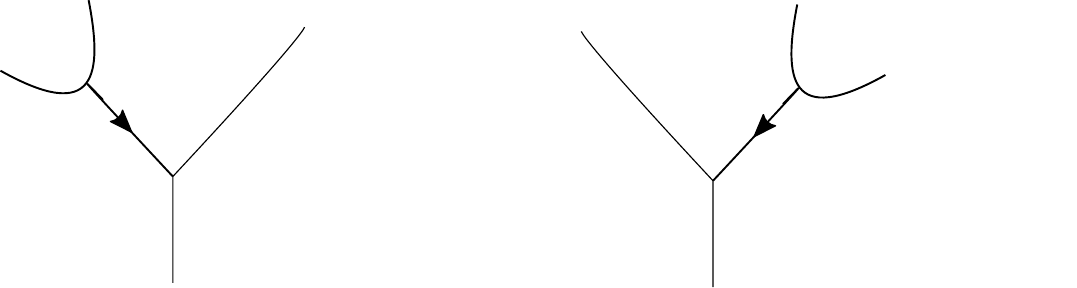}}%
    \put(0.07362999,0.12656347){\color[rgb]{0,0,0}\makebox(0,0)[lt]{\lineheight{1.25}\smash{\begin{tabular}[t]{l}$\lambda$\end{tabular}}}}%
    \put(0.30828362,0.13551801){\color[rgb]{0,0,0}\makebox(0,0)[lt]{\lineheight{1.25}\smash{\begin{tabular}[t]{l}$= \sum_\rho f_{\rho\lambda}$\end{tabular}}}}%
    \put(0.72233251,0.12281143){\color[rgb]{0,0,0}\makebox(0,0)[lt]{\lineheight{1.25}\smash{\begin{tabular}[t]{l}$\rho$\end{tabular}}}}%
  \end{picture}%
\endgroup%
\label{lemtreebase}\end{equation}
\begin{enumerate}[label=(\roman*)]

\vspace{3.5mm}
\item Capping off the rightmost pair of leaves in (\ref{lemtreebase}) gives 
\vspace{7mm}
\begin{equation*}\hspace{20mm}\centering\def\svgwidth{18cm}
\begingroup%
  \makeatletter%
  \providecommand\color[2][]{%
    \errmessage{(Inkscape) Color is used for the text in Inkscape, but the package 'color.sty' is not loaded}%
    \renewcommand\color[2][]{}%
  }%
  \providecommand\transparent[1]{%
    \errmessage{(Inkscape) Transparency is used (non-zero) for the text in Inkscape, but the package 'transparent.sty' is not loaded}%
    \renewcommand\transparent[1]{}%
  }%
  \providecommand\rotatebox[2]{#2}%
  \newcommand*\fsize{\dimexpr\f@size pt\relax}%
  \newcommand*\lineheight[1]{\fontsize{\fsize}{#1\fsize}\selectfont}%
  \ifx\svgwidth\undefined%
    \setlength{\unitlength}{899.05266709bp}%
    \ifx\svgscale\undefined%
      \relax%
    \else%
      \setlength{\unitlength}{\unitlength * \real{\svgscale}}%
    \fi%
  \else%
    \setlength{\unitlength}{\svgwidth}%
  \fi%
  \global\let\svgwidth\undefined%
  \global\let\svgscale\undefined%
  \makeatother%
  \begin{picture}(1,0.13564258)%
    \lineheight{1}%
    \setlength\tabcolsep{0pt}%
    \put(0,0){\includegraphics[width=\unitlength,page=1]{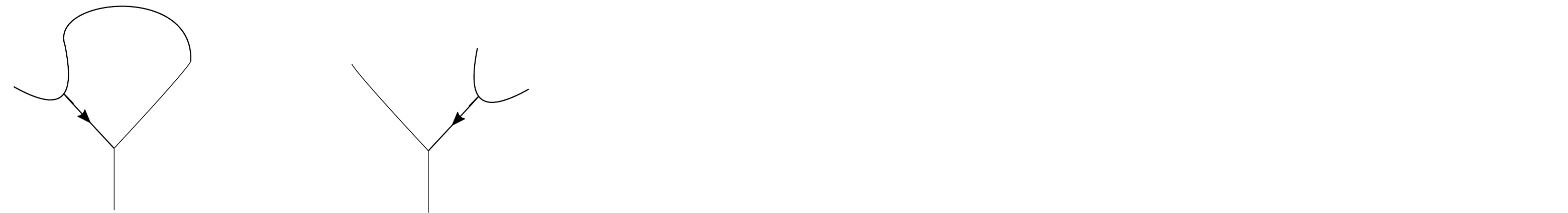}}%
    \put(0.03617034,0.04811769){\color[rgb]{0,0,0}\makebox(0,0)[lt]{\lineheight{1.25}\smash{\begin{tabular}[t]{l}$\lambda$\end{tabular}}}}%
    \put(0.1239063,0.05413103){\color[rgb]{0,0,0}\makebox(0,0)[lt]{\lineheight{1.25}\smash{\begin{tabular}[t]{l}$= \sum_\rho f_{\rho\lambda}$\end{tabular}}}}%
    \put(0.29163867,0.05167782){\color[rgb]{0,0,0}\makebox(0,0)[lt]{\lineheight{1.25}\smash{\begin{tabular}[t]{l}$\rho$\end{tabular}}}}%
    \put(0,0){\includegraphics[width=\unitlength,page=2]{treecappf1.pdf}}%
    \put(0.33038421,0.05568671){\color[rgb]{0,0,0}\makebox(0,0)[lt]{\lineheight{1.25}\smash{\begin{tabular}[t]{l}$\implies \varkappa_{q}$\end{tabular}}}}%
    \put(0,0){\includegraphics[width=\unitlength,page=3]{treecappf1.pdf}}%
    \put(0.42183722,0.07821981){\color[rgb]{0,0,0}\makebox(0,0)[lt]{\lineheight{1.25}\smash{\begin{tabular}[t]{l}$\lambda$\end{tabular}}}}%
    \put(0.51785595,0.05568671){\color[rgb]{0,0,0}\makebox(0,0)[lt]{\lineheight{1.25}\smash{\begin{tabular}[t]{l}$= \sum_\rho f_{\rho\lambda}\delta_{\rho0}d_q$\end{tabular}}}}%
    \put(0,0){\includegraphics[width=\unitlength,page=4]{treecappf1.pdf}}%
  \end{picture}%
\endgroup%
\end{equation*}        
\[\hspace{34mm}\implies \varkappa_q \frac{\sqrt{d_\lambda}}{d_q} = \sum_\rho f_{\rho\lambda}\delta_{\rho 0}.
\]

\vspace{20mm}
\item Capping off the leftmost pair of leaves in (\ref{lemtreebase}) gives
\vspace{7mm} \begin{equation*}\hspace{20mm}\centering\def\svgwidth{21cm}
\begingroup%
  \makeatletter%
  \providecommand\color[2][]{%
    \errmessage{(Inkscape) Color is used for the text in Inkscape, but the package 'color.sty' is not loaded}%
    \renewcommand\color[2][]{}%
  }%
  \providecommand\transparent[1]{%
    \errmessage{(Inkscape) Transparency is used (non-zero) for the text in Inkscape, but the package 'transparent.sty' is not loaded}%
    \renewcommand\transparent[1]{}%
  }%
  \providecommand\rotatebox[2]{#2}%
  \newcommand*\fsize{\dimexpr\f@size pt\relax}%
  \newcommand*\lineheight[1]{\fontsize{\fsize}{#1\fsize}\selectfont}%
  \ifx\svgwidth\undefined%
    \setlength{\unitlength}{968.57008566bp}%
    \ifx\svgscale\undefined%
      \relax%
    \else%
      \setlength{\unitlength}{\unitlength * \real{\svgscale}}%
    \fi%
  \else%
    \setlength{\unitlength}{\svgwidth}%
  \fi%
  \global\let\svgwidth\undefined%
  \global\let\svgscale\undefined%
  \makeatother%
  \begin{picture}(1,0.1600223)%
    \lineheight{1}%
    \setlength\tabcolsep{0pt}%
    \put(0.10031634,0.09175964){\color[rgb]{0,0,0}\makebox(0,0)[lt]{\lineheight{1.25}\smash{\begin{tabular}[t]{l}$= \sum_\rho f_{\rho\lambda}$\end{tabular}}}}%
    \put(0.27699855,0.09303915){\color[rgb]{0,0,0}\makebox(0,0)[lt]{\lineheight{1.25}\smash{\begin{tabular}[t]{l}$\implies \delta_{\lambda0}d_{q}$\end{tabular}}}}%
    \put(0,0){\includegraphics[width=\unitlength,page=1]{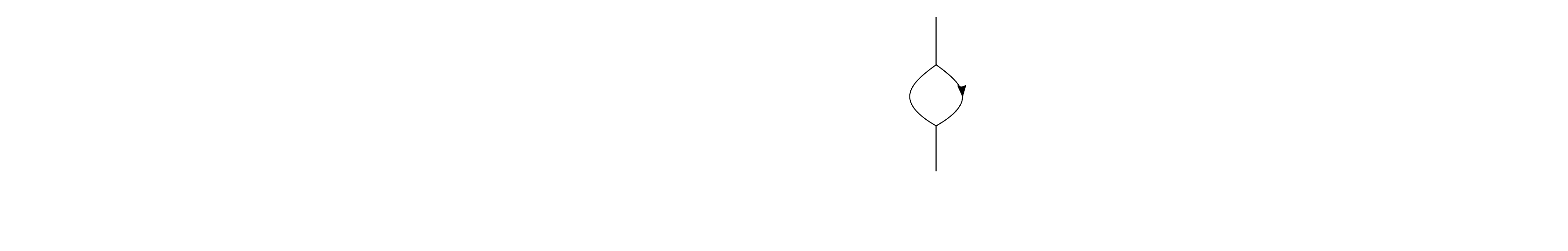}}%
    \put(0.61744223,0.11440887){\color[rgb]{0,0,0}\makebox(0,0)[lt]{\lineheight{1.25}\smash{\begin{tabular}[t]{l}$\rho$\end{tabular}}}}%
    \put(0.44192771,0.09408957){\color[rgb]{0,0,0}\makebox(0,0)[lt]{\lineheight{1.25}\smash{\begin{tabular}[t]{l}$= \varkappa_q \sum_\rho f_{\rho\lambda}$\end{tabular}}}}%
    \put(0,0){\includegraphics[width=\unitlength,page=2]{treecappf2.pdf}}%
    \put(0.02554332,0.09168533){\color[rgb]{0,0,0}\makebox(0,0)[lt]{\lineheight{1.25}\smash{\begin{tabular}[t]{l}$\lambda$\end{tabular}}}}%
    \put(0.24376351,0.09326109){\color[rgb]{0,0,0}\makebox(0,0)[lt]{\lineheight{1.25}\smash{\begin{tabular}[t]{l}$\rho$\end{tabular}}}}%
    \put(0.56627462,0.10796837){\color[rgb]{0,0,0}\makebox(0,0)[lt]{\lineheight{1.25}\smash{\begin{tabular}[t]{l}$q$\end{tabular}}}}%
    \put(0.28308688,0.00275458){\color[rgb]{0,0,0}\makebox(0,0)[lt]{\lineheight{1.25}\smash{\begin{tabular}[t]{l}$\implies \delta_{\lambda0}=\varkappa_q \sum_{\rho}\dfrac{\sqrt{d_\rho}}{d_q}f_{\rho\lambda}$\end{tabular}}}}%
  \end{picture}%
\endgroup%
\end{equation*}   

\vspace{20mm}
\item Stacking the adjoint tree of the right-hand side on (\ref{lemtreebase}) gives 
\vspace{7mm}
\begin{equation*}\hspace{20mm}\centering\def\svgwidth{13.5cm}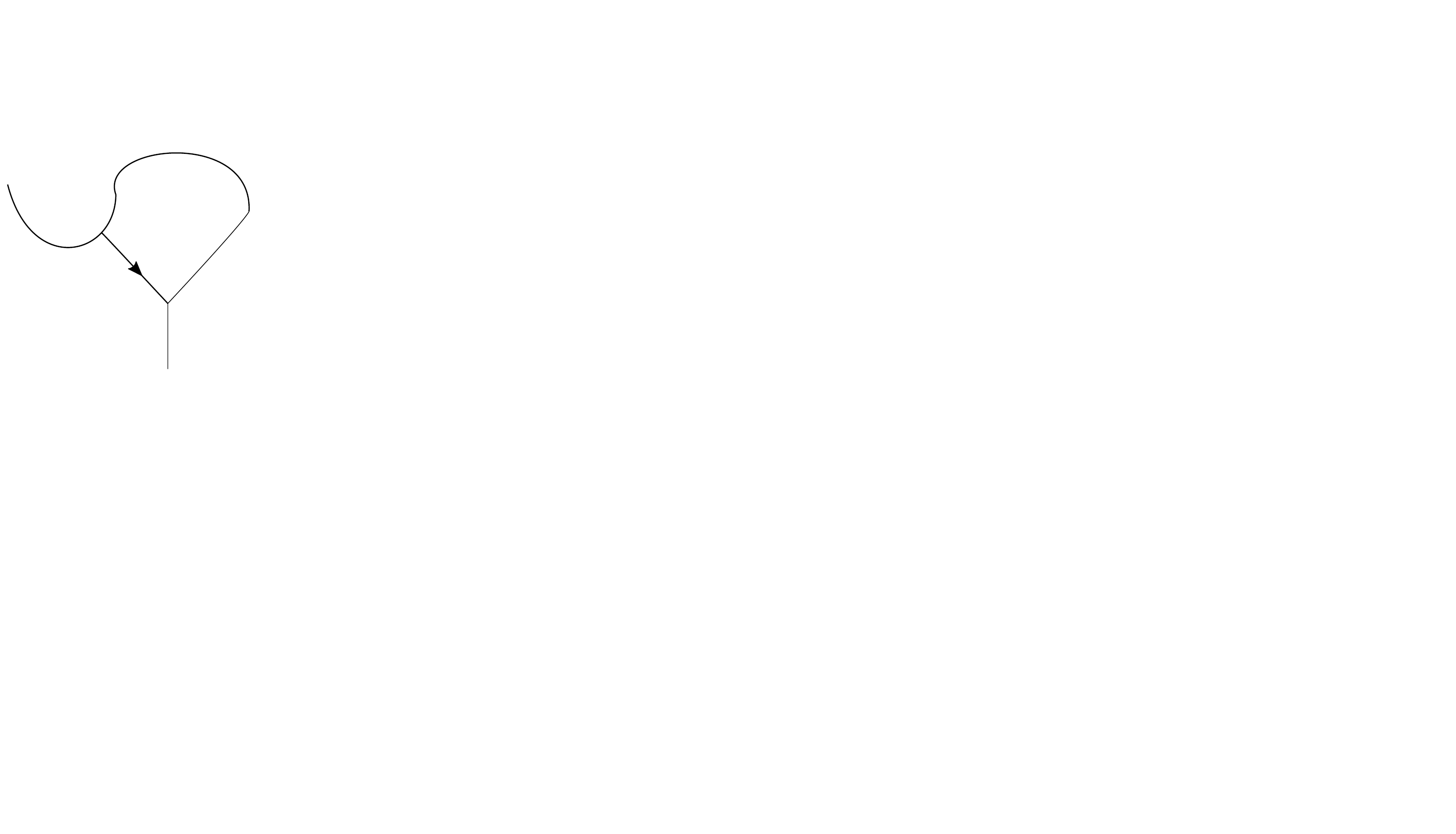\end{equation*}

\newpage
\item Stacking the adjoint tree of the left-hand side on (\ref{lemtreebase}) gives 
\vspace{4mm}
\begin{equation*}\hspace{20mm}\centering\def\svgwidth{13.5cm}
\begingroup%
  \makeatletter%
  \providecommand\color[2][]{%
    \errmessage{(Inkscape) Color is used for the text in Inkscape, but the package 'color.sty' is not loaded}%
    \renewcommand\color[2][]{}%
  }%
  \providecommand\transparent[1]{%
    \errmessage{(Inkscape) Transparency is used (non-zero) for the text in Inkscape, but the package 'transparent.sty' is not loaded}%
    \renewcommand\transparent[1]{}%
  }%
  \providecommand\rotatebox[2]{#2}%
  \newcommand*\fsize{\dimexpr\f@size pt\relax}%
  \newcommand*\lineheight[1]{\fontsize{\fsize}{#1\fsize}\selectfont}%
  \ifx\svgwidth\undefined%
    \setlength{\unitlength}{926.76615557bp}%
    \ifx\svgscale\undefined%
      \relax%
    \else%
      \setlength{\unitlength}{\unitlength * \real{\svgscale}}%
    \fi%
  \else%
    \setlength{\unitlength}{\svgwidth}%
  \fi%
  \global\let\svgwidth\undefined%
  \global\let\svgscale\undefined%
  \makeatother%
  \begin{picture}(1,0.43181235)%
    \lineheight{1}%
    \setlength\tabcolsep{0pt}%
    \put(0.02625688,0.27183255){\color[rgb]{0,0,0}\makebox(0,0)[lt]{\lineheight{1.25}\smash{\begin{tabular}[t]{l}$\lambda$\end{tabular}}}}%
    \put(0.02102523,0.3644401){\color[rgb]{0,0,0}\makebox(0,0)[lt]{\lineheight{1.25}\smash{\begin{tabular}[t]{l}$\mu$\end{tabular}}}}%
    \put(0.31587482,0.27604117){\color[rgb]{0,0,0}\makebox(0,0)[lt]{\lineheight{1.25}\smash{\begin{tabular}[t]{l}$\rho$\end{tabular}}}}%
    \put(0.28828843,0.36888468){\color[rgb]{0,0,0}\makebox(0,0)[lt]{\lineheight{1.25}\smash{\begin{tabular}[t]{l}$\mu$\end{tabular}}}}%
    \put(0.12363609,0.3205677){\color[rgb]{0,0,0}\makebox(0,0)[lt]{\lineheight{1.25}\smash{\begin{tabular}[t]{l}$= \sum_\rho f_{\rho\lambda}$\end{tabular}}}}%
    \put(0.38785371,0.32064088){\color[rgb]{0,0,0}\makebox(0,0)[lt]{\lineheight{1.25}\smash{\begin{tabular}[t]{l}$\implies \delta_{\mu\lambda} \dfrac{d_q}{\sqrt{d_{\lambda}}}$\end{tabular}}}}%
    \put(0,0){\includegraphics[width=\unitlength,page=1]{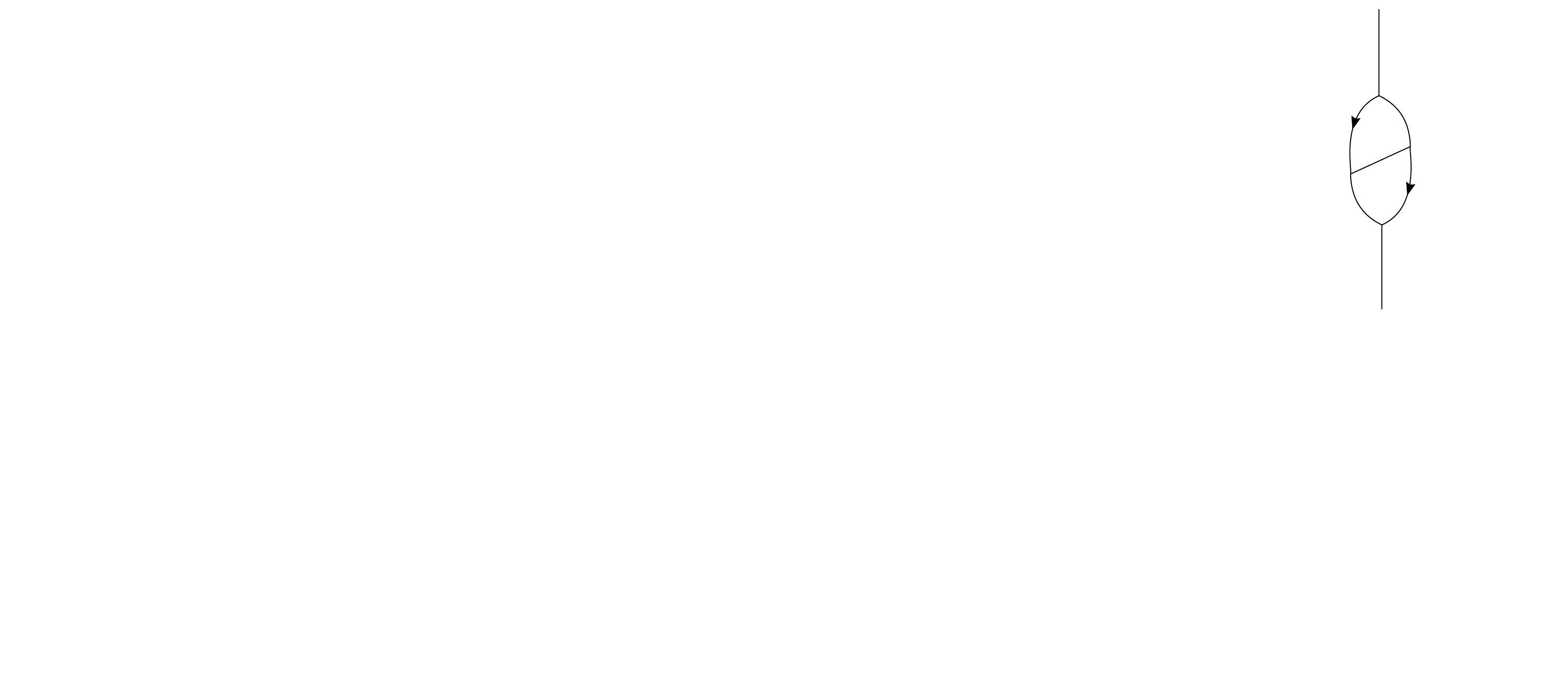}}%
    \put(0.65931676,0.32144437){\color[rgb]{0,0,0}\makebox(0,0)[lt]{\lineheight{1.25}\smash{\begin{tabular}[t]{l}$=\varkappa_q\sum_\rho f_{\rho\lambda}$\end{tabular}}}}%
    \put(0.83560011,0.34783263){\color[rgb]{0,0,0}\makebox(0,0)[lt]{\lineheight{1.25}\smash{\begin{tabular}[t]{l}$\mu$\end{tabular}}}}%
    \put(0,0){\includegraphics[width=\unitlength,page=2]{treecappf4.pdf}}%
    \put(0.57438008,0.34153164){\color[rgb]{0,0,0}\makebox(0,0)[lt]{\lineheight{1.25}\smash{\begin{tabular}[t]{l}$\lambda$\end{tabular}}}}%
    \put(0.90322888,0.2963211){\color[rgb]{0,0,0}\makebox(0,0)[lt]{\lineheight{1.25}\smash{\begin{tabular}[t]{l}$\rho$\end{tabular}}}}%
    \put(0.39382706,0.10134463){\color[rgb]{0,0,0}\makebox(0,0)[lt]{\lineheight{1.25}\smash{\begin{tabular}[t]{l}$\implies \varkappa_q d_q\delta_{\mu\lambda}$\end{tabular}}}}%
    \put(0,0){\includegraphics[width=\unitlength,page=3]{treecappf4.pdf}}%
    \put(0.81535294,0.12452639){\color[rgb]{0,0,0}\makebox(0,0)[lt]{\lineheight{1.25}\smash{\begin{tabular}[t]{l}$\mu$\end{tabular}}}}%
    \put(0.88298162,0.07150197){\color[rgb]{0,0,0}\makebox(0,0)[lt]{\lineheight{1.25}\smash{\begin{tabular}[t]{l}$\rho$\end{tabular}}}}%
    \put(0.67348678,0.09981232){\color[rgb]{0,0,0}\makebox(0,0)[lt]{\lineheight{1.25}\smash{\begin{tabular}[t]{l}$= \sum_\rho f_{\rho\lambda}$\end{tabular}}}}%
  \end{picture}%
\endgroup%
\end{equation*} 
\end{enumerate}
\vspace{-5mm}
\end{proof}

\vspace{1.5mm}
\noindent Note that plugging the adjoint of (iii) into (iv) yields $\sum_{\rho}f_{\rho\lambda}f^{*}_{\rho\mu}=\delta_{\lambda\mu}$, which agrees with the unitarity of $F^{qqq}_{q}$. 

\begin{remark}\textbf{(Duality)}\\
\label{dualityremk}
In the following, we wish to consider the action of rotation operator $\varphi$ on our canonical basis. Expanding some arbitrary $h\in\End(q^{\otimes2})$ in this basis, it is clear that $\varphi^{2}=\mathscr{C}$. When considering the image of an $x$-jack under $\varphi$, the directed edge calls for extra caution. For $x\neq\bm{1}$ we have
\begin{equation}\centering\def\svgwidth{5.5cm}
\begingroup%
  \makeatletter%
  \providecommand\color[2][]{%
    \errmessage{(Inkscape) Color is used for the text in Inkscape, but the package 'color.sty' is not loaded}%
    \renewcommand\color[2][]{}%
  }%
  \providecommand\transparent[1]{%
    \errmessage{(Inkscape) Transparency is used (non-zero) for the text in Inkscape, but the package 'transparent.sty' is not loaded}%
    \renewcommand\transparent[1]{}%
  }%
  \providecommand\rotatebox[2]{#2}%
  \newcommand*\fsize{\dimexpr\f@size pt\relax}%
  \newcommand*\lineheight[1]{\fontsize{\fsize}{#1\fsize}\selectfont}%
  \ifx\svgwidth\undefined%
    \setlength{\unitlength}{381.66534184bp}%
    \ifx\svgscale\undefined%
      \relax%
    \else%
      \setlength{\unitlength}{\unitlength * \real{\svgscale}}%
    \fi%
  \else%
    \setlength{\unitlength}{\svgwidth}%
  \fi%
  \global\let\svgwidth\undefined%
  \global\let\svgscale\undefined%
  \makeatother%
  \begin{picture}(1,0.31678549)%
    \lineheight{1}%
    \setlength\tabcolsep{0pt}%
    \put(0,0){\includegraphics[width=\unitlength,page=1]{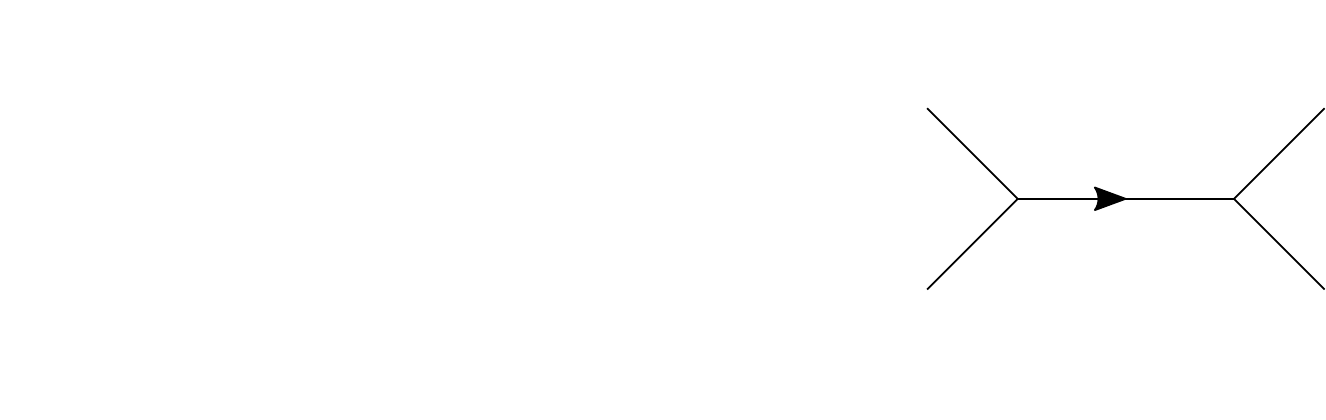}}%
    \put(0.79429845,0.19614429){\color[rgb]{0,0,0}\makebox(0,0)[lt]{\lineheight{1.25}\smash{\begin{tabular}[t]{l}$x$\end{tabular}}}}%
    \put(0,0){\includegraphics[width=\unitlength,page=2]{phijack.pdf}}%
    \put(0.31422367,0.15136987){\color[rgb]{0,0,0}\makebox(0,0)[lt]{\lineheight{1.25}\smash{\begin{tabular}[t]{l}$x$\end{tabular}}}}%
    \put(0,0){\includegraphics[width=\unitlength,page=3]{phijack.pdf}}%
    \put(0.55354982,0.14222931){\color[rgb]{0,0,0}\makebox(0,0)[lt]{\lineheight{1.25}\smash{\begin{tabular}[t]{l}=\end{tabular}}}}%
    \put(0.03100843,0.14635897){\color[rgb]{0,0,0}\makebox(0,0)[lt]{\lineheight{1.25}\smash{\begin{tabular}[t]{l}$\varphi$\end{tabular}}}}%
  \end{picture}%
\endgroup%
\label{phijack}\end{equation}
where we call the right-hand side a \textit{bone} morphism. Observing that the jack morphism may equivalently be represented with a slant,
\begin{subequations}\begin{eqnarray}
\begin{minipage}[t]{\linewidth}
          \raggedright
          \hspace{50mm}\adjustbox{valign=t}{\centering\def\svgwidth{5.5cm}
\begingroup%
  \makeatletter%
  \providecommand\color[2][]{%
    \errmessage{(Inkscape) Color is used for the text in Inkscape, but the package 'color.sty' is not loaded}%
    \renewcommand\color[2][]{}%
  }%
  \providecommand\transparent[1]{%
    \errmessage{(Inkscape) Transparency is used (non-zero) for the text in Inkscape, but the package 'transparent.sty' is not loaded}%
    \renewcommand\transparent[1]{}%
  }%
  \providecommand\rotatebox[2]{#2}%
  \newcommand*\fsize{\dimexpr\f@size pt\relax}%
  \newcommand*\lineheight[1]{\fontsize{\fsize}{#1\fsize}\selectfont}%
  \ifx\svgwidth\undefined%
    \setlength{\unitlength}{410.50722137bp}%
    \ifx\svgscale\undefined%
      \relax%
    \else%
      \setlength{\unitlength}{\unitlength * \real{\svgscale}}%
    \fi%
  \else%
    \setlength{\unitlength}{\svgwidth}%
  \fi%
  \global\let\svgwidth\undefined%
  \global\let\svgscale\undefined%
  \makeatother%
  \begin{picture}(1,0.29705954)%
    \lineheight{1}%
    \setlength\tabcolsep{0pt}%
    \put(0,0){\includegraphics[width=\unitlength,page=1]{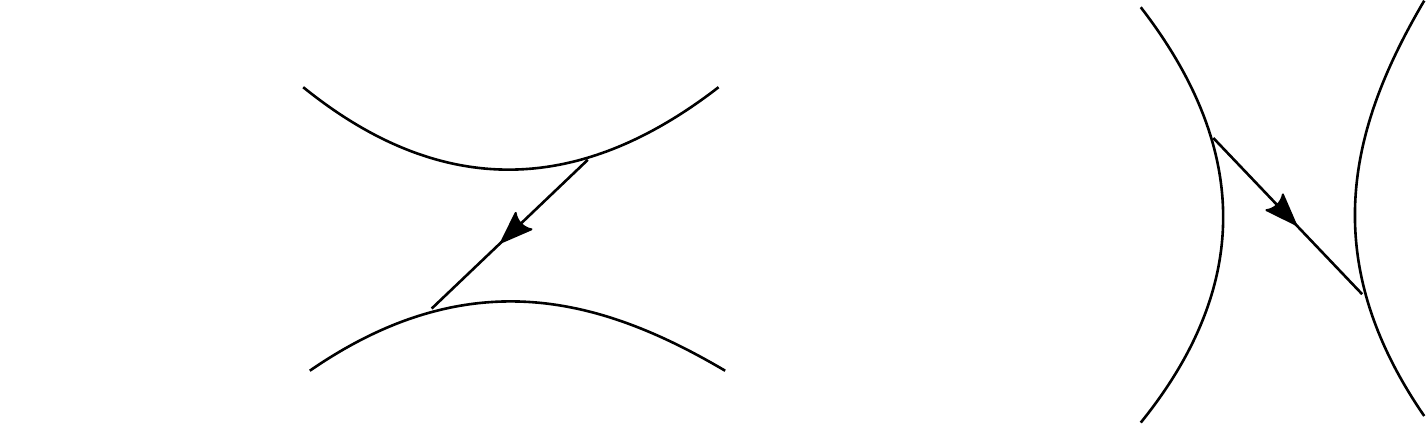}}%
    \put(0.37531375,0.11982219){\color[rgb]{0,0,0}\makebox(0,0)[lt]{\lineheight{1.25}\smash{\begin{tabular}[t]{l}$x$\end{tabular}}}}%
    \put(0.86909798,0.1913771){\color[rgb]{0,0,0}\makebox(0,0)[lt]{\lineheight{1.25}\smash{\begin{tabular}[t]{l}$x$\end{tabular}}}}%
    \put(0,0){\includegraphics[width=\unitlength,page=2]{slant1.pdf}}%
    \put(-0.00296872,0.14617281){\color[rgb]{0,0,0}\makebox(0,0)[lt]{\lineheight{1.25}\smash{\begin{tabular}[t]{l}$\varphi$\end{tabular}}}}%
    \put(0.69505898,0.14289468){\color[rgb]{0,0,0}\makebox(0,0)[lt]{\lineheight{1.25}\smash{\begin{tabular}[t]{l}=\end{tabular}}}}%
  \end{picture}%
\endgroup%
}\medskip
          \end{minipage} \label{slant1}\\
\begin{minipage}[t]{\linewidth}
          \raggedright
          \hspace{50mm}\adjustbox{valign=t}{\centering\def\svgwidth{5.5cm}
\begingroup%
  \makeatletter%
  \providecommand\color[2][]{%
    \errmessage{(Inkscape) Color is used for the text in Inkscape, but the package 'color.sty' is not loaded}%
    \renewcommand\color[2][]{}%
  }%
  \providecommand\transparent[1]{%
    \errmessage{(Inkscape) Transparency is used (non-zero) for the text in Inkscape, but the package 'transparent.sty' is not loaded}%
    \renewcommand\transparent[1]{}%
  }%
  \providecommand\rotatebox[2]{#2}%
  \newcommand*\fsize{\dimexpr\f@size pt\relax}%
  \newcommand*\lineheight[1]{\fontsize{\fsize}{#1\fsize}\selectfont}%
  \ifx\svgwidth\undefined%
    \setlength{\unitlength}{410.50717812bp}%
    \ifx\svgscale\undefined%
      \relax%
    \else%
      \setlength{\unitlength}{\unitlength * \real{\svgscale}}%
    \fi%
  \else%
    \setlength{\unitlength}{\svgwidth}%
  \fi%
  \global\let\svgwidth\undefined%
  \global\let\svgscale\undefined%
  \makeatother%
  \begin{picture}(1,0.29705957)%
    \lineheight{1}%
    \setlength\tabcolsep{0pt}%
    \put(0,0){\includegraphics[width=\unitlength,page=1]{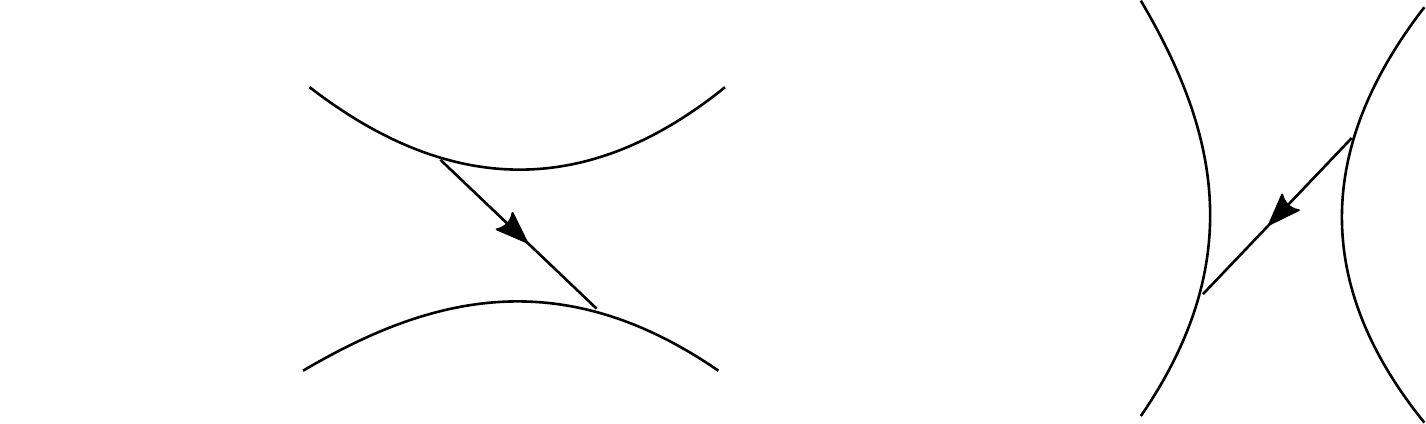}}%
    \put(0.3809643,0.12924027){\color[rgb]{0,0,0}\makebox(0,0)[lt]{\lineheight{1.25}\smash{\begin{tabular}[t]{l}$x$\end{tabular}}}}%
    \put(0.85968048,0.17630827){\color[rgb]{0,0,0}\makebox(0,0)[lt]{\lineheight{1.25}\smash{\begin{tabular}[t]{l}$x^*$\end{tabular}}}}%
    \put(0,0){\includegraphics[width=\unitlength,page=2]{slant2.pdf}}%
    \put(-0.00296874,0.14617292){\color[rgb]{0,0,0}\makebox(0,0)[lt]{\lineheight{1.25}\smash{\begin{tabular}[t]{l}$\varphi$\end{tabular}}}}%
    \put(0.69505882,0.14289479){\color[rgb]{0,0,0}\makebox(0,0)[lt]{\lineheight{1.25}\smash{\begin{tabular}[t]{l}=\end{tabular}}}}%
  \end{picture}%
\endgroup%
}\medskip
          \end{minipage} \label{slant2}
\end{eqnarray}\end{subequations}
For the $x$-bone to be well-defined, we must be able to identify (\ref{slant1}), (\ref{slant2}) and (\ref{phijack}). The adjunction of (\ref{slant1}) yields
\begin{equation*}\centering\hspace{6cm}\def\svgwidth{10cm}
\begingroup%
  \makeatletter%
  \providecommand\color[2][]{%
    \errmessage{(Inkscape) Color is used for the text in Inkscape, but the package 'color.sty' is not loaded}%
    \renewcommand\color[2][]{}%
  }%
  \providecommand\transparent[1]{%
    \errmessage{(Inkscape) Transparency is used (non-zero) for the text in Inkscape, but the package 'transparent.sty' is not loaded}%
    \renewcommand\transparent[1]{}%
  }%
  \providecommand\rotatebox[2]{#2}%
  \newcommand*\fsize{\dimexpr\f@size pt\relax}%
  \newcommand*\lineheight[1]{\fontsize{\fsize}{#1\fsize}\selectfont}%
  \ifx\svgwidth\undefined%
    \setlength{\unitlength}{671.15162767bp}%
    \ifx\svgscale\undefined%
      \relax%
    \else%
      \setlength{\unitlength}{\unitlength * \real{\svgscale}}%
    \fi%
  \else%
    \setlength{\unitlength}{\svgwidth}%
  \fi%
  \global\let\svgwidth\undefined%
  \global\let\svgscale\undefined%
  \makeatother%
  \begin{picture}(1,0.17406775)%
    \lineheight{1}%
    \setlength\tabcolsep{0pt}%
    \put(0,0){\includegraphics[width=\unitlength,page=1]{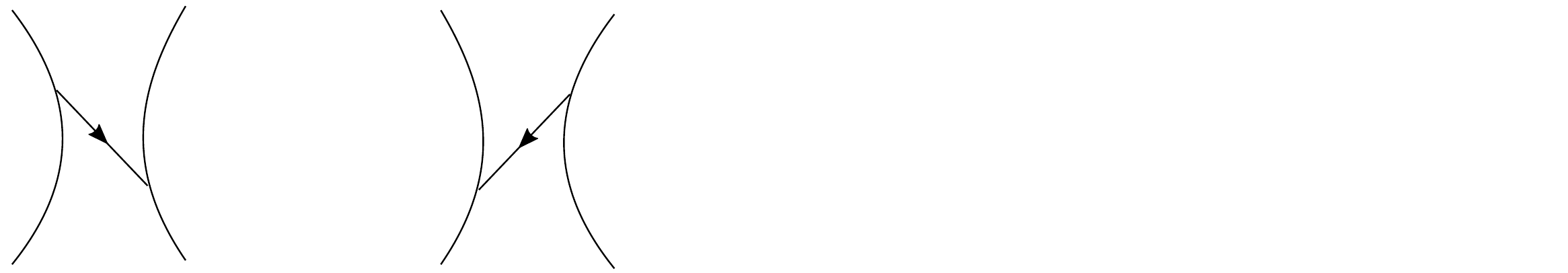}}%
    \put(0.17975366,0.08147346){\color[rgb]{0,0,0}\makebox(0,0)[lt]{\lineheight{1.25}\smash{\begin{tabular}[t]{l}$\overset{\dagger}{\widetilde{\longrightarrow}}$\\\end{tabular}}}}%
    \put(0.0566042,0.10304494){\color[rgb]{0,0,0}\makebox(0,0)[lt]{\lineheight{1.25}\smash{\begin{tabular}[t]{l}$x$\end{tabular}}}}%
    \put(0.32142938,0.10642502){\color[rgb]{0,0,0}\makebox(0,0)[lt]{\lineheight{1.25}\smash{\begin{tabular}[t]{l}$x$\end{tabular}}}}%
  \end{picture}%
\endgroup%
\end{equation*}
and since the bone is taken to be self-adjoint, we require that $x=x^{*}$. When considering $\varphi$ in the canonical basis we must therefore assume that $\mathscr{C}=\id$ i.e.\ $\{x_{i}\}_{i}$ are self-dual in (\ref{mainrule}). This obviates the need to direct any edges in our diagrams.
\end{remark}



\begin{theorem}\textbf{(Bones via jacks)}\\
\label{bojackthm}
Given fusion rule (\ref{mainrule}) with $x_{i}$ self-dual, we have
\begin{equation}
\Bone{\lambda} = \frac{\sqrt{d_\lambda}}{d_q}\Ccm + \varkappa_q \sum_i f_{i\lambda} \Jack{i} 
\label{boneviajacks}\end{equation}
\end{theorem}

\begin{proof}
Expanding the bone in the canonical basis, 
\begin{equation}\Bone{\lambda} = a^\lambda \Ccm +  \sum_i b_i^\lambda \Jack{i}
\label{canonbone}\end{equation}
Given a morphism $h\in\End(q^{\otimes2})$, let $h':=\id_{q}\otimes h \in\End(q^{\otimes3})$. Then we define the linear map $\Omega:h\mapsto h'\mapsto (\ev_{q}\otimes\id_{q})\circ h'$. Applying $\Omega$ to (\ref{canonbone}), we get\vspace{1mm}

\centering\def\svgwidth{8cm}
\begingroup%
  \makeatletter%
  \providecommand\color[2][]{%
    \errmessage{(Inkscape) Color is used for the text in Inkscape, but the package 'color.sty' is not loaded}%
    \renewcommand\color[2][]{}%
  }%
  \providecommand\transparent[1]{%
    \errmessage{(Inkscape) Transparency is used (non-zero) for the text in Inkscape, but the package 'transparent.sty' is not loaded}%
    \renewcommand\transparent[1]{}%
  }%
  \providecommand\rotatebox[2]{#2}%
  \newcommand*\fsize{\dimexpr\f@size pt\relax}%
  \newcommand*\lineheight[1]{\fontsize{\fsize}{#1\fsize}\selectfont}%
  \ifx\svgwidth\undefined%
    \setlength{\unitlength}{494.98585618bp}%
    \ifx\svgscale\undefined%
      \relax%
    \else%
      \setlength{\unitlength}{\unitlength * \real{\svgscale}}%
    \fi%
  \else%
    \setlength{\unitlength}{\svgwidth}%
  \fi%
  \global\let\svgwidth\undefined%
  \global\let\svgscale\undefined%
  \makeatother%
  \begin{picture}(1,0.35820263)%
    \lineheight{1}%
    \setlength\tabcolsep{0pt}%
    \put(0,0){\includegraphics[width=\unitlength,page=1]{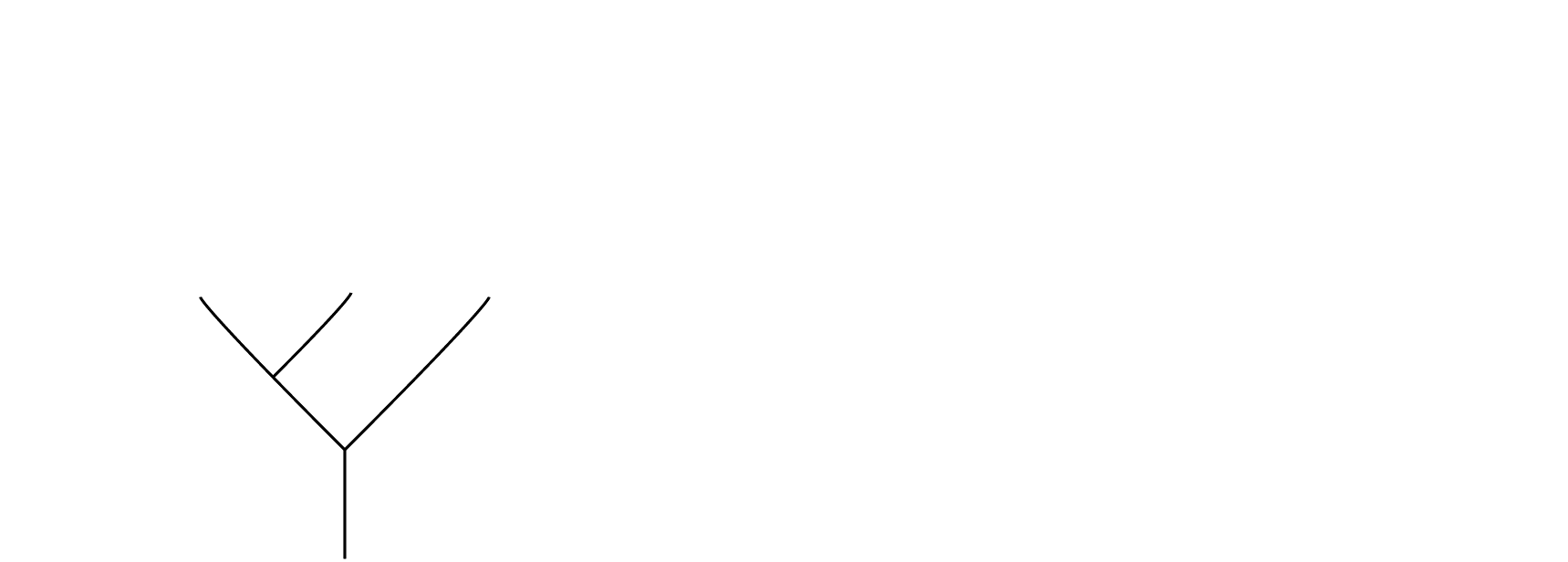}}%
    \put(0.5729797,0.06998244){\color[rgb]{0,0,0}\makebox(0,0)[lt]{\lineheight{1.25}\smash{\begin{tabular}[t]{l}$+\varkappa_q\sum_ib^\lambda_i$\end{tabular}}}}%
    \put(0.29584509,0.07072904){\color[rgb]{0,0,0}\makebox(0,0)[lt]{\lineheight{1.25}\smash{\begin{tabular}[t]{l}$= \varkappa_q a^\lambda$\end{tabular}}}}%
    \put(0.14693987,0.08001818){\color[rgb]{0,0,0}\makebox(0,0)[lt]{\lineheight{1.25}\smash{\begin{tabular}[t]{l}$\lambda$\end{tabular}}}}%
    \put(0,0){\includegraphics[width=\unitlength,page=2]{bojackpf1.pdf}}%
    \put(0.88533607,0.09557329){\color[rgb]{0,0,0}\makebox(0,0)[lt]{\lineheight{1.25}\smash{\begin{tabular}[t]{l}$i$\end{tabular}}}}%
    \put(0,0){\includegraphics[width=\unitlength,page=3]{bojackpf1.pdf}}%
    \put(-0.00163539,0.08715421){\color[rgb]{0,0,0}\makebox(0,0)[lt]{\lineheight{1.25}\smash{\begin{tabular}[t]{l}$\implies$\end{tabular}}}}%
    \put(0,0){\includegraphics[width=\unitlength,page=4]{bojackpf1.pdf}}%
    \put(0.2938594,0.28071904){\color[rgb]{0,0,0}\makebox(0,0)[lt]{\lineheight{1.25}\smash{\begin{tabular}[t]{l}$= a^\lambda$\end{tabular}}}}%
    \put(0.86695509,0.27958656){\color[rgb]{0,0,0}\makebox(0,0)[lt]{\lineheight{1.25}\smash{\begin{tabular}[t]{l}$i$\end{tabular}}}}%
    \put(0,0){\includegraphics[width=\unitlength,page=5]{bojackpf1.pdf}}%
    \put(0.58686029,0.27847086){\color[rgb]{0,0,0}\makebox(0,0)[lt]{\lineheight{1.25}\smash{\begin{tabular}[t]{l}$+\sum_ib^\lambda_i$\end{tabular}}}}%
    \put(0.17962664,0.28834722){\color[rgb]{0,0,0}\makebox(0,0)[lt]{\lineheight{1.25}\smash{\begin{tabular}[t]{l}$\lambda$\end{tabular}}}}%
  \end{picture}%
\endgroup%

\noindent \mbox{From (\ref{6jdef}), we see that $(a^{\lambda},b_{i}^{\lambda})=(\varkappa_{q}f_{0\lambda},\varkappa_{q}f_{i\lambda})$. The result follows from Lemma \ref{treecapping}(i).}
\end{proof}

\begin{corollary}
\label{bojackcor1} 
Let $D$ denote the matrix representation of a rotation operator $\varphi$ in the canonical basis. Then 
\begin{enumerate}[label=(\roman*)]
\item $D=\varkappa_{q}F^{qqq}_{q}$
\item $F^{qqq}_{q}$ is self-inverse
\item $f_{\lambda0}=f_{0\lambda}$ 
\item The parity of all entries in $\sigma(\varphi)$ cannot be the same
\end{enumerate}
\end{corollary}

\begin{proof}\hspace{2mm}
\begin{enumerate}[label=(\roman*)]
\item Follows directly from Theorem \ref{bojackthm}.
\item $D^{2}=\mathscr{C}$ where $\mathscr{C}=\id$ (since $\{x_{i}\}_{i}$ are self-dual), whence the result \mbox{follows by (i).}
\item For $\lambda=0$, note that (\ref{boneviajacks}) is (\ref{orientedid}) and so $f_{i0}=\varkappa_{q}\frac{\sqrt{d_{i}}}{d_{q}}$. The result follows from Lemma \ref{treecapping}(i). 
\item Since $\varphi$ is an involution for $\mathscr{C}=\id$, its spectrum can only consist of $\pm1$ s. Observe that $|\tr(F^{qqq}_{q})|<N$ since $|f_{ii}|\leq1$ and $|f_{00}|=\frac{1}{d_{q}}<1$.  By (i), $\tr(\varphi)=\varkappa_{q}\tr(F^{qqq}_{q})$ whence (iv) follows.
\end{enumerate}
\end{proof}

\noindent Corollary \ref{bojackcor1}(ii) can also be shown by applying linear map $\Omega':h\mapsto h''\mapsto (\id_{q}\otimes\ev_{q})\circ h''$ to (\ref{canonbone}), where $h\in\End(q^{\otimes2})$ and $h'':=h\otimes\id_{q}$.

\noindent Stated differently, Corollary \ref{bojackcor1}(iv) says that there are strictly less than $N$ linearly independent formal diagrams in $\End(q^{\otimes2})$ that are (anti)symmetric under rotation.




\begin{corollary}\label{popping}\textbf{(Bubble-popping)}
\begin{enumerate}[label=(\roman*)]
\item \begin{minipage}[t]{\linewidth}
          \raggedright
          \adjustbox{valign=t}{\centering\def\svgwidth{4cm}
\begingroup%
  \makeatletter%
  \providecommand\color[2][]{%
    \errmessage{(Inkscape) Color is used for the text in Inkscape, but the package 'color.sty' is not loaded}%
    \renewcommand\color[2][]{}%
  }%
  \providecommand\transparent[1]{%
    \errmessage{(Inkscape) Transparency is used (non-zero) for the text in Inkscape, but the package 'transparent.sty' is not loaded}%
    \renewcommand\transparent[1]{}%
  }%
  \providecommand\rotatebox[2]{#2}%
  \newcommand*\fsize{\dimexpr\f@size pt\relax}%
  \newcommand*\lineheight[1]{\fontsize{\fsize}{#1\fsize}\selectfont}%
  \ifx\svgwidth\undefined%
    \setlength{\unitlength}{121.11109696bp}%
    \ifx\svgscale\undefined%
      \relax%
    \else%
      \setlength{\unitlength}{\unitlength * \real{\svgscale}}%
    \fi%
  \else%
    \setlength{\unitlength}{\svgwidth}%
  \fi%
  \global\let\svgwidth\undefined%
  \global\let\svgscale\undefined%
  \makeatother%
  \begin{picture}(1,0.2484894)%
    \lineheight{1}%
    \setlength\tabcolsep{0pt}%
    \put(0,0){\includegraphics[width=\unitlength,page=1]{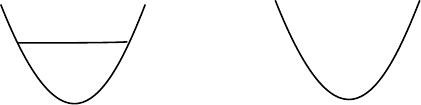}}%
    \put(0.35539956,0.0946726){\color[rgb]{0,0,0}\makebox(0,0)[lt]{\lineheight{1.25}\smash{\begin{tabular}[t]{l}$=\sqrt{d_i}$\end{tabular}}}}%
    \put(0.15184001,0.16724856){\color[rgb]{0,0,0}\makebox(0,0)[lt]{\lineheight{1.25}\smash{\begin{tabular}[t]{l}$i$\end{tabular}}}}%
  \end{picture}%
\endgroup%
}\medskip
          \end{minipage}
\item \begin{minipage}[t]{\linewidth}
          \raggedright
          \adjustbox{valign=t}{\centering\def\svgwidth{6cm}
\begingroup%
  \makeatletter%
  \providecommand\color[2][]{%
    \errmessage{(Inkscape) Color is used for the text in Inkscape, but the package 'color.sty' is not loaded}%
    \renewcommand\color[2][]{}%
  }%
  \providecommand\transparent[1]{%
    \errmessage{(Inkscape) Transparency is used (non-zero) for the text in Inkscape, but the package 'transparent.sty' is not loaded}%
    \renewcommand\transparent[1]{}%
  }%
  \providecommand\rotatebox[2]{#2}%
  \newcommand*\fsize{\dimexpr\f@size pt\relax}%
  \newcommand*\lineheight[1]{\fontsize{\fsize}{#1\fsize}\selectfont}%
  \ifx\svgwidth\undefined%
    \setlength{\unitlength}{471.74860436bp}%
    \ifx\svgscale\undefined%
      \relax%
    \else%
      \setlength{\unitlength}{\unitlength * \real{\svgscale}}%
    \fi%
  \else%
    \setlength{\unitlength}{\svgwidth}%
  \fi%
  \global\let\svgwidth\undefined%
  \global\let\svgscale\undefined%
  \makeatother%
  \begin{picture}(1,0.32777757)%
    \lineheight{1}%
    \setlength\tabcolsep{0pt}%
    \put(0,0){\includegraphics[width=\unitlength,page=1]{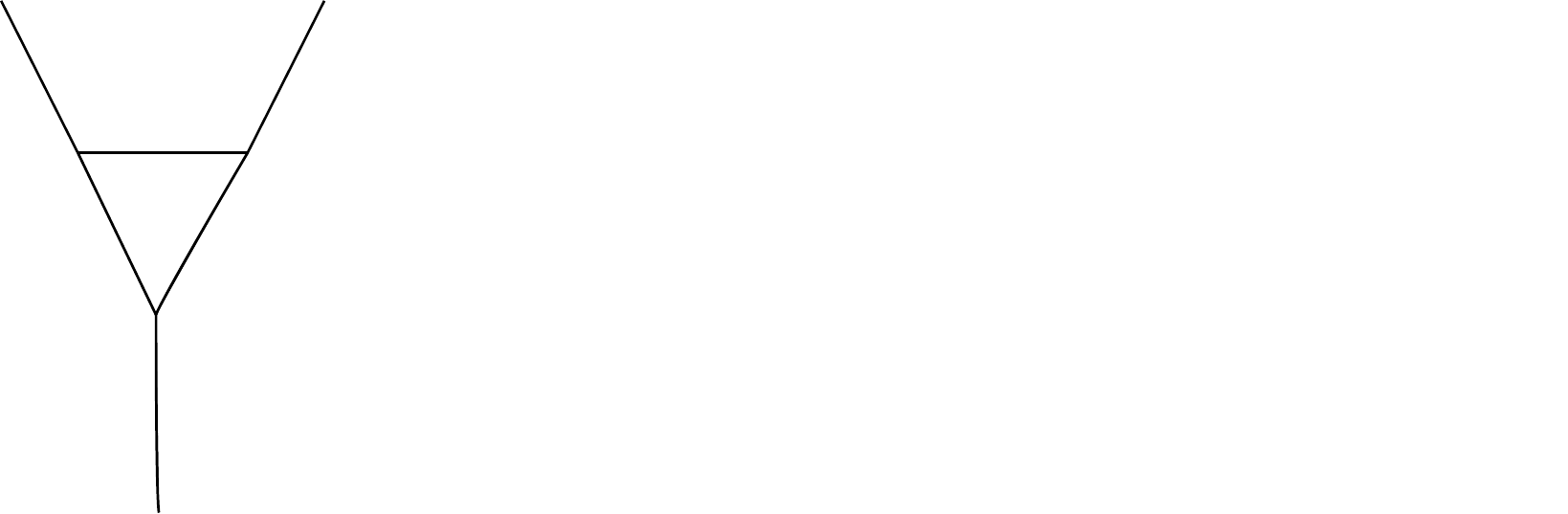}}%
    \put(0.09169285,0.25868468){\color[rgb]{0,0,0}\makebox(0,0)[lt]{\lineheight{1.25}\smash{\begin{tabular}[t]{l}$j$\end{tabular}}}}%
    \put(0.10558895,0.05845473){\color[rgb]{0,0,0}\makebox(0,0)[lt]{\lineheight{1.25}\smash{\begin{tabular}[t]{l}$i$\end{tabular}}}}%
    \put(0.23015321,0.17154673){\color[rgb]{0,0,0}\makebox(0,0)[lt]{\lineheight{1.25}\smash{\begin{tabular}[t]{l}$= \varkappa_q \dfrac{d_q}{\sqrt{d_i}}f_{ij}$\end{tabular}}}}%
    \put(0,0){\includegraphics[width=\unitlength,page=2]{martini.pdf}}%
    \put(0.66908648,0.05860289){\color[rgb]{0,0,0}\makebox(0,0)[lt]{\lineheight{1.25}\smash{\begin{tabular}[t]{l}$i$\end{tabular}}}}%
  \end{picture}%
\endgroup%
}\medskip
          \end{minipage}
\item \begin{minipage}[t]{\linewidth}
          \raggedright
          \adjustbox{valign=t}{\centering\def\svgwidth{12.5cm}
\begingroup%
  \makeatletter%
  \providecommand\color[2][]{%
    \errmessage{(Inkscape) Color is used for the text in Inkscape, but the package 'color.sty' is not loaded}%
    \renewcommand\color[2][]{}%
  }%
  \providecommand\transparent[1]{%
    \errmessage{(Inkscape) Transparency is used (non-zero) for the text in Inkscape, but the package 'transparent.sty' is not loaded}%
    \renewcommand\transparent[1]{}%
  }%
  \providecommand\rotatebox[2]{#2}%
  \newcommand*\fsize{\dimexpr\f@size pt\relax}%
  \newcommand*\lineheight[1]{\fontsize{\fsize}{#1\fsize}\selectfont}%
  \ifx\svgwidth\undefined%
    \setlength{\unitlength}{637.74419349bp}%
    \ifx\svgscale\undefined%
      \relax%
    \else%
      \setlength{\unitlength}{\unitlength * \real{\svgscale}}%
    \fi%
  \else%
    \setlength{\unitlength}{\svgwidth}%
  \fi%
  \global\let\svgwidth\undefined%
  \global\let\svgscale\undefined%
  \makeatother%
  \begin{picture}(1,0.18860967)%
    \lineheight{1}%
    \setlength\tabcolsep{0pt}%
    \put(0,0){\includegraphics[width=\unitlength,page=1]{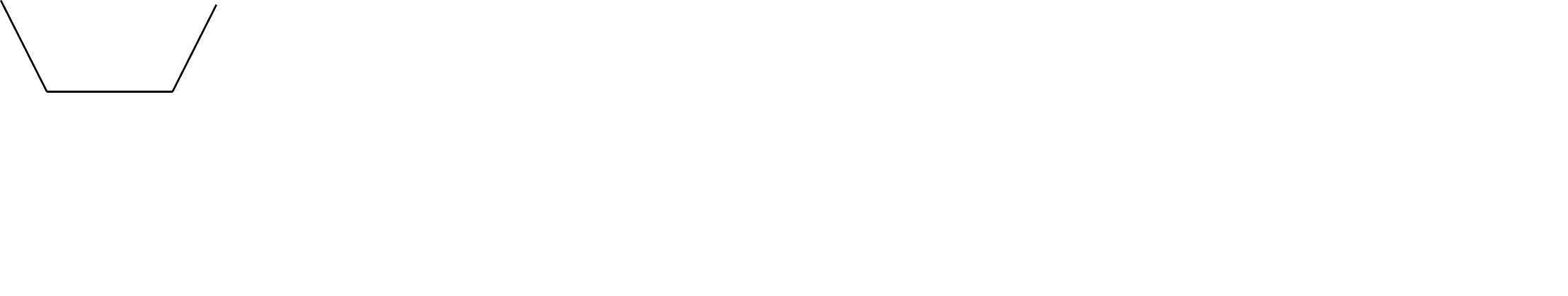}}%
    \put(0.0651979,0.02689461){\color[rgb]{0,0,0}\makebox(0,0)[lt]{\lineheight{1.25}\smash{\begin{tabular}[t]{l}$j$\end{tabular}}}}%
    \put(0.06280052,0.13880896){\color[rgb]{0,0,0}\makebox(0,0)[lt]{\lineheight{1.25}\smash{\begin{tabular}[t]{l}$i$\end{tabular}}}}%
    \put(0,0){\includegraphics[width=\unitlength,page=2]{boxpopper.pdf}}%
    \put(0.1317169,0.08903934){\color[rgb]{0,0,0}\makebox(0,0)[lt]{\lineheight{1.25}\smash{\begin{tabular}[t]{l}$= \dfrac{\sqrt{d_i d_j}}{d_q}$\end{tabular}}}}%
    \put(0,0){\includegraphics[width=\unitlength,page=3]{boxpopper.pdf}}%
    \put(0.41109746,0.08634165){\color[rgb]{0,0,0}\makebox(0,0)[lt]{\lineheight{1.25}\smash{\begin{tabular}[t]{l}$+ \sum_k\dfrac{d_q}{\sqrt{d_k}}f_{ki}f_{kj}$\end{tabular}}}}%
    \put(0,0){\includegraphics[width=\unitlength,page=4]{boxpopper.pdf}}%
    \put(0.67772905,0.08893517){\color[rgb]{0,0,0}\makebox(0,0)[lt]{\lineheight{1.25}\smash{\begin{tabular}[t]{l}$k$\end{tabular}}}}%
  \end{picture}%
\endgroup%
}\medskip
          \end{minipage}
\item \begin{minipage}[t]{\linewidth}
          \raggedright
          \adjustbox{valign=t}{\centering\def\svgwidth{7cm}
\begingroup%
  \makeatletter%
  \providecommand\color[2][]{%
    \errmessage{(Inkscape) Color is used for the text in Inkscape, but the package 'color.sty' is not loaded}%
    \renewcommand\color[2][]{}%
  }%
  \providecommand\transparent[1]{%
    \errmessage{(Inkscape) Transparency is used (non-zero) for the text in Inkscape, but the package 'transparent.sty' is not loaded}%
    \renewcommand\transparent[1]{}%
  }%
  \providecommand\rotatebox[2]{#2}%
  \newcommand*\fsize{\dimexpr\f@size pt\relax}%
  \newcommand*\lineheight[1]{\fontsize{\fsize}{#1\fsize}\selectfont}%
  \ifx\svgwidth\undefined%
    \setlength{\unitlength}{377.61251134bp}%
    \ifx\svgscale\undefined%
      \relax%
    \else%
      \setlength{\unitlength}{\unitlength * \real{\svgscale}}%
    \fi%
  \else%
    \setlength{\unitlength}{\svgwidth}%
  \fi%
  \global\let\svgwidth\undefined%
  \global\let\svgscale\undefined%
  \makeatother%
  \begin{picture}(1,0.32436929)%
    \lineheight{1}%
    \setlength\tabcolsep{0pt}%
    \put(0,0){\includegraphics[width=\unitlength,page=1]{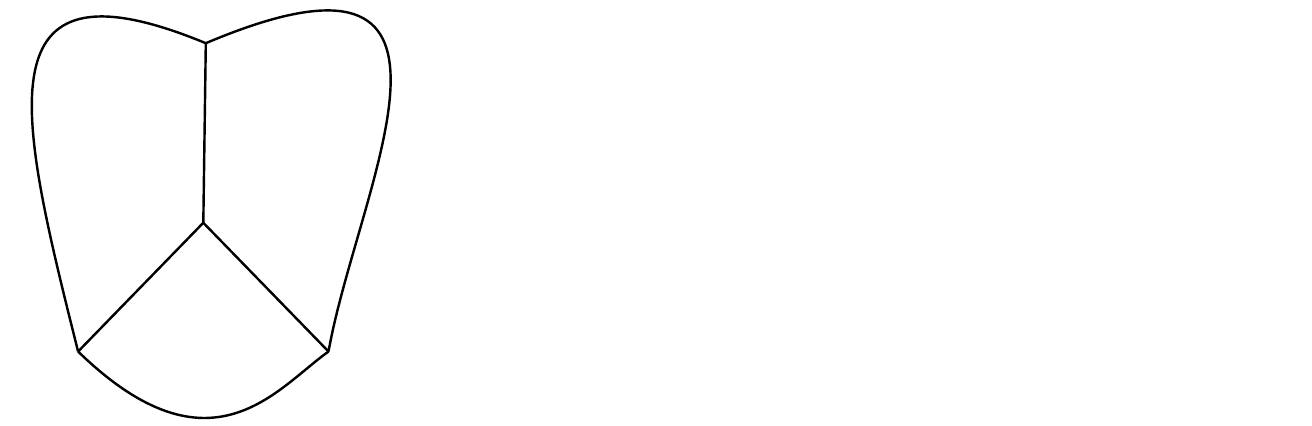}}%
    \put(0.16685652,0.19860017){\color[rgb]{0,0,0}\makebox(0,0)[lt]{\lineheight{1.25}\smash{\begin{tabular}[t]{l}$i$\end{tabular}}}}%
    \put(0.14100806,0.03320775){\color[rgb]{0,0,0}\makebox(0,0)[lt]{\lineheight{1.25}\smash{\begin{tabular}[t]{l}$j$\end{tabular}}}}%
    \put(0.31898612,0.16244091){\color[rgb]{0,0,0}\makebox(0,0)[lt]{\lineheight{1.25}\smash{\begin{tabular}[t]{l}$=\varkappa_q d_q^2 f_{ij}$\end{tabular}}}}%
  \end{picture}%
\endgroup%
}\medskip
          \end{minipage} 
\item \begin{minipage}[t]{\linewidth}
          \raggedright
          \adjustbox{valign=t}{\centering\def\svgwidth{12cm}
\begingroup%
  \makeatletter%
  \providecommand\color[2][]{%
    \errmessage{(Inkscape) Color is used for the text in Inkscape, but the package 'color.sty' is not loaded}%
    \renewcommand\color[2][]{}%
  }%
  \providecommand\transparent[1]{%
    \errmessage{(Inkscape) Transparency is used (non-zero) for the text in Inkscape, but the package 'transparent.sty' is not loaded}%
    \renewcommand\transparent[1]{}%
  }%
  \providecommand\rotatebox[2]{#2}%
  \newcommand*\fsize{\dimexpr\f@size pt\relax}%
  \newcommand*\lineheight[1]{\fontsize{\fsize}{#1\fsize}\selectfont}%
  \ifx\svgwidth\undefined%
    \setlength{\unitlength}{574.24723948bp}%
    \ifx\svgscale\undefined%
      \relax%
    \else%
      \setlength{\unitlength}{\unitlength * \real{\svgscale}}%
    \fi%
  \else%
    \setlength{\unitlength}{\svgwidth}%
  \fi%
  \global\let\svgwidth\undefined%
  \global\let\svgscale\undefined%
  \makeatother%
  \begin{picture}(1,0.11397035)%
    \lineheight{1}%
    \setlength\tabcolsep{0pt}%
    \put(0,0){\includegraphics[width=\unitlength,page=1]{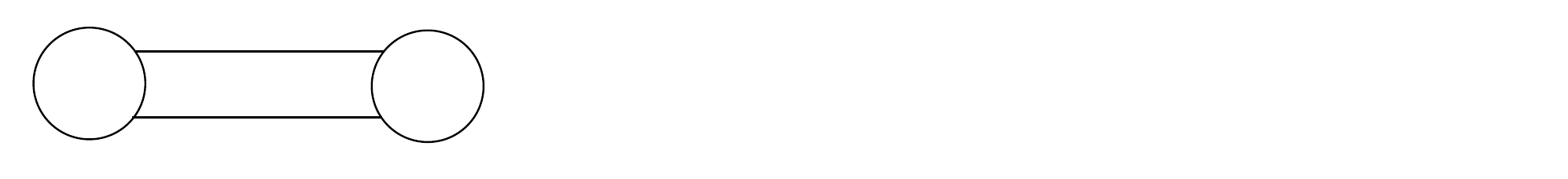}}%
    \put(0.15808611,0.09039981){\color[rgb]{0,0,0}\makebox(0,0)[lt]{\lineheight{1.25}\smash{\begin{tabular}[t]{l}$i$\end{tabular}}}}%
    \put(0.16069525,0.00678432){\color[rgb]{0,0,0}\makebox(0,0)[lt]{\lineheight{1.25}\smash{\begin{tabular}[t]{l}$j$\end{tabular}}}}%
    \put(0.32578227,0.05057694){\color[rgb]{0,0,0}\makebox(0,0)[lt]{\lineheight{1.25}\smash{\begin{tabular}[t]{l}$=\sqrt{d_i d_j}+d_q^2\sum_k f_{ki}f_{kj}$\end{tabular}}}}%
  \end{picture}%
\endgroup%
}\vspace{2mm}\newline{\footnotesize See also (\ref{bubbleprodsimplified}).}
          \end{minipage} 
\end{enumerate}
\end{corollary}
\begin{proof}\hspace{2mm}
\begin{enumerate}[label=(\roman*)]
\item Cup off the bottom of an $i$-bone and use (\ref{boneviajacks}) to get 
\begin{equation*}\hspace{7mm}\centering\def\svgwidth{6.5cm}
\begingroup%
  \makeatletter%
  \providecommand\color[2][]{%
    \errmessage{(Inkscape) Color is used for the text in Inkscape, but the package 'color.sty' is not loaded}%
    \renewcommand\color[2][]{}%
  }%
  \providecommand\transparent[1]{%
    \errmessage{(Inkscape) Transparency is used (non-zero) for the text in Inkscape, but the package 'transparent.sty' is not loaded}%
    \renewcommand\transparent[1]{}%
  }%
  \providecommand\rotatebox[2]{#2}%
  \newcommand*\fsize{\dimexpr\f@size pt\relax}%
  \newcommand*\lineheight[1]{\fontsize{\fsize}{#1\fsize}\selectfont}%
  \ifx\svgwidth\undefined%
    \setlength{\unitlength}{366.45074799bp}%
    \ifx\svgscale\undefined%
      \relax%
    \else%
      \setlength{\unitlength}{\unitlength * \real{\svgscale}}%
    \fi%
  \else%
    \setlength{\unitlength}{\svgwidth}%
  \fi%
  \global\let\svgwidth\undefined%
  \global\let\svgscale\undefined%
  \makeatother%
  \begin{picture}(1,0.29031049)%
    \lineheight{1}%
    \setlength\tabcolsep{0pt}%
    \put(0,0){\includegraphics[width=\unitlength,page=1]{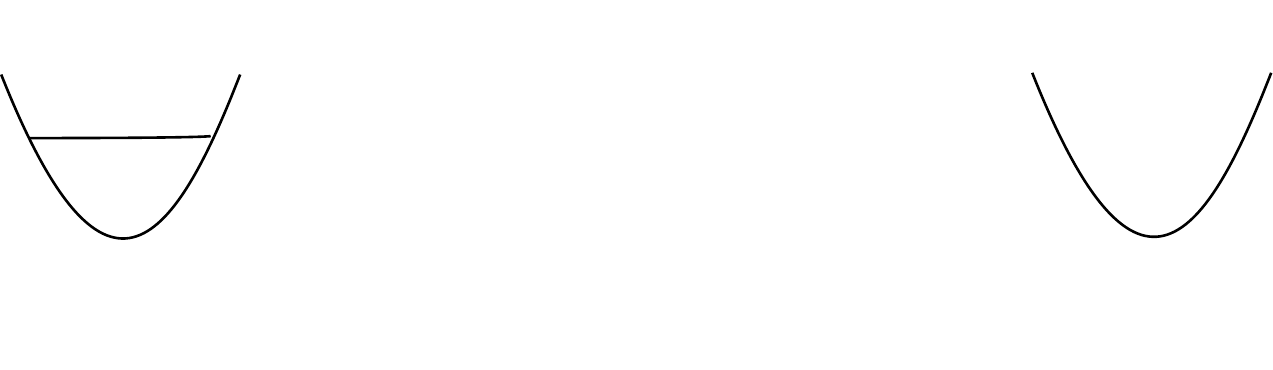}}%
    \put(0.63256843,0.15771575){\color[rgb]{0,0,0}\makebox(0,0)[lt]{\lineheight{1.25}\smash{\begin{tabular}[t]{l}$=\sqrt{d_i}$\end{tabular}}}}%
    \put(0.08354952,0.19322006){\color[rgb]{0,0,0}\makebox(0,0)[lt]{\lineheight{1.25}\smash{\begin{tabular}[t]{l}$i$\end{tabular}}}}%
    \put(0,0){\includegraphics[width=\unitlength,page=2]{whiskypf.pdf}}%
    \put(0.1960878,0.15741736){\color[rgb]{0,0,0}\makebox(0,0)[lt]{\lineheight{1.25}\smash{\begin{tabular}[t]{l}$=\dfrac{\sqrt{d_i}}{d_q}$\end{tabular}}}}%
  \end{picture}%
\endgroup%

\end{equation*}
Indeed, capping off both sides agrees with $\Theta(q,i,q)=\Phi(q,i,q)=d_{q}\sqrt{d_{i}}$.
\vspace{1.5mm}
\item Stacking a $j$-bone on top of an $i$-jack and using (\ref{boneviajacks}), we get 
\begin{equation}\hspace{15mm}\centering\def\svgwidth{9.5cm}
\begingroup%
  \makeatletter%
  \providecommand\color[2][]{%
    \errmessage{(Inkscape) Color is used for the text in Inkscape, but the package 'color.sty' is not loaded}%
    \renewcommand\color[2][]{}%
  }%
  \providecommand\transparent[1]{%
    \errmessage{(Inkscape) Transparency is used (non-zero) for the text in Inkscape, but the package 'transparent.sty' is not loaded}%
    \renewcommand\transparent[1]{}%
  }%
  \providecommand\rotatebox[2]{#2}%
  \newcommand*\fsize{\dimexpr\f@size pt\relax}%
  \newcommand*\lineheight[1]{\fontsize{\fsize}{#1\fsize}\selectfont}%
  \ifx\svgwidth\undefined%
    \setlength{\unitlength}{561.78584542bp}%
    \ifx\svgscale\undefined%
      \relax%
    \else%
      \setlength{\unitlength}{\unitlength * \real{\svgscale}}%
    \fi%
  \else%
    \setlength{\unitlength}{\svgwidth}%
  \fi%
  \global\let\svgwidth\undefined%
  \global\let\svgscale\undefined%
  \makeatother%
  \begin{picture}(1,0.22312985)%
    \lineheight{1}%
    \setlength\tabcolsep{0pt}%
    \put(0,0){\includegraphics[width=\unitlength,page=1]{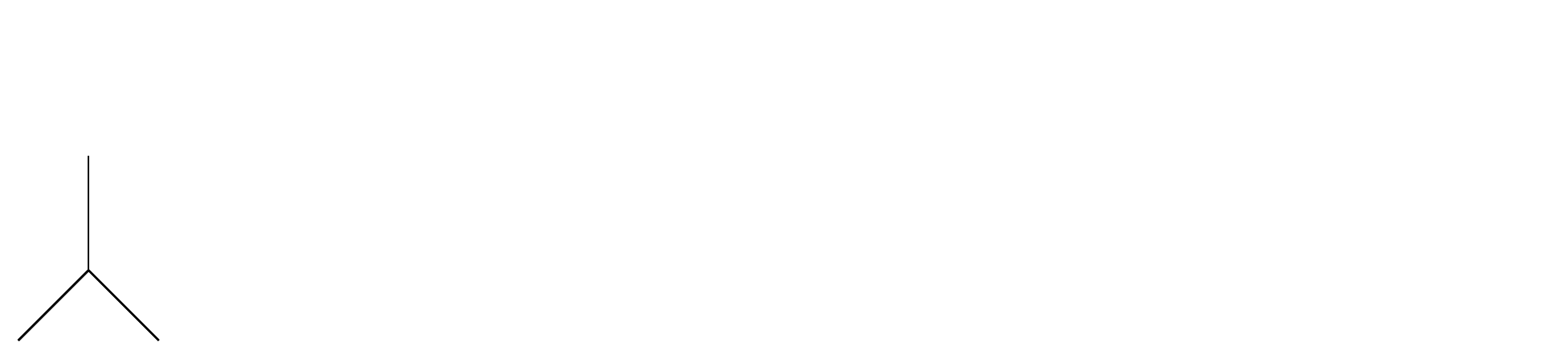}}%
    \put(0.0454185,0.19094053){\color[rgb]{0,0,0}\makebox(0,0)[lt]{\lineheight{1.25}\smash{\begin{tabular}[t]{l}$j$\end{tabular}}}}%
    \put(0,0){\includegraphics[width=\unitlength,page=2]{martinipf1.pdf}}%
    \put(0.11941285,0.08538575){\color[rgb]{0,0,0}\makebox(0,0)[lt]{\lineheight{1.25}\smash{\begin{tabular}[t]{l}$= \varkappa_q f_{ij}$\end{tabular}}}}%
    \put(0,0){\includegraphics[width=\unitlength,page=3]{martinipf1.pdf}}%
    \put(0.29833434,0.15833023){\color[rgb]{0,0,0}\makebox(0,0)[lt]{\lineheight{1.25}\smash{\begin{tabular}[t]{l}$i$\end{tabular}}}}%
    \put(0.29695802,0.04547287){\color[rgb]{0,0,0}\makebox(0,0)[lt]{\lineheight{1.25}\smash{\begin{tabular}[t]{l}$i$\end{tabular}}}}%
    \put(0.35353629,0.08578245){\color[rgb]{0,0,0}\makebox(0,0)[lt]{\lineheight{1.25}\smash{\begin{tabular}[t]{l}$= \varkappa_q \dfrac{d_q}{\sqrt{d_i}}f_{ij}$\end{tabular}}}}%
    \put(0,0){\includegraphics[width=\unitlength,page=4]{martinipf1.pdf}}%
    \put(0.5898086,0.10761377){\color[rgb]{0,0,0}\makebox(0,0)[lt]{\lineheight{1.25}\smash{\begin{tabular}[t]{l}$i$\end{tabular}}}}%
    \put(0.06567398,0.08582676){\color[rgb]{0,0,0}\makebox(0,0)[lt]{\lineheight{1.25}\smash{\begin{tabular}[t]{l}$i$\end{tabular}}}}%
  \end{picture}%
\endgroup%
\label{martinipf1}\end{equation}
\item Stack an $i$-bone on top of a $j$-bone and use  (\ref{boneviajacks}).
\vspace{1.5mm}
\item \begin{equation*}\hspace{10mm}\centering\def\svgwidth{14.5cm}
\begingroup%
  \makeatletter%
  \providecommand\color[2][]{%
    \errmessage{(Inkscape) Color is used for the text in Inkscape, but the package 'color.sty' is not loaded}%
    \renewcommand\color[2][]{}%
  }%
  \providecommand\transparent[1]{%
    \errmessage{(Inkscape) Transparency is used (non-zero) for the text in Inkscape, but the package 'transparent.sty' is not loaded}%
    \renewcommand\transparent[1]{}%
  }%
  \providecommand\rotatebox[2]{#2}%
  \newcommand*\fsize{\dimexpr\f@size pt\relax}%
  \newcommand*\lineheight[1]{\fontsize{\fsize}{#1\fsize}\selectfont}%
  \ifx\svgwidth\undefined%
    \setlength{\unitlength}{810.9445647bp}%
    \ifx\svgscale\undefined%
      \relax%
    \else%
      \setlength{\unitlength}{\unitlength * \real{\svgscale}}%
    \fi%
  \else%
    \setlength{\unitlength}{\svgwidth}%
  \fi%
  \global\let\svgwidth\undefined%
  \global\let\svgscale\undefined%
  \makeatother%
  \begin{picture}(1,0.15866314)%
    \lineheight{1}%
    \setlength\tabcolsep{0pt}%
    \put(0,0){\includegraphics[width=\unitlength,page=1]{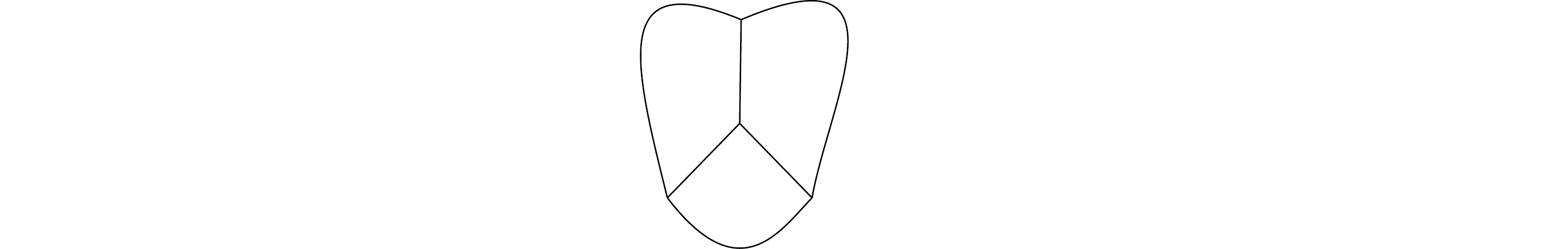}}%
    \put(0.47749188,0.10129545){\color[rgb]{0,0,0}\makebox(0,0)[lt]{\lineheight{1.25}\smash{\begin{tabular}[t]{l}$i$\end{tabular}}}}%
    \put(0.45414441,0.01505525){\color[rgb]{0,0,0}\makebox(0,0)[lt]{\lineheight{1.25}\smash{\begin{tabular}[t]{l}$j$\end{tabular}}}}%
    \put(0,0){\includegraphics[width=\unitlength,page=2]{pretzelpf1.pdf}}%
    \put(0.03076407,0.01001854){\color[rgb]{0,0,0}\makebox(0,0)[lt]{\lineheight{1.25}\smash{\begin{tabular}[t]{l}$j$\end{tabular}}}}%
    \put(0,0){\includegraphics[width=\unitlength,page=3]{pretzelpf1.pdf}}%
    \put(0.25814777,0.07597707){\color[rgb]{0,0,0}\makebox(0,0)[lt]{\lineheight{1.25}\smash{\begin{tabular}[t]{l}$i$\end{tabular}}}}%
    \put(0.07734043,0.0875182){\color[rgb]{0,0,0}\makebox(0,0)[lt]{\lineheight{1.25}\smash{\begin{tabular}[t]{l}$= \varkappa_q \dfrac{d_q}{\sqrt{d_i}}f_{ij}$\end{tabular}}}}%
    \put(0.30902906,0.08656498){\color[rgb]{0,0,0}\makebox(0,0)[lt]{\lineheight{1.25}\smash{\begin{tabular}[t]{l}$\implies$\end{tabular}}}}%
    \put(0.55215905,0.08656498){\color[rgb]{0,0,0}\makebox(0,0)[lt]{\lineheight{1.25}\smash{\begin{tabular}[t]{l}$= \varkappa_q \dfrac{d_q}{\sqrt{d_i}}f_{ij}$\end{tabular}}}}%
    \put(0,0){\includegraphics[width=\unitlength,page=4]{pretzelpf1.pdf}}%
    \put(0.80734203,0.07665125){\color[rgb]{0,0,0}\makebox(0,0)[lt]{\lineheight{1.25}\smash{\begin{tabular}[t]{l}$i$\end{tabular}}}}%
    \put(0.04464084,0.09054274){\color[rgb]{0,0,0}\makebox(0,0)[lt]{\lineheight{1.25}\smash{\begin{tabular}[t]{l}$i$\end{tabular}}}}%
  \end{picture}%
\endgroup%
\end{equation*}
whence the result follows from $\Phi(q,i,q)=d_{q}\sqrt{d_{i}}$.  Alternatively, this identity coincides with taking the quantum trace of Lemma \ref{treecapping}(iii) for $(\mu,\lambda)=(i,j)$.\\
\item Take the left and right partial traces of (iii) and plug in $\Phi(q,k,q)$.
\end{enumerate}
\end{proof}

\begin{corollary}\label{realsymm}
$F^{qqq}_{q}$ is real-symmetric. 
\end{corollary}

\begin{proof}
Corollary \ref{bojackcor1}(ii) tells us that $F^{qqq}_{q}$ is Hermitian. It thus suffices to show that $F^{qqq}_{q}$ is one of (a) \textit{real} or (b) \textit{symmetric}; nonetheless, we will show both explicitly. Applying the left and right partial traces to (\ref{martinipf1}), we obtain
\begin{equation}\hspace{30mm}\centering\def\svgwidth{6.5cm}
\begingroup%
  \makeatletter%
  \providecommand\color[2][]{%
    \errmessage{(Inkscape) Color is used for the text in Inkscape, but the package 'color.sty' is not loaded}%
    \renewcommand\color[2][]{}%
  }%
  \providecommand\transparent[1]{%
    \errmessage{(Inkscape) Transparency is used (non-zero) for the text in Inkscape, but the package 'transparent.sty' is not loaded}%
    \renewcommand\transparent[1]{}%
  }%
  \providecommand\rotatebox[2]{#2}%
  \newcommand*\fsize{\dimexpr\f@size pt\relax}%
  \newcommand*\lineheight[1]{\fontsize{\fsize}{#1\fsize}\selectfont}%
  \ifx\svgwidth\undefined%
    \setlength{\unitlength}{363.37972326bp}%
    \ifx\svgscale\undefined%
      \relax%
    \else%
      \setlength{\unitlength}{\unitlength * \real{\svgscale}}%
    \fi%
  \else%
    \setlength{\unitlength}{\svgwidth}%
  \fi%
  \global\let\svgwidth\undefined%
  \global\let\svgscale\undefined%
  \makeatother%
  \begin{picture}(1,0.35940635)%
    \lineheight{1}%
    \setlength\tabcolsep{0pt}%
    \put(0,0){\includegraphics[width=\unitlength,page=1]{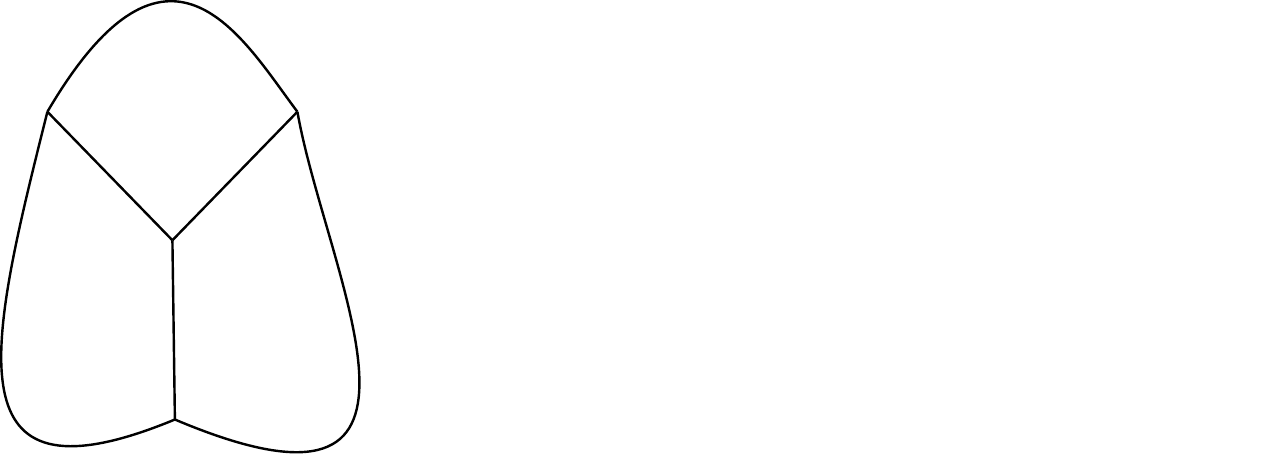}}%
    \put(0.14888293,0.09634437){\color[rgb]{0,0,0}\makebox(0,0)[lt]{\lineheight{1.25}\smash{\begin{tabular}[t]{l}$i$\end{tabular}}}}%
    \put(0.09935292,0.29322839){\color[rgb]{0,0,0}\makebox(0,0)[lt]{\lineheight{1.25}\smash{\begin{tabular}[t]{l}$j$\end{tabular}}}}%
    \put(0.30668663,0.1652335){\color[rgb]{0,0,0}\makebox(0,0)[lt]{\lineheight{1.25}\smash{\begin{tabular}[t]{l}$=\varkappa_q d_q^2 f_{ij}$\end{tabular}}}}%
  \end{picture}%
\endgroup%
\label{flippedpretzel}\end{equation}

\begin{enumerate}[label=(\alph*)]
\item Inverting the pretzel in Corollary \ref{popping}(iv) via adjunction and comparing the result to (\ref{flippedpretzel}), we see that $f_{ij}=f_{ij}^{*}$. We know that entries $f_{0\lambda}$ and $f_{\lambda0}$ are also real from Lemma \ref{treecapping}(i) and Corollary \ref{bojackcor1}(iii).
\item Note that (\ref{flippedpretzel}) can be deformed to the quantum trace of Lemma \ref{treecapping}(iii) for $(\mu,\lambda)=(j,i)$. Comparing scalars, we see that $f_{ij}=f_{ji}$. We also know that $f_{0\lambda}=f_{\lambda0}$ from Corollary \ref{bojackcor1}(iii).
\end{enumerate}
\end{proof}

\noindent In light of Corollary \ref{realsymm}, we may further simplify Corollary \ref{popping}(v) to
\begin{equation}\hspace{40mm}\centering\def\svgwidth{12cm}
\begingroup%
  \makeatletter%
  \providecommand\color[2][]{%
    \errmessage{(Inkscape) Color is used for the text in Inkscape, but the package 'color.sty' is not loaded}%
    \renewcommand\color[2][]{}%
  }%
  \providecommand\transparent[1]{%
    \errmessage{(Inkscape) Transparency is used (non-zero) for the text in Inkscape, but the package 'transparent.sty' is not loaded}%
    \renewcommand\transparent[1]{}%
  }%
  \providecommand\rotatebox[2]{#2}%
  \newcommand*\fsize{\dimexpr\f@size pt\relax}%
  \newcommand*\lineheight[1]{\fontsize{\fsize}{#1\fsize}\selectfont}%
  \ifx\svgwidth\undefined%
    \setlength{\unitlength}{525.44454259bp}%
    \ifx\svgscale\undefined%
      \relax%
    \else%
      \setlength{\unitlength}{\unitlength * \real{\svgscale}}%
    \fi%
  \else%
    \setlength{\unitlength}{\svgwidth}%
  \fi%
  \global\let\svgwidth\undefined%
  \global\let\svgscale\undefined%
  \makeatother%
  \begin{picture}(1,0.1331936)%
    \lineheight{1}%
    \setlength\tabcolsep{0pt}%
    \put(0,0){\includegraphics[width=\unitlength,page=1]{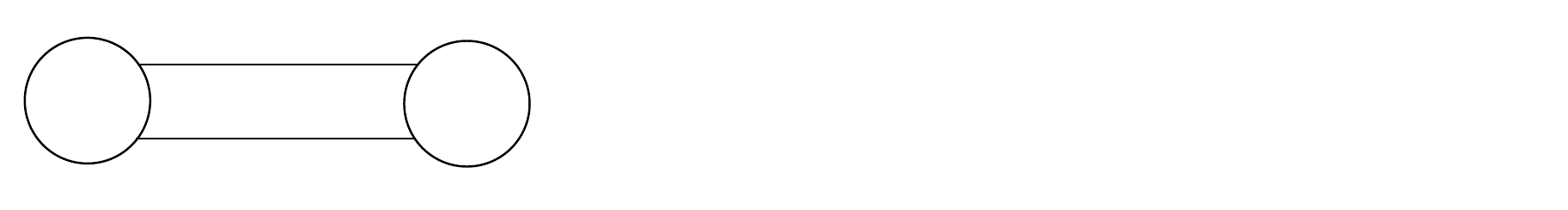}}%
    \put(0.1692182,0.10248466){\color[rgb]{0,0,0}\makebox(0,0)[lt]{\lineheight{1.25}\smash{\begin{tabular}[t]{l}$i$\end{tabular}}}}%
    \put(0.17214574,0.00866119){\color[rgb]{0,0,0}\makebox(0,0)[lt]{\lineheight{1.25}\smash{\begin{tabular}[t]{l}$j$\end{tabular}}}}%
    \put(0.35738341,0.05780143){\color[rgb]{0,0,0}\makebox(0,0)[lt]{\lineheight{1.25}\smash{\begin{tabular}[t]{l}$=\sqrt{d_i d_j}+\delta_{ij}d_q^2$\end{tabular}}}}%
  \end{picture}%
\endgroup%
\label{bubbleprodsimplified}\end{equation}

\subsection{Computing some $F$-symbols}
\label{compfsymb}
We now turn our attention to calculating $F^{qqq}_{q}$ for $q$ self-dual using the rotation operator. If $q\otimes q=\bm{1}$ then  $F^{qqq}_{q}=[f_{00}]=\left[\frac{\varkappa_{q}}{d_{q}}\right]$. In the case $q\otimes q=\bm{1}\oplus x$, we have
\begin{equation}F^{qqq}_{q}=\varkappa_{q}\begin{pmatrix}
\dfrac{1}{d_{q}} & \dfrac{\sqrt{d_{q}^{2}-1}}{d_{q}} \\[4mm]
\dfrac{\sqrt{d_{q}^{2}-1}}{d_{q}} & -\dfrac{1}{d_{q}}
\end{pmatrix}\label{quadf}\end{equation}
Since $x$ is necessarily self-dual, we may apply the corollaries of Theorem \ref{bojackthm}. Indeed, (\ref{quadf}) follows almost immediately from Lemma \ref{treecapping}(i) and Corollary \ref{bojackcor1}(iii); all that remains is to find $f_{xx}$. Applying Corollaries \ref{bojackcor1}(i) and (iv), we have $\tr\left(F^{qqq}_{q}\right)=\varkappa_{q}\tr(\varphi)=0$ whence $f_{xx}=-\frac{\varkappa_{q}}{d_{q}}$. \\ 

\noindent If we promote $\mathcal{C}$ to be ribbon, we may also determine $f_{xx}$ by combining the skein theory from Section \ref{quadcase} with Theorem \ref{bojackthm}. Resolving $\px$ as in (\ref{trotsky1}) and rotating, 
\begin{equation*}
\begin{split}
\varphi\left(\ \Px\ \right) & = \frac{1}{d_q}\alpha \Idm + \frac{\sqrt{d_x}}{d_q}\beta \Bone{x}\\
& = \frac{1}{d_q}\alpha\left(\frac{1}{d_q}\Ccm + \frac{\sqrt{d_x}}{d_q} \Jack{x}\right) + \frac{\sqrt{d_x}}{d_q}\beta \left(\frac{\sqrt{d_x}}{d_q}\Ccm + \varkappa_q f_{xx} \Jack{x} \right)\\
& = \frac{\alpha+\beta d_x}{d^2_q} \Ccm + \frac{\sqrt{d_x}}{d_q}\left( \frac{\alpha}{d_q}+\varkappa_q\beta f_{xx}\right)\Jack{x}
\end{split}
\end{equation*}
\noindent Comparing coefficients with the $\nx$-crossing, the cup-cap component corresponds to (\ref{floobendooben}), while the $x$-jack component yields $f_{xx}=\beta^{-2}-\frac{\varkappa_{q}}{d_{q}}\alpha\beta^{-1}$. Plugging in the values from (\ref{quadkeys}), we get $f_{xx}=-\frac{\varkappa_{q}}{d_{q}}$.

\begin{remark} Another approach to extracting information via the rotation operator is to stack a crossing on its image under $\varphi$ and then solve for the equation levied by Reidemeister-II. This approach is equivalent to the one taken above i.e.\ solving
\begin{equation}\hspace{20mm}\centering\def\svgwidth{5cm}
\begingroup%
  \makeatletter%
  \providecommand\color[2][]{%
    \errmessage{(Inkscape) Color is used for the text in Inkscape, but the package 'color.sty' is not loaded}%
    \renewcommand\color[2][]{}%
  }%
  \providecommand\transparent[1]{%
    \errmessage{(Inkscape) Transparency is used (non-zero) for the text in Inkscape, but the package 'transparent.sty' is not loaded}%
    \renewcommand\transparent[1]{}%
  }%
  \providecommand\rotatebox[2]{#2}%
  \newcommand*\fsize{\dimexpr\f@size pt\relax}%
  \newcommand*\lineheight[1]{\fontsize{\fsize}{#1\fsize}\selectfont}%
  \ifx\svgwidth\undefined%
    \setlength{\unitlength}{173.70824018bp}%
    \ifx\svgscale\undefined%
      \relax%
    \else%
      \setlength{\unitlength}{\unitlength * \real{\svgscale}}%
    \fi%
  \else%
    \setlength{\unitlength}{\svgwidth}%
  \fi%
  \global\let\svgwidth\undefined%
  \global\let\svgscale\undefined%
  \makeatother%
  \begin{picture}(1,0.50163524)%
    \lineheight{1}%
    \setlength\tabcolsep{0pt}%
    \put(0,0){\includegraphics[width=\unitlength,page=1]{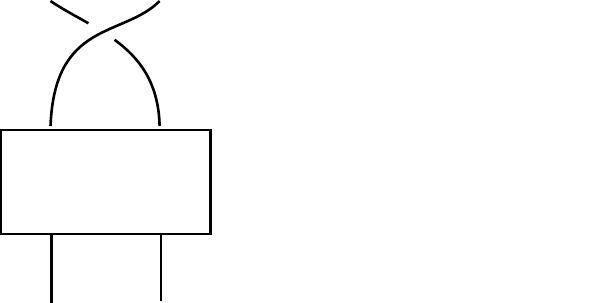}}%
    \put(0.03696519,0.18227525){\color[rgb]{0,0,0}\makebox(0,0)[lt]{\lineheight{1.25}\smash{\begin{tabular}[t]{l}$\varphi\left(\px\right)$\end{tabular}}}}%
    \put(0.39887946,0.27019099){\color[rgb]{0,0,0}\makebox(0,0)[lt]{\lineheight{1.25}\smash{\begin{tabular}[t]{l}$=\varkappa_q \Idm$\end{tabular}}}}%
  \end{picture}%
\endgroup%
\label{r2rotstack}\end{equation}
is clearly equivalent to solving $\varphi\left(\px\right)=\varkappa_{q} \ \nx$. Of course, this is solved with respect to some choice of basis $\mathcal{B}$. For the case $q\otimes q=\bm{1}\oplus x$, it is interesting to observe that fixing $\mathcal{B}=\left\{\idm,\ccm\right\}$ gave us information pertaining to $R^{qq}$, while fixing $\mathcal{B}$ canonical gave us information about $F^{qqq}_{q}$. Moreover, the information extracted via the latter basis relied on that found using the former basis.
\end{remark}

\vspace{3mm}
\noindent Now suppose that $q\otimes q=\bm{1}\oplus x\oplus y$ where $\mathscr{C}=\id$ and $\mathcal{C}$ is ribbon. Following (\ref{flageolet1}), 
\vspace{3.5mm}
\begin{align*}
\varphi\left(\ \Px\ \right)\overset{}&{=}A\Idm+B\Bone{x}+C\Ccm \\
\overset{\text{Thm.\ }\ref{bojackthm}}&{=}A\Idm+\left(\frac{\sqrt{d_{x}}}{d_{q}}B+C\right)\Ccm+\varkappa_{q}f_{xx}B\Jack{x}+\varkappa_{q}f_{yx}B\Jack{y} \\
\overset{(\ref{orientedid})}&{=}\left(A+\varkappa_{q}\frac{d_{q}}{\sqrt{d_{y}}}f_{yx}B\right)\Idm+\left(\frac{\sqrt{d_{x}}}{d_{q}}B+C-\frac{\varkappa_{q}}{\sqrt{d_{y}}}f_{yx}B\right)\Ccm\\
&\hspace{5mm}+\varkappa_{q}\left(f_{xx}B-\sqrt{\frac{d_{x}}{d_{y}}}f_{yx}B\right)\Jack{x}
\end{align*}
Solving $\varphi\left(\px\right)=\varkappa_{q}\nx$ with respect to basis $\left\{\idm,\ccm,\jack{x}\right\}$, we match coefficients with (\ref{flageolet2}) to obtain
\vspace{3.5mm}
\begin{subequations}\begin{align}
\Idm:& \quad A=\varkappa_{q}\left(C^{-1}-\frac{d_{q}}{\sqrt{d_{y}}}f_{yx}B\right) \label{gremlin1}\\
\Ccm:& \quad A'=\varkappa_{q}\left(C+\frac{\sqrt{d_{x}}}{d_{q}}B\right)-\frac{1}{\sqrt{d_{y}}}f_{yx}B \label{gremlin2}\\
\jack{x}:& \quad B'=B\left(f_{xx}-\sqrt{\frac{d_{x}}{d_{y}}}f_{yx}\right) \label{gremlin3} 
\end{align}\end{subequations}
\vspace{5mm}
\noindent Suppose $B,B'\neq0$. Recall from (\ref{thisguy}) that $B'=\pm B$. Also by Lemma \ref{treecapping}(i) \& (ii),
\begin{equation}-\varkappa_{q}\frac{\sqrt{d_{i}}}{d_{q}}=\sum_{j}\sqrt{d_{j}}f_{ji}\label{elpaso}\end{equation}
whence for $i=x$,
\begin{equation}f_{xx}=-\left(\frac{\varkappa_{q}}{d_{q}}+\sqrt{\frac{d_{y}}{d_{x}}}f_{yx}\right)\label{oggy}\end{equation}
\vspace{3mm}
\noindent Combining (\ref{gremlin3}) and (\ref{oggy}) eventually yields
\begin{equation}(f_{xx},f_{yx})=\left(\cfrac{\mp d_{x}}{(d_{q}\mp\varkappa_{q})d_{q}}\pm1 , \cfrac{\mp\sqrt{d_{x}d_{y}}}{(d_{q}\mp\varkappa_{q})d_{q} }\right)\quad , \quad B'=\pm B\end{equation}
Setting $i=y$ in (\ref{elpaso}) gives $f_{yy}=-(\frac{\varkappa_{q}}{d_{q}}+\sqrt{\frac{d_{x}}{d_{y}}}f_{xy})$. By Corollary \ref{realsymm}, we have $f_{xy}=f_{yx}$ whence
\begin{equation}f_{yy}=\cfrac{\mp d_{y}}{(d_{q}\mp\varkappa_{q})d_{q}}\pm1 \quad , \quad B'=\pm B\end{equation}

\vspace{5mm}
\begin{theorem}
\label{cubicfsymbols}
Let $\mathcal{C}$ be a unitary ribbon fusion category containing a fusion rule 
\begin{equation}q\otimes q=\bm{1}\oplus x\oplus y\label{whizzle}\end{equation} 
where $x,y,q\in\Irr(\mathcal{C})$. If $R^{qq}_{x}\neq R^{qq}_{y}$ then $R^{qq}_{x}R^{qq}_{y}=\pm1$ where
\vspace{2.5mm}
\begin{itemize}
\item If $R^{qq}_{x}R^{qq}_{y}=-1$ then $d_{q}=\varkappa_{q}\left(\cfrac{\vartheta_{q}-\vartheta_{q}^{-1}}{R^{qq}_{x}+R^{qq}_{y}}+1\right)$ and the associated skein relation is given by (a) the framed Dubrovnik polynomial (\ref{dubskein}) for $\varkappa_{q}=1$ and (b) (\ref{dubtwinskein}) for $\varkappa_{q}=-1$.
\vspace{2.5mm}
\item If $R^{qq}_{x}R^{qq}_{y}=1$ then $d_{q}=\varkappa_{q}\left(\cfrac{\vartheta_{q}+\vartheta_{q}^{-1}}{R^{qq}_{x}+R^{qq}_{y}}-1\right)$ and the associated skein relation is given by (c) the framed Kauffman polynomial (\ref{kauffskein}) for $\varkappa_{q}=1$ and (d) (\ref{kaufftwinskein}) for $\varkappa_{q}=-1$.
\end{itemize}
\vspace{2.5mm}
If $x$ and $y$ are self-dual then
\vspace{2.5mm}
\begin{enumerate}[label=(\roman*)]
\item For $R^{qq}_{x}R^{qq}_{y}=-1$, we have
\vspace{2mm}
\begin{equation}F^{qqq}_{q}=\varkappa_{q}\begin{pmatrix}
\cfrac{1}{d_{q}} & \cfrac{\sqrt{d_{x}}}{d_{q}} & \cfrac{\sqrt{d_{y}}}{d_{q}} \\[3ex]
\cfrac{\sqrt{d_{x}}}{d_{q}} & \cfrac{-d_{x}}{(\varkappa_{q}d_{q}-1)d_{q}}+1 & \cfrac{-\sqrt{d_{x}d_{y}}}{(\varkappa_{q}d_{q}-1)d_{q} } \\[4ex]
\cfrac{\sqrt{d_{y}}}{d_{q}} &  \cfrac{-\sqrt{d_{x}d_{y}}}{(\varkappa_{q}d_{q}-1)d_{q} } & \cfrac{-d_{y}}{(\varkappa_{q}d_{q}-1)d_{q}}+1
\end{pmatrix}\label{bazinga1}\end{equation}
\vspace{2.5mm}
\item For $R^{qq}_{x}R^{qq}_{y}=1$, we have
\vspace{2mm}
\begin{equation}F^{qqq}_{q}=\varkappa_{q}\begin{pmatrix}
\cfrac{1}{d_{q}} & \cfrac{\sqrt{d_{x}}}{d_{q}} & \cfrac{\sqrt{d_{y}}}{d_{q}} \\[3ex]
\cfrac{\sqrt{d_{x}}}{d_{q}} & \cfrac{d_{x}}{(\varkappa_{q}d_{q}+1)d_{q}}-1 & \cfrac{\sqrt{d_{x}d_{y}}}{(\varkappa_{q}d_{q}+1)d_{q} } \\[4ex]
\cfrac{\sqrt{d_{y}}}{d_{q}} &  \cfrac{\sqrt{d_{x}d_{y}}}{(\varkappa_{q}d_{q}+1)d_{q} } & \cfrac{d_{y}}{(\varkappa_{q}d_{q}+1)d_{q}}-1
\end{pmatrix}\label{bazinga2}\end{equation}
\end{enumerate}
\end{theorem}

\newpage

\begin{corollary}For a unitary ribbon fusion category $\mathcal{C}$ containing a fusion rule of the form (\ref{whizzle}) with $x$ and $y$ self-dual, we have
\begin{enumerate}[label=(\roman*)]
\item $q$ is symmetrically self-dual
\item $\sigma(\varphi)=\begin{cases}\{+1,+1,-1\} \ \ , \ \ \Lambda_{\mathcal{C},q} \text{ is the framed Dubrovnik polynomial}
\\ \{+1,-1,-1\} \ \ , \ \ \Lambda_{\mathcal{C},q} \text{ is the framed Kauffman polynomial}\end{cases}$
\end{enumerate}\label{noasd}\end{corollary}

\begin{proof}\hspace{2mm}
\begin{enumerate}[label=(\roman*)]
\item $\tr\left(F^{qqq}_{q}\right)=\pm3$ for $R^{qq}_{x}R^{qq}_{y}=\mp1$ when $\varkappa_{q}=-1$. The result follows by Corollaries \ref{bojackcor1} (i) and (iv).
\item By (i), $\Lambda_{\mathcal{C},q}$ is either the framed Dubrovnik or Kauffman polynomial. For the former, $\tr(F^{qqq}_{q})=1$ and for the latter $\tr(F^{qqq}_{q})=-1$. The result follows by Corollary \ref{bojackcor1} (i).  
\end{enumerate}
\end{proof}

\vspace{3mm}
\subsection{Some new bases}
\label{newbies}
As an application of the results thus far, we establish some new bases for $\End(q^{\otimes2})$ where 
\begin{equation}q\otimes q=\bm{1}\oplus x\oplus y\label{cusfubasis}\end{equation}
in a unitary ribbon fusion category $\mathcal{C}$ with $x$ and $y$ self-dual. First, let us make some observations for $q$ such that
\begin{equation}q\otimes q=\bm{1}\oplus\bigoplus_{i=1}^{k}x_{i}\end{equation}
with $q,x_{i}\in\Irr(\mathcal{C})$ and distinct self-dual objects $x_{i}$. We restrict our search to bases $\mathcal{B}$ satisfying the following property\footnote{Recall that the method employed for determining $\Lambda_{\C,q}$ in Section \ref{skeinrev} relied on expressing a crossing as a linear combination of morphisms that were invariant under the action of the rotation operator (up to permutation). This motivates the study of bases satisfying \textbf{(P1)}.}: 
\vspace{2mm}
\begin{center}\textbf{(P1) }\textit{The elements of $\mathcal{B}$ are permuted under the action of $\varphi$ (up to a sign for $1$-cycles).}\end{center}
\vspace{2mm}
\mbox{We will see that bases satisfying this property are closely related to the eigenbasis of $\varphi$.}\\
Clearly, the matrix representation of $\varphi$ in such a basis is a symmetric permutation matrix with $-1$'s permitted along the diagonal. The permutation consists of $2$-cycles and signed $1$-cycles. Then $+1$'s and $-1$'s along the matrix diagonal respectively correspond to `positive' and `negative' $1$-cycles. Let
\[N:=\dim\left(\End(q^{\otimes2})\right) \ ,\ n:=\#\{\text{cycles}\}\]
\[ b:=\#\{\text{positive $1$-cycles}\} \ , \ f:=\#\{\text{negative $1$-cycles}\} \] 
Assume $\mathcal{B}$ satisfying \textbf{(P1)} exists. Then write $\mathcal{B}=\{D_{ij}\}_{(i,j)}^{(n,l(i))}$ where $i$ indexes the $n$ cycles and $l(i)$ is the length of the $i^{th}$ cycle. We have
\begin{align}\varphi(D_{ij})=\begin{cases}
&\hspace{-2mm}\phantom{-}D_{i,\overline{j+1}} \ , \ \ \  l(i)=2 \\
&\phantom{-}D_{ij} \quad , \quad i \text{ indexes a positive $1$-cycle}\\
&-D_{ij} \quad , \quad i \text{ indexes a negative $1$-cycle}
\end{cases}\end{align}
where $\overline{j}$ denotes $j$ modulo $2$. Recall that $\sigma(\varphi)$ must consist of a mixture of $\pm1$'s. Let $V_{1}$ and $V_{-1}$ respectively denote the $+1$ and $-1$ eigenspaces of $\varphi$. Then
\begin{equation}\dim(V_{1})=n-f \ \ , \ \ \dim(V_{-1})=n-b\label{pheigenspaces}\end{equation}
whence 
\begin{equation}N=2n-b-f\text{ \ \ and \ \ }\left\lceil\frac{N}{2}\right\rceil\leq n\leq N\end{equation}
where the upper bound is realised when $\mathcal{B}$ is an eigenbasis for $\varphi$. 
\begin{example}
We can use the above to determine the possible actions (as a signed permutation) of $\varphi$ on $\mathcal{B}$ given $\sigma(\varphi)$. We will denote such an action by the signed cycle type $(a_{1},\ldots,a_{n})$ where $|a_{i}|=l(i)$ and $a_{i}=\pm1$ encodes the sign of a $1$-cycle. In the following, we exclude the instances where $n=N$ (i.e.\ eigenbases). 
\begin{enumerate}[label=(\roman*)]
\item $N=2$\ \ : \quad \quad \begin{tabular}{c|c}
$(n,b,f)$ & $(1,0,0)$ \\
\hline
$\sigma(\varphi)$ & $\{+1,-1\}$ \\
\hline
Cycle type & $(2)$
\end{tabular}

\vspace{2mm}
\item $N=3$\ \ : \quad \quad \begin{tabular}{c|cc}
$(n,b,f)$ & $(2,1,0)$ & $(2,0,1)$ \\
\hline
$\sigma(\varphi)$ & $\{+1,+1,-1\}$ & $\{+1,-1,-1\}$\\
\hline
Cycle type & $(2,1)$ & $(2,-1) $
\end{tabular}

\vspace{2mm}
\item $N=4$\ \ : ($1^{\text{st}}$ instance where there are two distinct cycle types for the same $\sigma(\varphi)$). \\

\begin{tabular}{c|cccc}
$(n,b,f)$ & $(3,1,1)$ & $(3,2,0)$ & (3,0,2) & (2,0,0) \\
\hline
$\sigma(\varphi)$ & $\{+1+1,-1,-1\}$ & $\{+1,+1,+1,-1\}$ & $\{+1,-1,-1,-1\}$ & $\{+1,+1,-1,-1\}$\\
\hline
Cycle type & $(2,1,-1)$ & $(2,1,1)$ & $(2,-1,-1)$ & $(2,2)$
\end{tabular}
\end{enumerate}
etc.
\label{tabspecsperms}
\end{example}

\vspace{3mm}
\begin{itemize}
\item We already encountered basis $\left\{\idm,\ccm\right\}$ corresponding to Example \ref{tabspecsperms}(i).\\
\item Let $\C$ be a unitary $G_2$ ribbon category and take $q\in\Irr(\C)$ with fusion rule $q\otimes q=\bm{1}\oplus x \oplus y \oplus q$. There exists a basis $\left\{\ccm, \idm, \jack{}, \bone{}\right\}$ on $\End(q^{\otimes2}$ \cite{Kup2}. This is basis of cycle type $(2,2)$, whence we see from Example \ref{tabspecsperms}(iii) and Corollary \ref{bojackcor1}(i) that $F^{qqq}_{q}$ is traceless.  \\
\item We show by construction that there exist bases corresponding to \mbox{Example \ref{tabspecsperms}(ii).}
\end{itemize}

\vspace{3mm}
\noindent We define 
\begin{equation}\jj{X}:=\Jack{X}+\Bone{X}\ \ \quad \text{and}\ \ \quad \jjd{X}:=\Bone{X}-\Jack{X}\end{equation}
Observe that for (\ref{cusfubasis}),
\begin{subequations}\begin{align}
\Idm+\Ccm&=\cfrac{\sqrt{d_{x}}}{d_{q}-1}\ \jj{x}+\cfrac{\sqrt{d_{y}}}{d_{q}-1}\ \jj{y} \\
\Idm-\Ccm&=-\cfrac{\sqrt{d_{x}}}{d_{q}+1}\ \jjd{x}-\cfrac{\sqrt{d_{y}}}{d_{q}+1}\ \jjd{y} 
\end{align}\end{subequations}
whence
\begin{subequations}\begin{align}
\sqrt{d_{x}}\ \jj{x}+\sqrt{d_{y}}\ \jjd{y}=(d_{q}-1)\left(\ \Idm+\Ccm\ \right)-2\sqrt{d_{y}}\ \Jack{y} \\
\sqrt{d_{x}}\ \jj{x}-\sqrt{d_{y}}\ \jjd{y}=(d_{q}+1)\left(\ \Idm-\Ccm\ \right)+2\sqrt{d_{x}}\ \Bone{x} 
\end{align}\end{subequations}

\noindent Recall from Corollary \ref{noasd}(i) that $\varkappa_{q}=1$. Expanding in the canonical basis,
\begin{subequations}\begin{align}
\sqrt{d_{x}}\ \jj{x}+\sqrt{d_{y}}\ \jjd{y} &= \left(d_{q}-\cfrac{1}{d_{q}}\right)\ \Ccm + \sqrt{d_{x}}\left(1-\cfrac{1}{d_{q}}\right) \Jack{x} -\sqrt{d_{y}}\left(1+\cfrac{1}{d_{q}}\right)\Jack{y} \\
\sqrt{d_{x}}\ \jj{x}-\sqrt{d_{y}}\ \jjd{y}&= \left(\cfrac{1}{d_{q}}-d_{q}+2\cfrac{d_{x}}{d_{q}}\right) \Ccm +\sqrt{d_{x}}\left(1+\cfrac{1}{d_{q}}+2f_{xx}\right) \Jack{x} \label{leiwulong}\\
&\hspace{5mm}+\left[\sqrt{d_{y}}\left(1+\cfrac{1}{d_{q}}\right)+2\sqrt{d_{x}}f_{yx}\right] \Jack{y} \nonumber
\end{align}\end{subequations}
where in (\ref{leiwulong}) we used (\ref{boneviajacks}). Let
\begin{equation}\jj{xy}^{+}:=\sqrt{d_{x}}\ \jj{x}+\sqrt{d_{y}}\ \jjd{y} \ \quad \ \ , \ \ \quad \jj{xy}^{-}:=\sqrt{d_{x}}\ \jj{x}-\sqrt{d_{y}}\ \jjd{y}  \end{equation}

\vspace{3mm}
\begin{lemma}
$\jj{xy}^{+}$ and $\jj{xy}^{-}$ are linearly independent.
\label{rocksfoe}\end{lemma}
\renewcommand\qedsymbol{\mbox{\larger\Lightning}} 
\begin{proof}
$\varphi(\jj{xy}^{+})=\jj{xy}^{-}$. Suppose $\jj{xy}^{-}=z\jj{xy}^{+}$ for some $z\in\mathbb{C}$. Then $\varphi(\jj{xy}^{+})=z\jj{xy}^{+}$ whence $\jj{xy}^{+}=\pm\jj{xy}^{-}$. For $z=+1$ and $z=-1$ we respectively get $\jjd{y}=0$ and $\jj{x}=0$, both of which yield a contradiction.
\end{proof} 
\renewcommand\qedsymbol{$\square$} 

\vspace{3mm}
\begin{theorem}
\label{sleepyhead}
Let $q$ be defined as in (\ref{cusfubasis}). Then
\begin{enumerate}[label=(\roman*)]
\item $\left\{\jj{xy}^{+}\ ,\ \jj{xy}^{-}\ ,\ \idm+\ccm\right\}$ defines a basis for $\End(q^{\otimes2})$ when $\Lambda_{\mathcal{C},q}$ is the framed Dubrovnik polynomial.
\item $\left\{\jj{xy}^{+}\ ,\ \jj{xy}^{-}\ ,\ \idm-\ccm\right\}$ defines a basis for $\End(q^{\otimes2})$ when $\Lambda_{\mathcal{C},q}$ is the framed Kauffman polynomial.
\end{enumerate}
\end{theorem}
\noindent Note that the bases in the above theorem satisfy \textbf{(P1)} (see Example \ref{tabspecsperms}(ii)), and that we can permute labels $x$ and $y$ by the symmetry of our construction.

\renewcommand\qedsymbol{\mbox{\larger\Lightning}} 
\begin{proof}
By Lemma \ref{rocksfoe}, it suffices in each case to show that the final basis element is not a linear combination of the first two. For $c_{1},c_{2}\in\mathbb{C}$,
\begin{equation*}c_{1}\jj{xy}^{+}+c_{2}\jj{xy}^{-}=a_{1}\Ccm+a_{2}\Jack{x}+a_{3}\Jack{y}\end{equation*}
where 
\[a_{1}:=\cfrac{1}{d_{q}}\left[(1-d_{q}^{2})(c_{2}-c_{1})+2c_{2}d_{x}\right] \ \ , \ \ a_{2}:=\sqrt{d_{x}}\left[\left(1+\cfrac{1}{d_{q}}\right)c_{2}+\left(1-\cfrac{1}{d_{q}}\right)c_{1}+2c_{2}f_{xx}\right] \]
\[a_{3}:=\sqrt{d_{y}}\left[\left(1+\cfrac{1}{d_{q}}\right)(c_{2}-c_{1})+2c_{2}\sqrt{\cfrac{d_{x}}{d_{y}}}f_{yx}\right] \]
Recall from (\ref{gremlin3}) that 
\begin{equation}\sqrt{\cfrac{d_{x}}{d_{y}}}f_{yx}=f_{xx}-r\end{equation}
where $r=1,-1$ when $\Lambda_{\mathcal{C},q}$ is the framed Dubrovnik and framed Kauffman polynomial respectively. Thus,
\begin{equation}a_{3}=\sqrt{\cfrac{d_{y}}{d_{x}}}\ a_{2}-2\sqrt{d_{y}}(c_{1}+rc_{2})\label{sleepergreens}\end{equation}
\begin{enumerate}[label=(\roman*)]
\vspace{5mm}
\item Suppose there exist $c_{1}$ and $c_{2}$ such that $c_{1}\jj{xy}^{+}+c_{2}\jj{xy}^{-}=\idm+\ccm$. Then $(a_{1},a_{2},a_{3})=\left(1+\cfrac{1}{d_{q}},\cfrac{\sqrt{d_{x}}}{d_{q}},\cfrac{\sqrt{d_{y}}}{d_{q}}\right)$. Setting $r=1$, (\ref{sleepergreens}) gives $c_{1}=-c_{2}$. Now comparing values for $a_{1}$ yields
\begin{equation*}c_{1}=\cfrac{1}{2}\left(\cfrac{1+d_{q}}{d_{y}}\right) \quad , \quad c_{2}=-\cfrac{1}{2}\left(\cfrac{1+d_{q}}{d_{y}}\right)\end{equation*}
whence comparing values for $a_{2}$ yields
\begin{equation*}f_{xx}=-\cfrac{d_{y}}{d_{q}(1+d_{q})}-\cfrac{1}{d_{q}}\end{equation*}
where we can manipulate the right-hand side to get
\begin{align*}-\cfrac{d_{y}}{d_{q}(1+d_{q})}-\cfrac{1}{d_{q}}&=-1+\cfrac{d_{x}}{d_{q}(d_{q}-1)}-\cfrac{2d_{x}}{d_{q}(d_{q}^{2}-1)}\\
\overset{(\ref{bazinga1})}&{=}-f_{xx}-\cfrac{2d_{x}}{d_{q}(d_{q}^{2}-1)} \end{align*}
implying that $f_{xx}=-\cfrac{d_{x}}{d_{q}(d_{q}^{2}-1)}$\ . Rearranging the expression for $f_{xx}$ from (\ref{bazinga1}),
\begin{equation*}f_{xx}=1-\cfrac{d_{x}}{(d_{q}-1)d_{q}}=\cfrac{d_{q}(d_{q}^{2}-1)-d_{x}(d_{q}+1)}{d_{q}(d_{q}^{2}-1)}=\cfrac{d_{q}d_{y}-d_{x}}{d_{q}(d_{q}^{2}-1)}\end{equation*}
whence we arrive at a contradiction since $d_{q},d_{y}\neq0$. \pushQED{\qed}\qedhere\popQED

\vspace{7.5mm}

\item Suppose there exist $c_{1}$ and $c_{2}$ such that $c_{1}\jj{xy}^{+}+c_{2}\jj{xy}^{-}=\idm-\ccm$. Then $(a_{1},a_{2},a_{3})=\left(\cfrac{1}{d_{q}}-1,\cfrac{\sqrt{d_{x}}}{d_{q}},\cfrac{\sqrt{d_{y}}}{d_{q}}\right)$. Setting $r=-1$, (\ref{sleepergreens}) gives $c_{1}=c_{2}$. Now comparing values for $a_{1}$ yields
\begin{equation*}c_{1}=c_{2}=\cfrac{1}{2}\left(\cfrac{1-d_{q}}{d_{x}}\right)\end{equation*}
whence comparing values for $a_{3}$ yields
\begin{equation*}f_{yx}=\cfrac{\sqrt{d_{x}d_{y}}}{(d_{q}-1)d_{q}}\end{equation*}
From (\ref{bazinga2}), $f_{yx}=\cfrac{\sqrt{d_{x}d_{y}}}{(d_{q}+1)d_{q}}$ whence we arrive at a contradiction since $d_{q}\neq0$.
\pushQED{\qed}\qedhere\popQED
\end{enumerate}
\renewcommand\qedsymbol{\phantom{$\square$}} 
\end{proof}
\renewcommand\qedsymbol{$\square$} 

\newpage
\begin{corollary}\textbf{(Diagonalising $\varphi$)} \\
Following the notation from (\ref{pheigenspaces}), we have $\End(q^{\otimes2})=V_{1}\oplus V_{-1}$.
\begin{enumerate}[label=(\roman*)]
\item If $\Lambda_{\mathcal{C},q}$ is the framed Dubrovnik polynomial then 
\begin{equation}V_{1}=\spn\left\{\jj{x}\ ,\ \idm+\ccm\right\} \quad , \quad V_{-1}=\spn\left\{\jjd{y}\right\}\label{boisy1}\end{equation}
\item If $\Lambda_{\mathcal{C},q}$ is the framed Kauffman polynomial then 
\begin{equation}V_{1}=\spn\left\{\jj{x}\right\} \quad , \quad V_{-1}=\spn\left\{\jjd{y}\ ,\ \idm-\ccm\right\}\label{boisy2}\end{equation}
\end{enumerate} 
\end{corollary}

\vspace{3mm}
\begin{proof}
Given $b_{1},b_{2},b_{3}\in\mathbb{C}$, we have 
\[\varphi\left(b_{1}\jj{xy}^{+}+b_{2}\jj{xy}^{-}+b_{3}\left(\idm\pm\ccm\right)\right)=b_{1}\jj{xy}^{-}+b_{2}\jj{xy}^{+}+b_{3}\left(\ccm\pm\idm\right)\]
The result then easily follows from solving
\[b_{1}\jj{xy}^{-}+b_{2}\jj{xy}^{+}+b_{3}\left(\ccm\pm\idm\right)=\pm\left(b_{1}\jj{xy}^{+}+b_{2}\jj{xy}^{-}+b_{3}\left(\idm\pm\ccm\right)\right)\]
\end{proof}

\begin{remark}\hspace{2mm}
\begin{enumerate}[label=(\roman*)]
\item For the $N=2$ case, $\left\{\idm+\ccm ,\idm-\ccm\right\}$ trivially defines an eigenbasis for $\varphi$.
\item By symmetry of our construction, we may permute labels $x$ and $y$ in (\ref{boisy1}) and (\ref{boisy2}). For e.g.\ (\ref{boisy1}), this tells us that $\jj{y}\in V_{1}$ and $\jjd{x}\in V_{-1}$. Recovering the precise linear relations is a straightforward task. 
\end{enumerate}\end{remark}

\vspace{2mm}
\section{Concluding Remarks and Outlook}
\label{outlook}
Using the rotation operator, we exploited the graphical calculus as a tool for exploring unitary spherical fusion categories (and their braided counterparts). We also used this approach to learn more about the link invariants associated to fusion rules of a particular form. Below, we summarise some of the highlights of the paper and discuss some possible directions for future work.

\subsection{Quantum invariants}\hspace{2mm}\\[1mm]
\label{conqinv}
Using the rotation operator, we extended \cite[Theorems 3.1 \& 3.2]{MSP1} to cover the anti-symmetrically self-dual cases. This produced skein relations (\ref{kbtwinskein}), (\ref{dubtwinskein}) and (\ref{kaufftwinskein}): to the knowledge of the authors, these have not previously appeared in the literature. In Appendix \ref{addendum}, we briefly investigated some properties of invariants associated to antisymmetrically self-dual objects.\\



\noindent We considered the framed link invariants $\Lambda_{\C,q}$ associated to (\ref{ourguysec3}) for $k=1,2$. A natural extension of this narrative would be to solve the problem below.

\setcounter{oproblem}{0} 

\vspace{1.5mm}
\begin{tcolorbox}[width=1.1\linewidth, center, breakable]
\begin{oproblem}\label{soulo5}
\textit{What is $\Lambda_{\C,q}$ when (i) $k\geq3$ , (ii) the fusion rule for $q^{\otimes2}$ is not multiplicity-free? }
\end{oproblem}
Partial results are known for (i) when $k=3$. If $q^{\otimes2}=\bm{1} \oplus x\oplus y \oplus q$, then $\Lambda_{\C,q}$ is said to be Kuperberg's $G_{2}$ invariant in most ``nontrivial'' cases \cite{MSP1}. See also \cite{wenzlpathalgs}.
\end{tcolorbox}

\noindent We narrowed our focus to discussing skein-theoretic methods for evaluating link diagrams in $\End(\bm{1})$ when all components are labelled by the same self-dual element \mbox{$q\in\Irr(\C)$.} More generally, one could ask the same question but for 
\vspace{-1mm}
\begin{enumerate}[label=(\roman*)]
\item ``Polychromatic'' link diagrams (i.e.\ components may have distinct labels), or 
\item When the labels are not necessarily self-dual (so that orientation matters).
\end{enumerate}
\vspace{-1mm}
For instance, when $\C$ is a Temperley-Lieb-Jones (TLJ) category, then any polychromatic link diagram can be evaluated as an element of the Kauffman bracket skein algebra: each component is replaced by the corresponding closed Jones-Wenzl idempotent, and the diagram is evaluated via skein relations (\ref{kbskein}). In TLJ categories, all objects are symmetrically self-dual. An important class of TQFTs known as Jones-Kauffman theories are described by TLJ categories (e.g. Ising and Fibonacci MTCs): here, Jones-Wenzl idempotents may be reinterpreted as anyons \cite{lureps}. \\

In Appendix \ref{reptheory}, we studied unitary representations of $\CC[B_{n}]$ that factor through the Iwahori-Hecke and Temperley-Lieb algebras. This resulted in a skein relation  (\ref{homflyfrskein}) for the framed HOMFLY-PT polynomial, which specialised to (\ref{kbskein})-(\ref{kbtwinskein}) in the context of a RFC $\C$ (since $b=\pm a^{-1}$ when $\varkappa_{q}=\pm1$). In this vein, we pose the following question.

\vspace{1.5mm}
\begin{tcolorbox}[width=1.1\linewidth, center, breakable]
\begin{oproblem}\label{asgrv}
\textit{Is there some $3$-variable link polynomial that specialises to (\ref{dubskein})-(\ref{kaufftwinskein})?}
\end{oproblem}
\noindent At the end of Appendix \ref{reptheory}, we see that the representation of $\CC[B_{n}]$ associated to $\Lambda_{\C,q}$ for $q^{\otimes2}=\bm{1}\oplus x\oplus y$ should factor through the cubic Hecke algebra $H_{n}(Q,3)$ and the Temperley-Lieb algebra. This motivates Problem \ref{asgrv} by analogy with the $q^{\otimes2}=\bm{1}\oplus x$ exposition.
\end{tcolorbox}

\vspace{5mm}
\subsection{F-Symbols}\hspace{2mm}\\[1mm]
\label{confmats}
In Section \ref{mainresults}, we considered the action of the rotation operator $\varphi$ on a basis of jumping jacks for $\End(q^{\otimes2})$. This was for a unitary spherical fusion category $\C$ containing a fusion rule of the form $q^{\otimes2}= \bm{1}\oplus\bigoplus_{i}x_{i}$ with all the $x_{i}$ self-dual. We deduced that $\varphi=\varkappa_{q}F^{qqq}_{q}$ (Theorem \ref{bojackthm}) and that $F^{qqq}_{q}$ is real-symmetric (Corollary \ref{realsymm}). For instances where $\C$ admits a braiding and $q^{\otimes2}=\bm{1}\oplus x \oplus y$ (with $x$ and $y$ self-dual), we found formulae for $F^{qqq}_{q}$ in terms of the quantum dimensions (Theorem \ref{cubicfsymbols}), and concluded that $\varkappa_{q}\neq-1$ (Corollary \ref{noasd}). We also saw that the spectrum of the rotation operator distinguishes between the Dubrovnik and Kauffman invariants (Corollary \ref{noasd}). Obvious extensions of this work would entail relaxing various assumptions e.g. as in the problem below.

\vspace{1.5mm}
\begin{tcolorbox}[width=1.1\linewidth, center, breakable]
\begin{oproblem}\textit{Can the results of Section \ref{mainresults} be extended to the case where}
\begin{enumerate}[label=(\alph*)]
\item \textit{the simple summands in $q^{\otimes2}$ are non self-dual?}
\item $\C$ \textit{is not unitary? }
\end{enumerate}
\textit{and can we extend Theorem \ref{cubicfsymbols} to determine general formulae for $F^{qqq}_{q}$ when}
\begin{enumerate}[label=(\alph*), resume]
\item  $\C$ \textit{does not admit a braiding?}
\end{enumerate}
\end{oproblem}


\end{tcolorbox}



\newpage

\newpage
\appendix

\section{Pivotal Coefficients}
\label{pivcoffappx}
A pivotal structure on a fusion category $\mathcal{C}$ affords our diagrams $2\pi$- rotational isotopy. For example, for a trivalent vertex with $a, b$ and $c \in \Irr(\mathcal{C})$, we have \vspace{1.5mm}

\begin{center}\def\svgwidth{6.5cm}
\begingroup%
  \makeatletter%
  \providecommand\color[2][]{%
    \errmessage{(Inkscape) Color is used for the text in Inkscape, but the package 'color.sty' is not loaded}%
    \renewcommand\color[2][]{}%
  }%
  \providecommand\transparent[1]{%
    \errmessage{(Inkscape) Transparency is used (non-zero) for the text in Inkscape, but the package 'transparent.sty' is not loaded}%
    \renewcommand\transparent[1]{}%
  }%
  \providecommand\rotatebox[2]{#2}%
  \newcommand*\fsize{\dimexpr\f@size pt\relax}%
  \newcommand*\lineheight[1]{\fontsize{\fsize}{#1\fsize}\selectfont}%
  \ifx\svgwidth\undefined%
    \setlength{\unitlength}{566.97100301bp}%
    \ifx\svgscale\undefined%
      \relax%
    \else%
      \setlength{\unitlength}{\unitlength * \real{\svgscale}}%
    \fi%
  \else%
    \setlength{\unitlength}{\svgwidth}%
  \fi%
  \global\let\svgwidth\undefined%
  \global\let\svgscale\undefined%
  \makeatother%
  \begin{picture}(1,0.33864969)%
    \lineheight{1}%
    \setlength\tabcolsep{0pt}%
    \put(0,0){\includegraphics[width=\unitlength,page=1]{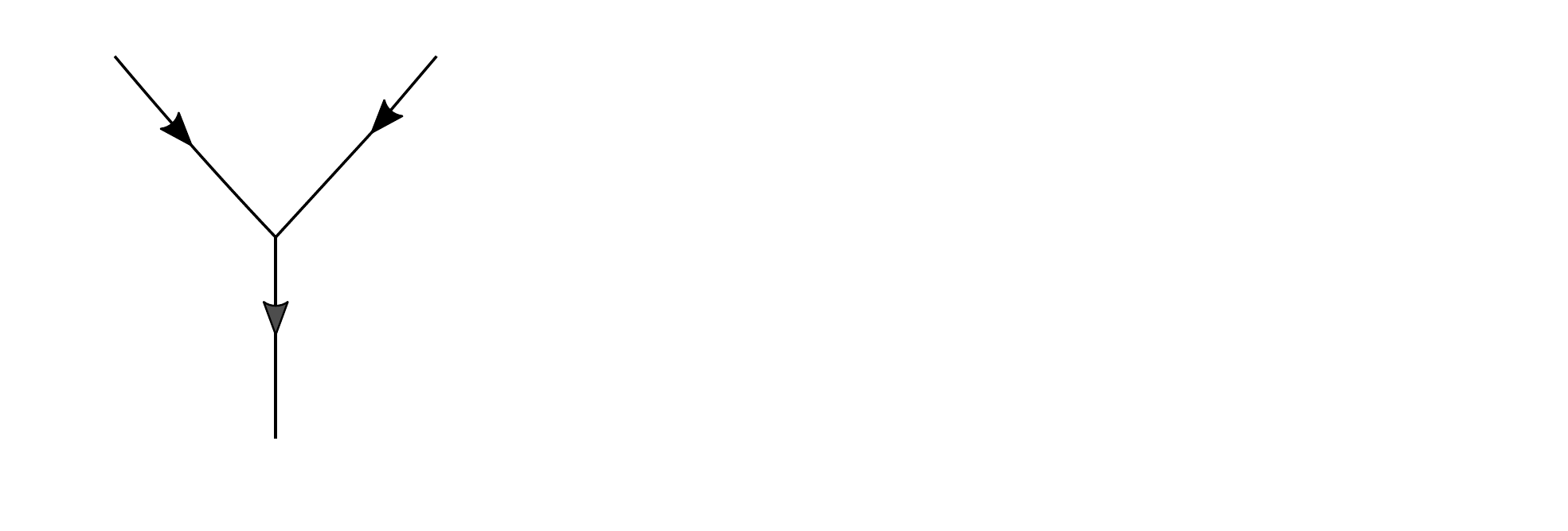}}%
    \put(0.00726508,0.28574212){\color[rgb]{0,0,0}\makebox(0,0)[lt]{\lineheight{1.25}\smash{\begin{tabular}[t]{l}$a$\end{tabular}}}}%
    \put(0.28843923,0.28302015){\color[rgb]{0,0,0}\makebox(0,0)[lt]{\lineheight{1.25}\smash{\begin{tabular}[t]{l}$b$\end{tabular}}}}%
    \put(0.12120915,0.05955429){\color[rgb]{0,0,0}\makebox(0,0)[lt]{\lineheight{1.25}\smash{\begin{tabular}[t]{l}$c$\end{tabular}}}}%
    \put(0.85688739,0.24616985){\color[rgb]{0,0,0}\makebox(0,0)[lt]{\lineheight{1.25}\smash{\begin{tabular}[t]{l}$a$\end{tabular}}}}%
    \put(0.93032725,0.24694677){\color[rgb]{0,0,0}\makebox(0,0)[lt]{\lineheight{1.25}\smash{\begin{tabular}[t]{l}$b$\end{tabular}}}}%
    \put(0.45030869,0.03984017){\color[rgb]{0,0,0}\makebox(0,0)[lt]{\lineheight{1.25}\smash{\begin{tabular}[t]{l}$c$\end{tabular}}}}%
    \put(0,0){\includegraphics[width=\unitlength,page=2]{diagram1.pdf}}%
    \put(0.36833682,0.17444266){\color[rgb]{0,0,0}\makebox(0,0)[lt]{\lineheight{1.25}\smash{\begin{tabular}[t]{l}=\end{tabular}}}}%
    \put(0,0){\includegraphics[width=\unitlength,page=3]{diagram1.pdf}}%
  \end{picture}%
\endgroup%
\label{pivstuff101}\end{center}

\noindent Applying the identity to a cap, this gives\\
\begin{center}
\def\svgwidth{10cm}
\begingroup%
  \makeatletter%
  \providecommand\color[2][]{%
    \errmessage{(Inkscape) Color is used for the text in Inkscape, but the package 'color.sty' is not loaded}%
    \renewcommand\color[2][]{}%
  }%
  \providecommand\transparent[1]{%
    \errmessage{(Inkscape) Transparency is used (non-zero) for the text in Inkscape, but the package 'transparent.sty' is not loaded}%
    \renewcommand\transparent[1]{}%
  }%
  \providecommand\rotatebox[2]{#2}%
  \newcommand*\fsize{\dimexpr\f@size pt\relax}%
  \newcommand*\lineheight[1]{\fontsize{\fsize}{#1\fsize}\selectfont}%
  \ifx\svgwidth\undefined%
    \setlength{\unitlength}{1274.84079917bp}%
    \ifx\svgscale\undefined%
      \relax%
    \else%
      \setlength{\unitlength}{\unitlength * \real{\svgscale}}%
    \fi%
  \else%
    \setlength{\unitlength}{\svgwidth}%
  \fi%
  \global\let\svgwidth\undefined%
  \global\let\svgscale\undefined%
  \makeatother%
  \begin{picture}(1,0.14318709)%
    \lineheight{1}%
    \setlength\tabcolsep{0pt}%
    \put(0,0){\includegraphics[width=\unitlength,page=1]{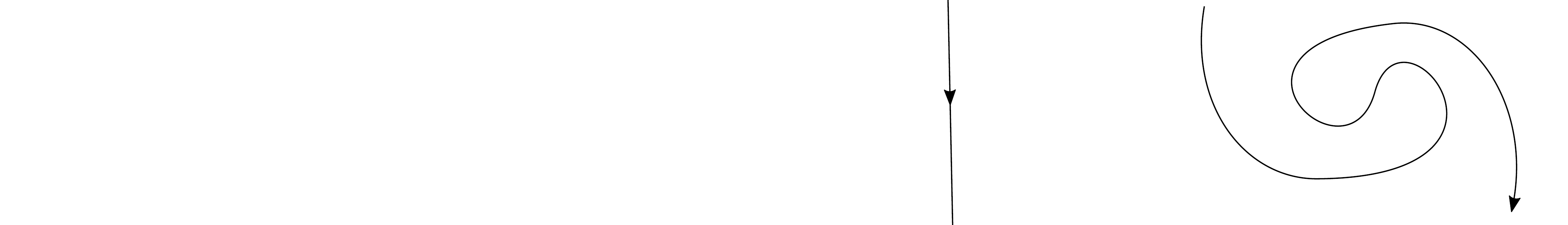}}%
    \put(0.68252886,0.05196946){\color[rgb]{0,0,0}\makebox(0,0)[lt]{\lineheight{1.25}\smash{\begin{tabular}[t]{l}=\end{tabular}}}}%
    \put(0.61220109,0.00273735){\color[rgb]{0,0,0}\makebox(0,0)[lt]{\lineheight{1.25}\smash{\begin{tabular}[t]{l}$a$\end{tabular}}}}%
    \put(0.97361939,0.00590439){\color[rgb]{0,0,0}\makebox(0,0)[lt]{\lineheight{1.25}\smash{\begin{tabular}[t]{l}$a$\end{tabular}}}}%
    \put(0,0){\includegraphics[width=\unitlength,page=2]{diagram2.pdf}}%
    \put(-0.0015018,0.01167744){\color[rgb]{0,0,0}\makebox(0,0)[lt]{\lineheight{1.25}\smash{\begin{tabular}[t]{l}$a$\end{tabular}}}}%
    \put(0.10818379,0.00938471){\color[rgb]{0,0,0}\makebox(0,0)[lt]{\lineheight{1.25}\smash{\begin{tabular}[t]{l}$a^*$\end{tabular}}}}%
    \put(0.37520555,0.01406206){\color[rgb]{0,0,0}\makebox(0,0)[lt]{\lineheight{1.25}\smash{\begin{tabular}[t]{l}$a$\end{tabular}}}}%
    \put(0.42804179,0.01291556){\color[rgb]{0,0,0}\makebox(0,0)[lt]{\lineheight{1.25}\smash{\begin{tabular}[t]{l}$a^*$\end{tabular}}}}%
    \put(0.1665848,0.05462921){\color[rgb]{0,0,0}\makebox(0,0)[lt]{\lineheight{1.25}\smash{\begin{tabular}[t]{l}=\end{tabular}}}}%
    \put(0.48309321,0.06811186){\color[rgb]{0,0,0}\makebox(0,0)[lt]{\lineheight{1.25}\smash{\begin{tabular}[t]{l}$\implies$\end{tabular}}}}%
  \end{picture}%
\endgroup%

\end{center}

\noindent We write
\begin{center}
\def\svgwidth{4cm}
\begingroup%
  \makeatletter%
  \providecommand\color[2][]{%
    \errmessage{(Inkscape) Color is used for the text in Inkscape, but the package 'color.sty' is not loaded}%
    \renewcommand\color[2][]{}%
  }%
  \providecommand\transparent[1]{%
    \errmessage{(Inkscape) Transparency is used (non-zero) for the text in Inkscape, but the package 'transparent.sty' is not loaded}%
    \renewcommand\transparent[1]{}%
  }%
  \providecommand\rotatebox[2]{#2}%
  \newcommand*\fsize{\dimexpr\f@size pt\relax}%
  \newcommand*\lineheight[1]{\fontsize{\fsize}{#1\fsize}\selectfont}%
  \ifx\svgwidth\undefined%
    \setlength{\unitlength}{385.92570147bp}%
    \ifx\svgscale\undefined%
      \relax%
    \else%
      \setlength{\unitlength}{\unitlength * \real{\svgscale}}%
    \fi%
  \else%
    \setlength{\unitlength}{\svgwidth}%
  \fi%
  \global\let\svgwidth\undefined%
  \global\let\svgscale\undefined%
  \makeatother%
  \begin{picture}(1,0.51147984)%
    \lineheight{1}%
    \setlength\tabcolsep{0pt}%
    \put(0,0){\includegraphics[width=\unitlength,page=1]{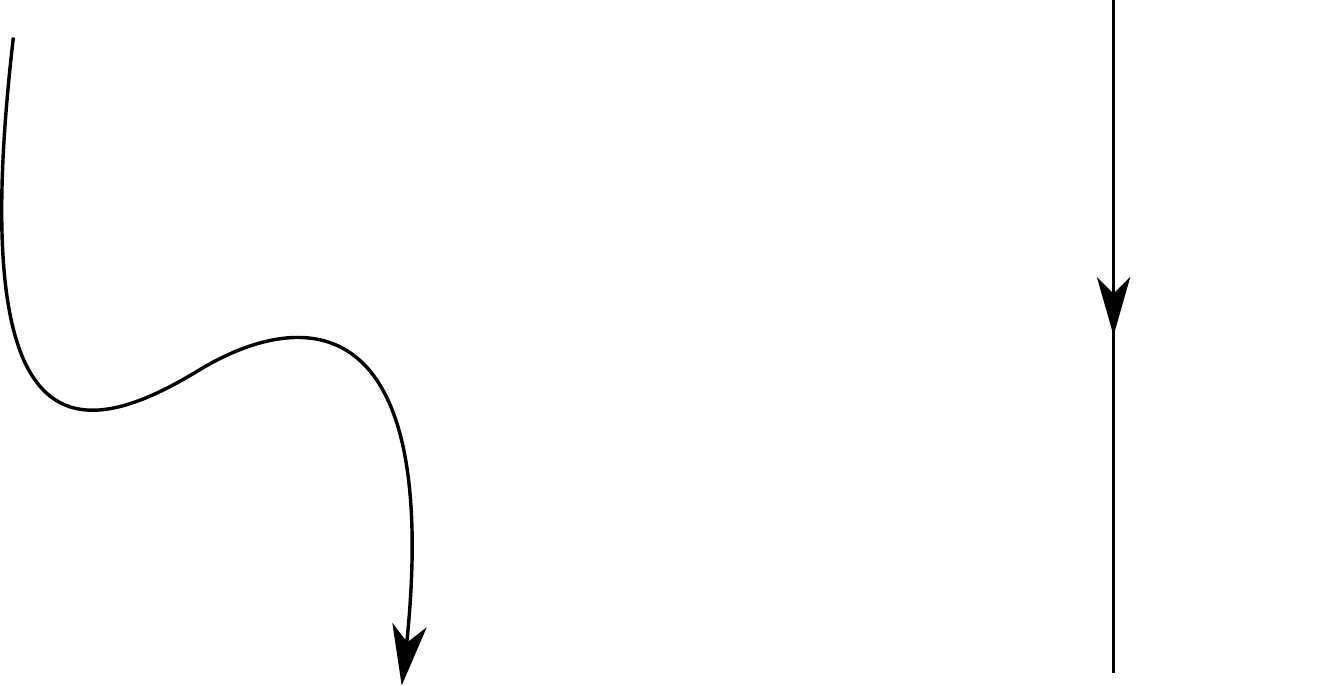}}%
    \put(0.48126626,0.24689152){\color[rgb]{0,0,0}\makebox(0,0)[lt]{\lineheight{1.25}\smash{\begin{tabular}[t]{l}$= t_a$\end{tabular}}}}%
    \put(0.33828009,0.02027188){\color[rgb]{0,0,0}\makebox(0,0)[lt]{\lineheight{1.25}\smash{\begin{tabular}[t]{l}$a$\end{tabular}}}}%
    \put(0.86296473,0.01690237){\color[rgb]{0,0,0}\makebox(0,0)[lt]{\lineheight{1.25}\smash{\begin{tabular}[t]{l}$a$\end{tabular}}}}%
  \end{picture}%
\endgroup%

   \end{center}
where $t_a\in \CC^{\times}$ by way of the pivotal structure.\\

\begin{proposition}\label{pivpropappx}
($i$) $|t_a|=1$ \ \ , \ \ ($ii$) $t_{a^*} = t_a^*$.
\end{proposition}
\begin{proof}\hspace{2mm}\vspace{2mm}
\begin{itemize}
    \item[(i)] \hspace{2mm} \begin{center}\vspace{-5mm}\def\svgwidth{10cm}
\begingroup%
  \makeatletter%
  \providecommand\color[2][]{%
    \errmessage{(Inkscape) Color is used for the text in Inkscape, but the package 'color.sty' is not loaded}%
    \renewcommand\color[2][]{}%
  }%
  \providecommand\transparent[1]{%
    \errmessage{(Inkscape) Transparency is used (non-zero) for the text in Inkscape, but the package 'transparent.sty' is not loaded}%
    \renewcommand\transparent[1]{}%
  }%
  \providecommand\rotatebox[2]{#2}%
  \newcommand*\fsize{\dimexpr\f@size pt\relax}%
  \newcommand*\lineheight[1]{\fontsize{\fsize}{#1\fsize}\selectfont}%
  \ifx\svgwidth\undefined%
    \setlength{\unitlength}{795.62737412bp}%
    \ifx\svgscale\undefined%
      \relax%
    \else%
      \setlength{\unitlength}{\unitlength * \real{\svgscale}}%
    \fi%
  \else%
    \setlength{\unitlength}{\svgwidth}%
  \fi%
  \global\let\svgwidth\undefined%
  \global\let\svgscale\undefined%
  \makeatother%
  \begin{picture}(1,0.37231665)%
    \lineheight{1}%
    \setlength\tabcolsep{0pt}%
    \put(0.12354918,0.19190456){\color[rgb]{0,0,0}\makebox(0,0)[lt]{\lineheight{1.25}\smash{\begin{tabular}[t]{l}$a$\end{tabular}}}}%
    \put(0,0){\includegraphics[width=\unitlength,page=1]{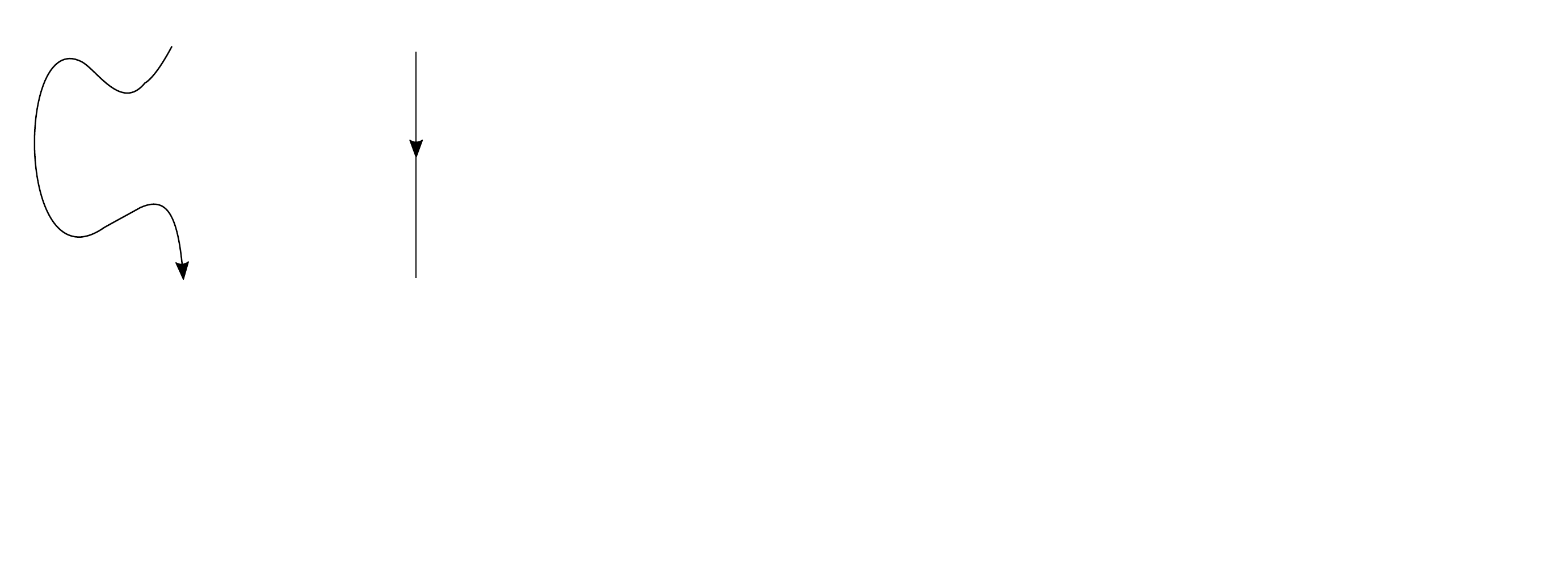}}%
    \put(0.13813457,0.26641763){\color[rgb]{0,0,0}\makebox(0,0)[lt]{\lineheight{1.25}\smash{\begin{tabular}[t]{l}$=\tilde{t}_a$\end{tabular}}}}%
    \put(0.27074585,0.19085744){\color[rgb]{0,0,0}\makebox(0,0)[lt]{\lineheight{1.25}\smash{\begin{tabular}[t]{l}$a$\end{tabular}}}}%
    \put(0,0){\includegraphics[width=\unitlength,page=2]{diagram4.pdf}}%
    \put(0.327417,0.26807781){\color[rgb]{0,0,0}\makebox(0,0)[lt]{\lineheight{1.25}\smash{\begin{tabular}[t]{l}$\overset{\widetilde{Tr}}{\implies}$\end{tabular}}}}%
    \put(0.57162629,0.17856027){\color[rgb]{0,0,0}\makebox(0,0)[lt]{\lineheight{1.25}\smash{\begin{tabular}[t]{l}$a$\end{tabular}}}}%
    \put(0.69352317,0.26650744){\color[rgb]{0,0,0}\makebox(0,0)[lt]{\lineheight{1.25}\smash{\begin{tabular}[t]{l}$= \tilde{t}_a$\end{tabular}}}}%
    \put(0.95168228,0.24609117){\color[rgb]{0,0,0}\makebox(0,0)[lt]{\lineheight{1.25}\smash{\begin{tabular}[t]{l}$a$\end{tabular}}}}%
    \put(0.62192259,0.05974644){\color[rgb]{0,0,0}\makebox(0,0)[lt]{\lineheight{1.25}\smash{\begin{tabular}[t]{l}$= \tilde{t}_a$\end{tabular}}}}%
    \put(0.86597037,0.03173224){\color[rgb]{0,0,0}\makebox(0,0)[lt]{\lineheight{1.25}\smash{\begin{tabular}[t]{l}$a$\end{tabular}}}}%
    \put(0.59541597,0.03195908){\color[rgb]{0,0,0}\makebox(0,0)[lt]{\lineheight{1.25}\smash{\begin{tabular}[t]{l}$a$\end{tabular}}}}%
    \put(0.38070992,0.05899781){\color[rgb]{0,0,0}\makebox(0,0)[lt]{\lineheight{1.25}\smash{\begin{tabular}[t]{l}$\implies$\end{tabular}}}}%
    \put(0,0){\includegraphics[width=\unitlength,page=3]{diagram4.pdf}}%
  \end{picture}%
\endgroup%
\end{center}

    \item[(ii)] Observe that 
   \begin{center}
      \def\svgwidth{8.5cm}
\begingroup%
  \makeatletter%
  \providecommand\color[2][]{%
    \errmessage{(Inkscape) Color is used for the text in Inkscape, but the package 'color.sty' is not loaded}%
    \renewcommand\color[2][]{}%
  }%
  \providecommand\transparent[1]{%
    \errmessage{(Inkscape) Transparency is used (non-zero) for the text in Inkscape, but the package 'transparent.sty' is not loaded}%
    \renewcommand\transparent[1]{}%
  }%
  \providecommand\rotatebox[2]{#2}%
  \newcommand*\fsize{\dimexpr\f@size pt\relax}%
  \newcommand*\lineheight[1]{\fontsize{\fsize}{#1\fsize}\selectfont}%
  \ifx\svgwidth\undefined%
    \setlength{\unitlength}{888.62153782bp}%
    \ifx\svgscale\undefined%
      \relax%
    \else%
      \setlength{\unitlength}{\unitlength * \real{\svgscale}}%
    \fi%
  \else%
    \setlength{\unitlength}{\svgwidth}%
  \fi%
  \global\let\svgwidth\undefined%
  \global\let\svgscale\undefined%
  \makeatother%
  \begin{picture}(1,0.22556199)%
    \lineheight{1}%
    \setlength\tabcolsep{0pt}%
    \put(0,0){\includegraphics[width=\unitlength,page=1]{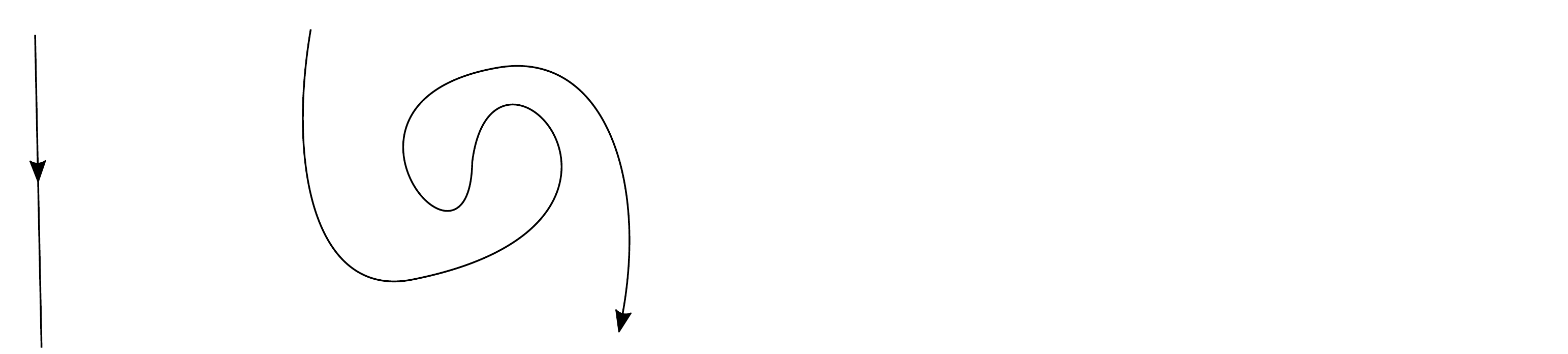}}%
    \put(0.03286471,0.00782897){\color[rgb]{0,0,0}\makebox(0,0)[lt]{\lineheight{1.25}\smash{\begin{tabular}[t]{l}$a$\end{tabular}}}}%
    \put(0.41458037,0.01491114){\color[rgb]{0,0,0}\makebox(0,0)[lt]{\lineheight{1.25}\smash{\begin{tabular}[t]{l}$a$\end{tabular}}}}%
    \put(0,0){\includegraphics[width=\unitlength,page=2]{diagram5.pdf}}%
    \put(0.44645877,0.10720898){\color[rgb]{0,0,0}\makebox(0,0)[lt]{\lineheight{1.25}\smash{\begin{tabular}[t]{l}$= t_a^{*}$\end{tabular}}}}%
    \put(0.70471515,0.01588645){\color[rgb]{0,0,0}\makebox(0,0)[lt]{\lineheight{1.25}\smash{\begin{tabular}[t]{l}$a$\end{tabular}}}}%
    \put(0.88548214,0.01446675){\color[rgb]{0,0,0}\makebox(0,0)[lt]{\lineheight{1.25}\smash{\begin{tabular}[t]{l}$a$\end{tabular}}}}%
    \put(0.08522161,0.10073643){\color[rgb]{0,0,0}\makebox(0,0)[lt]{\lineheight{1.25}\smash{\begin{tabular}[t]{l}$=$\end{tabular}}}}%
    \put(0.73744456,0.10972946){\color[rgb]{0,0,0}\makebox(0,0)[lt]{\lineheight{1.25}\smash{\begin{tabular}[t]{l}$=t_a^{*} t_a$\end{tabular}}}}%
  \end{picture}%
\endgroup%

   \end{center} 
whence the result follows using (i).
    \end{itemize}
\end{proof}
\vspace{2mm}
\noindent Consequently, when $a=a^*$, we have that $t_a = \pm1$. In this instance, $t_a$ is called the \textit{Frobenius-Schur indicator} and is written $\varkappa_a:= t_a$. As stated in Section \ref{frobschsec}, $t_a$ is called the pivotal coefficient of $a\in \Irr(\mathcal{C})$. It is straightforward to show that $t_a$ is gauge-invariant if and only if $a=a^*$.When $a$ is non self-dual it is typical to fix the gauge such that $t_a=1$.
When $a$ is self-dual,
\begin{enumerate}[label=(\roman*)]
    \item $a$ is called \textit{symmetrically self-dual} if $\varkappa_a = 1$
    \item $a$ is called \textit{antisymmetrically self-dual} if $\varkappa_a = -1$
\end{enumerate}

\noindent If $\mathcal{C}$ is a unitary ribbon fusion category, note that
\vspace{2mm}
\begin{center}
\def\svgwidth{11cm}
\begingroup%
  \makeatletter%
  \providecommand\color[2][]{%
    \errmessage{(Inkscape) Color is used for the text in Inkscape, but the package 'color.sty' is not loaded}%
    \renewcommand\color[2][]{}%
  }%
  \providecommand\transparent[1]{%
    \errmessage{(Inkscape) Transparency is used (non-zero) for the text in Inkscape, but the package 'transparent.sty' is not loaded}%
    \renewcommand\transparent[1]{}%
  }%
  \providecommand\rotatebox[2]{#2}%
  \newcommand*\fsize{\dimexpr\f@size pt\relax}%
  \newcommand*\lineheight[1]{\fontsize{\fsize}{#1\fsize}\selectfont}%
  \ifx\svgwidth\undefined%
    \setlength{\unitlength}{1165.69785771bp}%
    \ifx\svgscale\undefined%
      \relax%
    \else%
      \setlength{\unitlength}{\unitlength * \real{\svgscale}}%
    \fi%
  \else%
    \setlength{\unitlength}{\svgwidth}%
  \fi%
  \global\let\svgwidth\undefined%
  \global\let\svgscale\undefined%
  \makeatother%
  \begin{picture}(1,0.35774271)%
    \lineheight{1}%
    \setlength\tabcolsep{0pt}%
    \put(0,0){\includegraphics[width=\unitlength,page=1]{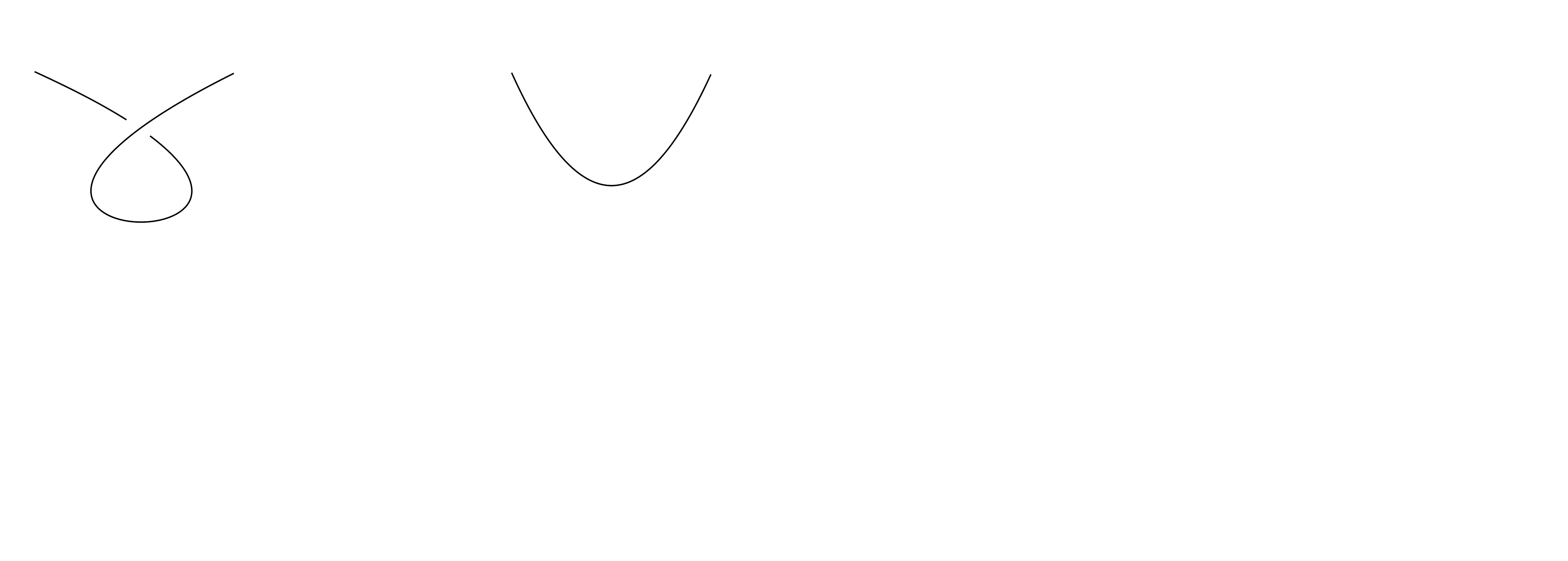}}%
    \put(-0.00194776,0.31734683){\color[rgb]{0,0,0}\makebox(0,0)[lt]{\lineheight{1.25}\smash{\begin{tabular}[t]{l}$a$\end{tabular}}}}%
    \put(0.12873076,0.32036555){\color[rgb]{0,0,0}\makebox(0,0)[lt]{\lineheight{1.25}\smash{\begin{tabular}[t]{l}$a^{*}$\end{tabular}}}}%
    \put(0.3057303,0.32118498){\color[rgb]{0,0,0}\makebox(0,0)[lt]{\lineheight{1.25}\smash{\begin{tabular}[t]{l}$a$\end{tabular}}}}%
    \put(0.44011888,0.31708784){\color[rgb]{0,0,0}\makebox(0,0)[lt]{\lineheight{1.25}\smash{\begin{tabular}[t]{l}$a^{*}$\end{tabular}}}}%
    \put(0.18855001,0.26710186){\color[rgb]{0,0,0}\makebox(0,0)[lt]{\lineheight{1.25}\smash{\begin{tabular}[t]{l}$= R_0^{aa^{*}}$\end{tabular}}}}%
    \put(0.4880653,0.2673282){\color[rgb]{0,0,0}\makebox(0,0)[lt]{\lineheight{1.25}\smash{\begin{tabular}[t]{l}$\implies$\end{tabular}}}}%
    \put(0,0){\includegraphics[width=\unitlength,page=2]{diagram7.pdf}}%
    \put(0.56816783,0.33568013){\color[rgb]{0,0,0}\makebox(0,0)[lt]{\lineheight{1.25}\smash{\begin{tabular}[t]{l}$a$\end{tabular}}}}%
    \put(0.69528026,0.27283801){\color[rgb]{0,0,0}\makebox(0,0)[lt]{\lineheight{1.25}\smash{\begin{tabular}[t]{l}$= R^{aa^{*}}_0$\end{tabular}}}}%
    \put(0,0){\includegraphics[width=\unitlength,page=3]{diagram7.pdf}}%
    \put(0.94498631,0.20763971){\color[rgb]{0,0,0}\makebox(0,0)[lt]{\lineheight{1.25}\smash{\begin{tabular}[t]{l}$a$\end{tabular}}}}%
    \put(0.47152651,0.08066554){\color[rgb]{0,0,0}\makebox(0,0)[lt]{\lineheight{1.25}\smash{\begin{tabular}[t]{l}$\implies$\end{tabular}}}}%
    \put(0,0){\includegraphics[width=\unitlength,page=4]{diagram7.pdf}}%
    \put(0.55162901,0.14901746){\color[rgb]{0,0,0}\makebox(0,0)[lt]{\lineheight{1.25}\smash{\begin{tabular}[t]{l}$a$\end{tabular}}}}%
    \put(0.68237676,0.0810357){\color[rgb]{0,0,0}\makebox(0,0)[lt]{\lineheight{1.25}\smash{\begin{tabular}[t]{l}$= R^{aa^{*}}_0 t_a$\end{tabular}}}}%
    \put(0,0){\includegraphics[width=\unitlength,page=5]{diagram7.pdf}}%
    \put(0.92737901,0.00453639){\color[rgb]{0,0,0}\makebox(0,0)[lt]{\lineheight{1.25}\smash{\begin{tabular}[t]{l}$a$\end{tabular}}}}%
    \put(0,0){\includegraphics[width=\unitlength,page=6]{diagram7.pdf}}%
  \end{picture}%
\endgroup%

\end{center}
\vspace{2mm}
That is, $\vartheta_a^{-1}=R^{aa^{*}}_0 t_a$. In particular, for $a$ self-dual we have $\vartheta_a=\varkappa_a(R^{aa}_0)^{-1}$.


\newpage
\section{Normalisation}
\label{normaddappx}

\noindent Let $\mathcal{C}$ be a spherical fusion category with positive dagger structure. Recall from \mbox{Section \ref{daginnuni}} that the $\Hom$-spaces of $\C$ come with a Hermitian inner product $\langle \cdot,\cdot \rangle$. We can write
\vspace{1.5mm}
\begin{equation}
    \widetilde{Tr}(\id_x)=\left\langle \raisebox{-3mm}{\def\svgwidth{2cm}
\begingroup%
  \makeatletter%
  \providecommand\color[2][]{%
    \errmessage{(Inkscape) Color is used for the text in Inkscape, but the package 'color.sty' is not loaded}%
    \renewcommand\color[2][]{}%
  }%
  \providecommand\transparent[1]{%
    \errmessage{(Inkscape) Transparency is used (non-zero) for the text in Inkscape, but the package 'transparent.sty' is not loaded}%
    \renewcommand\transparent[1]{}%
  }%
  \providecommand\rotatebox[2]{#2}%
  \newcommand*\fsize{\dimexpr\f@size pt\relax}%
  \newcommand*\lineheight[1]{\fontsize{\fsize}{#1\fsize}\selectfont}%
  \ifx\svgwidth\undefined%
    \setlength{\unitlength}{193.28830989bp}%
    \ifx\svgscale\undefined%
      \relax%
    \else%
      \setlength{\unitlength}{\unitlength * \real{\svgscale}}%
    \fi%
  \else%
    \setlength{\unitlength}{\svgwidth}%
  \fi%
  \global\let\svgwidth\undefined%
  \global\let\svgscale\undefined%
  \makeatother%
  \begin{picture}(1,0.36208433)%
    \lineheight{1}%
    \setlength\tabcolsep{0pt}%
    \put(0,0){\includegraphics[width=\unitlength,page=1]{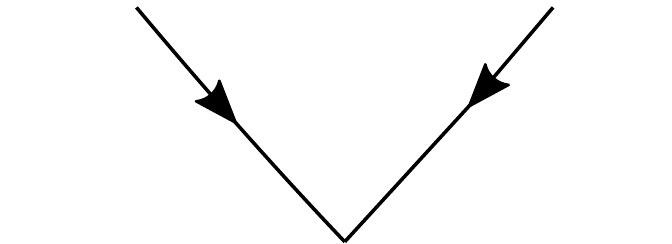}}%
    \put(-0.00615022,0.29955165){\color[rgb]{0,0,0}\makebox(0,0)[lt]{\lineheight{1.25}\smash{\begin{tabular}[t]{l}$x^{*}$\end{tabular}}}}%
    \put(0.85414203,0.29132028){\color[rgb]{0,0,0}\makebox(0,0)[lt]{\lineheight{1.25}\smash{\begin{tabular}[t]{l}$x$\end{tabular}}}}%
  \end{picture}%
\endgroup%
} \ , \  \raisebox{-3mm}{\def\svgwidth{2cm}}\right\rangle=:d_x>0
\end{equation}
\noindent We want to assign a factor of $\nu^{x^{*}x}_0$ to the above cup in order to normalise it with respect to the inner product i.e. $\nu^{x^*x}_0= d_x^{-\frac{1}{2}}$. Sphericality gives us $d_x=d_{x^*}$ and so $\nu^{x^*x}_0=\nu^{xx^*}_0$. Using the same notation as in (\ref{innerproddef}), we have
\begin{equation}\label{innprodrevisited}
    \raisebox{-8.5mm}{\includegraphics[width=0.08\textwidth]{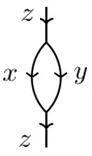}} \ =\left\langle\trivalentxyz , \trivalentxyz \right\rangle \ \Ststrandz = \lambda_{zxyz} \ \Ststrandz \ \ , \ \ \lambda_{zxyz}>0
\end{equation}






\noindent We want to scale the trivalent vertices appearing in (\ref{innprodrevisited}) by a normalising factor \mbox{$\nu^{xy}_z=\lambda^{-1/2}_{zxyz}$} in a manner that is consistent with the factor $\nu^{xx^*}_0$. Using the dagger structure, note that the normalisation factor for the adjoint trivalent vertex is also $\nu^{xy}_{z}$. Expanding the identity operator for $\End(x \otimes y)$ in the canonical basis, 
\begin{center}
\def\svgwidth{21.5cm}
\begingroup%
  \makeatletter%
  \providecommand\color[2][]{%
    \errmessage{(Inkscape) Color is used for the text in Inkscape, but the package 'color.sty' is not loaded}%
    \renewcommand\color[2][]{}%
  }%
  \providecommand\transparent[1]{%
    \errmessage{(Inkscape) Transparency is used (non-zero) for the text in Inkscape, but the package 'transparent.sty' is not loaded}%
    \renewcommand\transparent[1]{}%
  }%
  \providecommand\rotatebox[2]{#2}%
  \newcommand*\fsize{\dimexpr\f@size pt\relax}%
  \newcommand*\lineheight[1]{\fontsize{\fsize}{#1\fsize}\selectfont}%
  \ifx\svgwidth\undefined%
    \setlength{\unitlength}{1950.20404005bp}%
    \ifx\svgscale\undefined%
      \relax%
    \else%
      \setlength{\unitlength}{\unitlength * \real{\svgscale}}%
    \fi%
  \else%
    \setlength{\unitlength}{\svgwidth}%
  \fi%
  \global\let\svgwidth\undefined%
  \global\let\svgscale\undefined%
  \makeatother%
  \begin{picture}(1,0.1115704)%
    \lineheight{1}%
    \setlength\tabcolsep{0pt}%
    \put(0,0){\includegraphics[width=\unitlength,page=1]{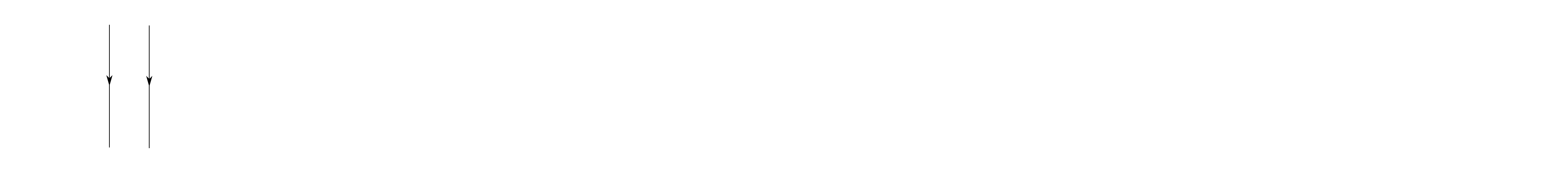}}%
    \put(-0.00091789,0.06175879){\color[rgb]{0,0,0}\makebox(0,0)[lt]{\lineheight{1.25}\smash{\begin{tabular}[t]{l}$\id_{x\otimes y}=$\end{tabular}}}}%
    \put(0.1124906,0.06200915){\color[rgb]{0,0,0}\makebox(0,0)[lt]{\lineheight{1.25}\smash{\begin{tabular}[t]{l}$=\sum_{\mu, z}(\nu^{xy}_z)^2$\end{tabular}}}}%
    \put(0,0){\includegraphics[width=\unitlength,page=2]{diagram13.pdf}}%
    \put(0.25063088,0.05538087){\color[rgb]{0,0,0}\makebox(0,0)[lt]{\lineheight{1.25}\smash{\begin{tabular}[t]{l}$z$\end{tabular}}}}%
    \put(0.22045045,0.10046863){\color[rgb]{0,0,0}\makebox(0,0)[lt]{\lineheight{1.25}\smash{\begin{tabular}[t]{l}$x$\end{tabular}}}}%
    \put(0.25695824,0.10093069){\color[rgb]{0,0,0}\makebox(0,0)[lt]{\lineheight{1.25}\smash{\begin{tabular}[t]{l}$y$\end{tabular}}}}%
    \put(0.22345412,0.0073506){\color[rgb]{0,0,0}\makebox(0,0)[lt]{\lineheight{1.25}\smash{\begin{tabular}[t]{l}$x$\end{tabular}}}}%
    \put(0.25996196,0.00781266){\color[rgb]{0,0,0}\makebox(0,0)[lt]{\lineheight{1.25}\smash{\begin{tabular}[t]{l}$y$\end{tabular}}}}%
    \put(0.06124851,0.00272926){\color[rgb]{0,0,0}\makebox(0,0)[lt]{\lineheight{1.25}\smash{\begin{tabular}[t]{l}$x$\end{tabular}}}}%
    \put(0.09082449,0.00226719){\color[rgb]{0,0,0}\makebox(0,0)[lt]{\lineheight{1.25}\smash{\begin{tabular}[t]{l}$y$\end{tabular}}}}%
    \put(0.28083828,0.06087872){\color[rgb]{0,0,0}\makebox(0,0)[lt]{\lineheight{1.25}\smash{\begin{tabular}[t]{l}$\overset{\widetilde{Tr}}{\implies}$\end{tabular}}}}%
    \put(0.33943891,0.0608662){\color[rgb]{0,0,0}\makebox(0,0)[lt]{\lineheight{1.25}\smash{\begin{tabular}[t]{l}$d_x d_y=\sum_{\mu, z}(\nu^{xy}_z)^2$\end{tabular}}}}%
    \put(0,0){\includegraphics[width=\unitlength,page=3]{diagram13.pdf}}%
    \put(0.53975569,0.05709378){\color[rgb]{0,0,0}\makebox(0,0)[lt]{\lineheight{1.25}\smash{\begin{tabular}[t]{l}$z$\end{tabular}}}}%
    \put(0.51930356,0.10056){\color[rgb]{0,0,0}\makebox(0,0)[lt]{\lineheight{1.25}\smash{\begin{tabular}[t]{l}$x$\end{tabular}}}}%
    \put(0.54598016,0.10289767){\color[rgb]{0,0,0}\makebox(0,0)[lt]{\lineheight{1.25}\smash{\begin{tabular}[t]{l}$y$\end{tabular}}}}%
    \put(0.51865913,0.01068491){\color[rgb]{0,0,0}\makebox(0,0)[lt]{\lineheight{1.25}\smash{\begin{tabular}[t]{l}$x$\end{tabular}}}}%
    \put(0.54386569,0.0067712){\color[rgb]{0,0,0}\makebox(0,0)[lt]{\lineheight{1.25}\smash{\begin{tabular}[t]{l}$y$\end{tabular}}}}%
    \put(0,0){\includegraphics[width=\unitlength,page=4]{diagram13.pdf}}%
    \put(0.61097948,0.05904534){\color[rgb]{0,0,0}\makebox(0,0)[lt]{\lineheight{1.25}\smash{\begin{tabular}[t]{l}$=\sum_z N^{xy}_z B^{xy}_z(\nu^{xy}_z)^2$\end{tabular}}}}%
  \end{picture}%
\endgroup%

\end{center}
where we define $B^{xy}_z:=\raisebox{-9mm}{\def\svgwidth{2cm}
\begingroup%
  \makeatletter%
  \providecommand\color[2][]{%
    \errmessage{(Inkscape) Color is used for the text in Inkscape, but the package 'color.sty' is not loaded}%
    \renewcommand\color[2][]{}%
  }%
  \providecommand\transparent[1]{%
    \errmessage{(Inkscape) Transparency is used (non-zero) for the text in Inkscape, but the package 'transparent.sty' is not loaded}%
    \renewcommand\transparent[1]{}%
  }%
  \providecommand\rotatebox[2]{#2}%
  \newcommand*\fsize{\dimexpr\f@size pt\relax}%
  \newcommand*\lineheight[1]{\fontsize{\fsize}{#1\fsize}\selectfont}%
  \ifx\svgwidth\undefined%
    \setlength{\unitlength}{176.9163345bp}%
    \ifx\svgscale\undefined%
      \relax%
    \else%
      \setlength{\unitlength}{\unitlength * \real{\svgscale}}%
    \fi%
  \else%
    \setlength{\unitlength}{\svgwidth}%
  \fi%
  \global\let\svgwidth\undefined%
  \global\let\svgscale\undefined%
  \makeatother%
  \begin{picture}(1,1.18022641)%
    \lineheight{1}%
    \setlength\tabcolsep{0pt}%
    \put(0,0){\includegraphics[width=\unitlength,page=1]{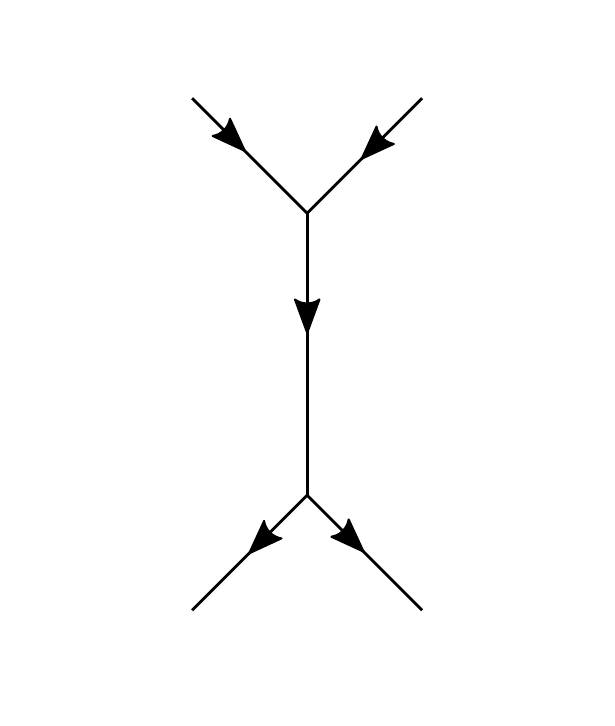}}%
    \put(0.52257709,0.57971456){\color[rgb]{0,0,0}\makebox(0,0)[lt]{\lineheight{1.25}\smash{\begin{tabular}[t]{l}$z$\end{tabular}}}}%
    \put(0.29712668,1.05885644){\color[rgb]{0,0,0}\makebox(0,0)[lt]{\lineheight{1.25}\smash{\begin{tabular}[t]{l}$x$\end{tabular}}}}%
    \put(0.59119117,1.08462529){\color[rgb]{0,0,0}\makebox(0,0)[lt]{\lineheight{1.25}\smash{\begin{tabular}[t]{l}$y$\end{tabular}}}}%
    \put(0.29002293,0.0681352){\color[rgb]{0,0,0}\makebox(0,0)[lt]{\lineheight{1.25}\smash{\begin{tabular}[t]{l}$x$\end{tabular}}}}%
    \put(0.56788265,0.02499318){\color[rgb]{0,0,0}\makebox(0,0)[lt]{\lineheight{1.25}\smash{\begin{tabular}[t]{l}$y$\end{tabular}}}}%
    \put(0,0){\includegraphics[width=\unitlength,page=2]{Bxyz.pdf}}%
  \end{picture}%
\endgroup%
}$ 

\begin{proposition}\hspace{2mm}
\begin{enumerate}[label=(\roman*)]
    \item $B^{xy}_z$ satisfies the same symmetries as $N^{xy}_z$ for a fusion category i.e. 
    \begin{align*}
        B^{xy}_z = B^{z^{*}x}_{y^*} = B^{yz^*}_{x^*} \quad , \quad  B^{xy}_z = B^{y^{*}x^{*}}_{z^*} 
    \end{align*}
    \item $B^{xy}_z \in \RR_{>0}$
    \item $B^{x0}_z=B^{0x}_z=\delta_{xz}d_x$
    \item $B^{xy}_z= d_z(\nu^{xy}_z)^{-2}$
\end{enumerate}
\end{proposition}
\begin{proof}\hspace{2mm}\vspace{1.5mm}
\begin{enumerate}[label=(\roman*)]
    \item \hspace{2cm}\begin{center}\vspace{-10mm}
	\def\svgwidth{17cm}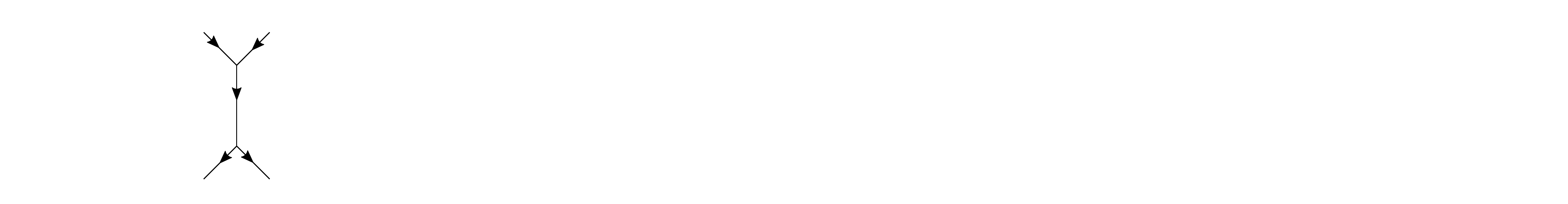 \end{center}
	Also, \\
\begin{center}\vspace{-5mm}
	\def\svgwidth{17cm}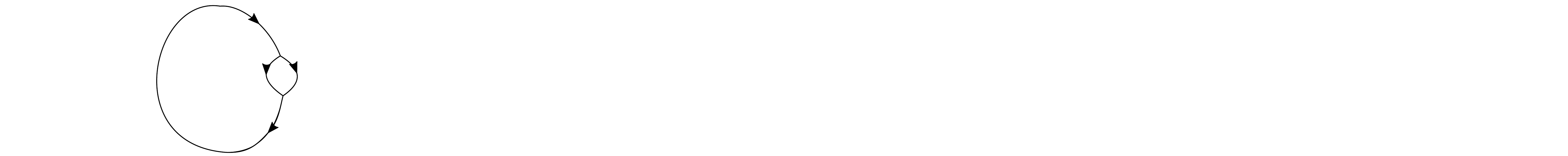
    \end{center}
    
    \item From the proof of (i), $B^{xy}_z=\scalebox{1.5}{||}\trivalentdoozy\scalebox{1.5}{||}^2 \  d_y=(\nu^{x^{*}z}_y)^{-2}d_y>0$ 
    \item Follows immediately upon inspection of $B^{xy}_z$
    \item \hspace{2mm}\begin{center}\vspace{-5mm}
      \def\svgwidth{16cm}
\begingroup%
  \makeatletter%
  \providecommand\color[2][]{%
    \errmessage{(Inkscape) Color is used for the text in Inkscape, but the package 'color.sty' is not loaded}%
    \renewcommand\color[2][]{}%
  }%
  \providecommand\transparent[1]{%
    \errmessage{(Inkscape) Transparency is used (non-zero) for the text in Inkscape, but the package 'transparent.sty' is not loaded}%
    \renewcommand\transparent[1]{}%
  }%
  \providecommand\rotatebox[2]{#2}%
  \newcommand*\fsize{\dimexpr\f@size pt\relax}%
  \newcommand*\lineheight[1]{\fontsize{\fsize}{#1\fsize}\selectfont}%
  \ifx\svgwidth\undefined%
    \setlength{\unitlength}{1264.02393732bp}%
    \ifx\svgscale\undefined%
      \relax%
    \else%
      \setlength{\unitlength}{\unitlength * \real{\svgscale}}%
    \fi%
  \else%
    \setlength{\unitlength}{\svgwidth}%
  \fi%
  \global\let\svgwidth\undefined%
  \global\let\svgscale\undefined%
  \makeatother%
  \begin{picture}(1,0.16116833)%
    \lineheight{1}%
    \setlength\tabcolsep{0pt}%
    \put(0,0){\includegraphics[width=\unitlength,page=1]{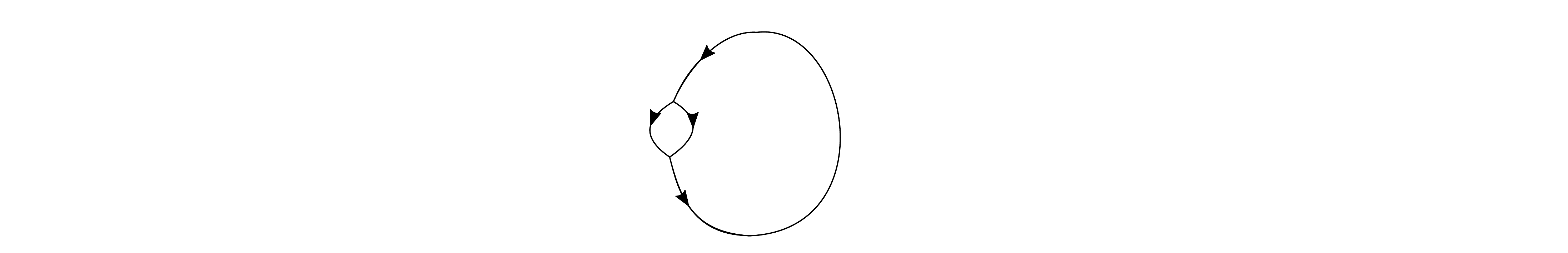}}%
    \put(0.38588908,0.08315794){\color[rgb]{0,0,0}\makebox(0,0)[lt]{\lineheight{1.25}\smash{\begin{tabular}[t]{l}$y^{*}$\end{tabular}}}}%
    \put(0.45029881,0.08328931){\color[rgb]{0,0,0}\makebox(0,0)[lt]{\lineheight{1.25}\smash{\begin{tabular}[t]{l}$x^{*}$\end{tabular}}}}%
    \put(-0.00165328,0.08015662){\color[rgb]{0,0,0}\makebox(0,0)[lt]{\lineheight{1.25}\smash{\begin{tabular}[t]{l}$B^{xy}_z \overset{(i)}{=}B^{yz^{*}}_{x^{*}}=$\end{tabular}}}}%
    \put(0,0){\includegraphics[width=\unitlength,page=2]{diagram15.pdf}}%
    \put(0.24290547,0.07509992){\color[rgb]{0,0,0}\makebox(0,0)[lt]{\lineheight{1.25}\smash{\begin{tabular}[t]{l}$x^{*}$\end{tabular}}}}%
    \put(0.21135101,0.14778768){\color[rgb]{0,0,0}\makebox(0,0)[lt]{\lineheight{1.25}\smash{\begin{tabular}[t]{l}$y$\end{tabular}}}}%
    \put(0.31818236,0.09296978){\color[rgb]{0,0,0}\makebox(0,0)[lt]{\lineheight{1.25}\smash{\begin{tabular}[t]{l}$z$\end{tabular}}}}%
    \put(0.21035675,0.00349806){\color[rgb]{0,0,0}\makebox(0,0)[lt]{\lineheight{1.25}\smash{\begin{tabular}[t]{l}$y$\end{tabular}}}}%
    \put(0,0){\includegraphics[width=\unitlength,page=3]{diagram15.pdf}}%
    \put(0.3347549,0.07645224){\color[rgb]{0,0,0}\makebox(0,0)[lt]{\lineheight{1.25}\smash{\begin{tabular}[t]{l}$=$\end{tabular}}}}%
    \put(0.41620499,0.12549237){\color[rgb]{0,0,0}\makebox(0,0)[lt]{\lineheight{1.25}\smash{\begin{tabular}[t]{l}$z^{*}$\end{tabular}}}}%
    \put(0.39975064,0.00962098){\color[rgb]{0,0,0}\makebox(0,0)[lt]{\lineheight{1.25}\smash{\begin{tabular}[t]{l}$z^{*}$\end{tabular}}}}%
    \put(0.54384391,0.07684753){\color[rgb]{0,0,0}\makebox(0,0)[lt]{\lineheight{1.25}\smash{\begin{tabular}[t]{l}$=$\end{tabular}}}}%
    \put(0,0){\includegraphics[width=\unitlength,page=4]{diagram15.pdf}}%
    \put(0.63932853,0.07226421){\color[rgb]{0,0,0}\makebox(0,0)[lt]{\lineheight{1.25}\smash{\begin{tabular}[t]{l}$x^{*}$\end{tabular}}}}%
    \put(0.71509403,0.09137844){\color[rgb]{0,0,0}\makebox(0,0)[lt]{\lineheight{1.25}\smash{\begin{tabular}[t]{l}$y^{*}$\end{tabular}}}}%
    \put(0.68559121,0.11914357){\color[rgb]{0,0,0}\makebox(0,0)[lt]{\lineheight{1.25}\smash{\begin{tabular}[t]{l}$z$\end{tabular}}}}%
    \put(0.68481376,0.01888706){\color[rgb]{0,0,0}\makebox(0,0)[lt]{\lineheight{1.25}\smash{\begin{tabular}[t]{l}$z$\end{tabular}}}}%
    \put(0.74318088,0.0750608){\color[rgb]{0,0,0}\makebox(0,0)[lt]{\lineheight{1.25}\smash{\begin{tabular}[t]{l}$=(\nu^{xy}_z)^{-2}d_z$\end{tabular}}}}%
  \end{picture}%
\endgroup%
 
    \end{center}
\end{enumerate}
\end{proof}
\vspace{2mm}
\noindent By (iv), we have that $(\nu^{xy}_z)^{2}=\dfrac{d_z}{B^{xy}_z}$ where $B^{xy}_z$ must satisfy (i)-(iii) above. Note that 
\[ (\nu^{x^{*}x}_0)^{2}=\dfrac{d_0}{B^{x^{*}x}_0}=\dfrac{1}{B^{x0}_x}=\dfrac{1}{d_x} \]
which is consistent with our cap and cup normalisation. The simplest such candidate (satisfying (i)-(iii)) for $B^{xy}_z$ is $\sqrt{d_x d_y d_z}$, whose corresponding normalisation is
\begin{equation}
    \nu^{xy}_z=\sqrt[4]{\frac{d_z}{d_x d_y}}
\end{equation}
(and consequently $\lambda_{zxyz}=\sqrt{\frac{d_x d_y}{d_z}}$) which is the convention used in the literature.\footnote{For example, an alternative choice for $B^{xy}_z$ could be $\frac{1}{3}(\sqrt{d_x d_y d_z}+d_x+d_y+d_z-1)$.} It is easy to see that the $F$-symbols (\ref{6jdef}) and $R$-symbols (\ref{braideigen}) are invariant under the prescribed normalisation. 


\newpage

\section{Rotated Morphisms}
\label{rotappx}
\noindent The rotation operator $\varphi$ as defined in Section \ref{prelims} has a simple generalisation to $\Hom(q^{\otimes m}, q^{\otimes n})$ (where $q$ is still assumed to be self-dual, and $m,n\in\mathbb{Z}_{>0}$). Let at least one of $m$ or $n$ be greater than one, and $f\in\Hom(q^{\otimes m}, q^{\otimes n})$. We have
\begin{equation}\centering \includegraphics[width=0.8\textwidth]{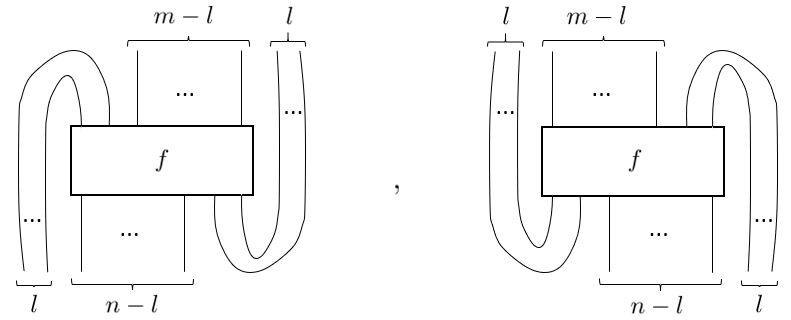}\label{quasigenrot}\end{equation}
where the left and right diagrams respectively illustrate $\varphi(f)$ for an (anti)clockwise rotation and where $1\leq l \leq\min\{m,n\}$. A further variant is studied in \cite{BJT}. 

\begin{example}
Let $m=n=3$ and $l=1$ with $\varphi$ anticlockwise. Then $\varphi^{6}=\id$. Let,
\[ f_{0}= \raisebox{-3.5ex}{\includegraphics[width=0.13\textwidth]{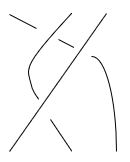}} \quad , \quad f_{1}=\raisebox{-3.5ex}{\includegraphics[width=0.13\textwidth]{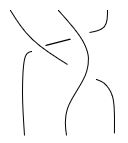}} \quad , \quad f_{2}=\raisebox{-3.5ex}{\includegraphics[width=0.11\textwidth]{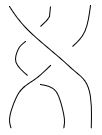}} \]
\[ f_{3}= \raisebox{-3.5ex}{\includegraphics[width=0.13\textwidth]{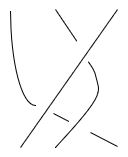}} \quad , \quad f_{4}=\raisebox{-3.5ex}{\includegraphics[width=0.13\textwidth]{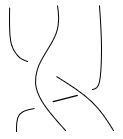}} \quad , \quad f_{5}=\raisebox{-3.5ex}{\includegraphics[width=0.11\textwidth]{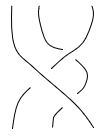}} \]

\noindent Observe that $\varphi(f_{\overline{k}})=\varkappa_{q}f_{\overline{k+1}}$ where $\overline{k}$ denotes a residue modulo $6$, and that $f_{\overline{k+3}}=f_{\overline{k}}$. Suppose there exist $\alpha,\beta\in\mathbb{C}$ such that $f_{2}=\alpha f_{0}+\beta f_{1}$. Applying $\varphi$, we get
\begin{align*}
f_{0}=\alpha f_{1} + \beta f_{2}&=\alpha f_{1} + \beta(\alpha f_{0}+\beta f_{1}) \\
\implies f_{1}&=\dfrac{1-\alpha\beta}{\alpha+\beta^{2}}f_{0} 
\end{align*}
where in the final line we assume that $\beta^{3}\neq-1$ and $\alpha\neq\beta^{-1}$. Thus, $f_{0},f_{1}$ and $f_{2}$ are either (a) linearly independent, (b) linearly dependent with $f_{2}=\alpha f_{0}+\beta f_{1}$ such that $\beta^{3}=-1$ and $\alpha=\beta^{-1}$ or (c) collinear. In the collinear case, note that $f_{\overline{k}}$ is an eigenvector of $\varphi$; coupling this with the fact that $\varphi^{3}(f_{\overline{k}})=\varkappa_{q}f_{\overline{k}}$, we have that $\varphi(f_{\overline{k}})=\omega f_{\overline{k}}$ where $\omega$ is a $3^{rd}$ root of $\varkappa_{q}$. It follows that $f_{\overline{k+1}}=\varkappa_{q}\omega f_{\overline{k}}$ (and so all of the morphisms are related to one another by a scaling of some $6^{th}$ root of unity).
\end{example}

\newpage

\section{Unitary Representations of the Braid Group}
\label{reptheory}
We now review part of the exposition in Section \ref{skeinrev} from a slightly different perspective. While much of the discourse here is well-known, we feel that it would be amiss to exclude this material from our presentation. The $n$-strand braid group is given by
\begin{equation}B_{n}=\left\langle\begin{array}{c|c}
      \sigma_{1},\ldots,\sigma_{n-1} & \begin{array}{c}
	\sigma_{i}\sigma_{i+1}\sigma_{i}=\sigma_{i+1}\sigma_{i}\sigma_{i+1} \\
	\sigma_{i}\sigma_{j}=\sigma_{j}\sigma_{i} \ , \ |i-j|\geq2
\end{array} \end{array}\right\rangle\end{equation}
whose graphical interpretation is given in Figure \ref{braidgens}. We also have the symmetric group
\begin{equation}S_{n}=\left\langle\begin{array}{c|c}
      s_{1},\ldots,s_{n-1} & \begin{array}{c}
      s_{i}^{2}=e \\
	s_{i}s_{i+1}s_{i}=s_{i+1}s_{i}s_{i+1} \\
	s_{i}s_{j}=s_{j}s_{i} \ , \ |i-j|\geq2
\end{array} \end{array}\right\rangle\end{equation}
\begin{figure}[H]\centering 
\def\svgwidth{14cm}
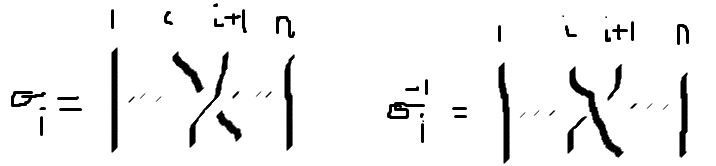 
\caption{Braid words are read from right-to-left. Braids are drawn and composed from top-to-bottom in accord with our pessimistic convention.} \label{braidgens}\end{figure}
\vspace{-2mm}
\noindent There is an epimorphism $\psi : B_{n}\to S_{n}$ where $\psi(\sigma_{i}^{\pm1})=s_{i}$. The pure braid group $PB_{n}$ is a normal subgroup of $B_{n}$ given by $\ker\psi$. That is, $\faktor{B_{n}}{PB_{n}}\cong S_{n}$. There is a closely related quotient for the algebra $\mathbb{C}[B_{n}]$; namely, we take the ideal $Q(\sigma_{i})$ generated by $(\sigma_{i}-r_{1})(\sigma_{i}-r_{2})$ where $r_{1},r_{2}\in\mathbb{C}^{\times}$. Then
\begin{equation}\faktor{\mathbb{C}[B_{n}]}{Q(\sigma_{i})}\cong H_{n}(r_{1},r_{2})\end{equation}
where $H_{n}(r_{1},r_{2})$ is called the \textit{Iwahori-Hecke algebra}. Indeed, $H_{n}(\pm1,\mp1)\cong\mathbb{C}[S_{n}]$ (and so the Iwahori-Hecke algebra can be thought of as a \textit{deformation} of $\mathbb{C}[S_{n}]$). Let $T_{1},\ldots,T_{n-1}$ be the generators of the $H_{n}(r_{1},r_{2})$. The generators satisfy relations
\begin{subequations}\begin{align}
(T_{i}-r_{1})(T_{i}-r_{2})&=0 \label{heckereln}\\
T_{i}T_{i+1}T_{i}&=T_{i+1}T_{i}T_{i+1} \\
T_{i}T_{j}&=T_{j}T_{i} \ , \ |i-j|\geq2
\end{align}\end{subequations}
where (\ref{heckereln}) is called the \textit{Hecke relation}. Viewing the Iwahori-Hecke algebra as a vector space, we have $\dim_{\mathbb{C}}(H_{n})=n!$. The \textit{generalised Hecke algebra} $H_{n}(Q,k)$ is given by the quotient of $\mathbb{C}[B_{n}]$ by the ideal $Q(\sigma_{i})$ which is now generated by $\Pi_{j=1}^{k}(\sigma_{i}-r_{j})$ where $r_{j}\in\mathbb{C}^{\times}$ and $k\geq2$. $H_{n}(Q,k)$ has the same presentation as the Iwahori-Hecke algebra except that (\ref{heckereln}) is now replaced with the generalised Hecke relation $\Pi_{j=1}^{k}(T_{i}-r_{j})=0$.\\

\noindent Let $\mathcal{C}$ be a ribbon fusion category and take some $q\in\Irr(\mathcal{C})$. Then
\begin{equation}\End\left(q^{\otimes n}\right)=\bigoplus_{X}\Hom\left(q^{\otimes n},X\right)\otimes\Hom\left(X,q^{\otimes n}\right)\end{equation}
 where $X$ indexes all the simple objects appearing in the decomposition of $q^{\otimes n}$. Fixing a fusion basis on $\Hom\left(q^{\otimes n},x\right)$ for some $X=x$ defines a linear representation 
\begin{equation}\rho:B_{n}\to U(V^{q^{n}}_{x}) \quad , \quad V^{q^{n}}_{x}:=\Hom\left(q^{\otimes n},x\right)\label{bgrep}\end{equation}
where $U(V^{q^{n}}_{x})$ denotes the group of unitary matrices on $V^{q^{n}}_{x}$. Let $n\geq2$. There exists at least one $i$ such that $\rho(\sigma_{i})=\mathcal{R}$, where $\mathcal{R}$ is a diagonal matrix whose eigenvalues are some subset of the eigenvalues of $R^{qq}$ (eigenvalues are counted without multiplicity here). Let $\{r_{1},\ldots,r_{k}\}$ denote the eigenvalues of $\mathcal{R}$ where the \mbox{$r_{i}\in U(1)$} are distinct and may appear in $\mathcal{R}$ with arbitrary nonzero multiplicity. We define
\begin{equation}p(Z)=(Z-r_{1}I_{s})\cdot\ldots\cdot(Z-r_{k}I_{s})\end{equation}
where $k\leq s:=\dim(V^{q^{n}}_{x})$, $I_{s}$ is the $s\times s$ identity matrix and $Z$ is an $s\times s$ matrix with entries in $\mathbb{C}$. It is clear that $p(\mathcal{R}')=0$ (where $\mathcal{R}'$ is any matrix similar to $\mathcal{R}$); this is an instance of the Cayley-Hamilton theorem. It follows that $p\left(\rho(\sigma_{i})\right)=0$ for all $i$, whence\footnote{In (\ref{factheck1}), $\rho$ is a $\mathbb{C}$-linear extension of  $\rho$ in (\ref{bgrep}). Through an abuse of notation, we implicitly assume that the representation in (\ref{factheck1}) is restricted to $B_{n}$ so as to coincide with (\ref{bgrep}).}
\begin{equation}\rho:\mathbb{C}[B_{n}]\to H_{n}(Q,k) \to U(V^{q^{n}}_{x})\label{factheck1}\end{equation}
i.e.\ $\rho$ factors through the generalised Hecke algebra $H_{n}(Q,k)$.\\

In Section \ref{trivcase}, we considered a fusion rule $q\otimes q=\bm{1}$. Clearly, $\px=z\ \nx$ for some $z\in\mathbb{C}^{\times}$. Applying (\ref{rotx}), we see that $z=\pm1$ for $\varkappa_{q}=\pm1$. For $\varkappa_{q}=+1$, (\ref{factheck1}) becomes
\begin{equation}\begin{split}\rho^{(u)}:\ \mathbb{C}[B_{n}] &\to \mathbb{C}[S_{n]} \to U(1) \\
\sigma_{j}&\longmapsto \ s_{j} \ \longmapsto u \end{split} \quad , \quad u=\pm1 \label{trivrep1}\end{equation}
For $\varkappa_{q}=-1$, (\ref{factheck1}) becomes
\begin{equation}\begin{split}\rho^{(u)}:\ \mathbb{C}[B_{n}] &\to \mathbb{C}[S_{n]} \to U(1) \\
\sigma_{j}&\longmapsto \pm is_{j} \longmapsto u \end{split} \quad , \quad u=\pm i \label{trivrep2}\end{equation}

In Section \ref{quadcase}, we considered a fusion rule $q\otimes q=\bm{1}\oplus y$. This means that our crossings can be written
\begin{subequations}\begin{align}
\Px=a \ \Idm + b \ \Ccm \label{ronin1}\\
\Nx=c \ \Idm + d \ \Ccm \label{ronin2}
\end{align}\end{subequations}
where $a,b,c,d\in\mathbb{C}^{\times}$. This motivates the idea that the homomorphism (\ref{factheck1}) should also factor through some algebra of cup-cap diagrams and non-intersecting strands (for $n\geq3$); this is precisely the \textit{Temperley-Lieb algebra} $TL_{n}(\delta)$: an associative $\mathscr{A}$-algebra (where $\mathscr{A}$ is a commutative ring) with generators $U_{1},\ldots,U_{n-1}$ satisfying relations
\begin{subequations}\begin{align}
U_{i}^{2}&=\delta U_{i} \ \ \ , \ \delta\in\mathscr{A}\\
U_{i}U_{j}U_{i}&=U_{i} \quad \ , \ |i-j|=1 \\
U_{i}U_{j}&=U_{j}U_{i} \ , \ |i-j|\geq2
\end{align}\end{subequations}
\begin{figure}[H]\centering 
\def\svgwidth{10cm}
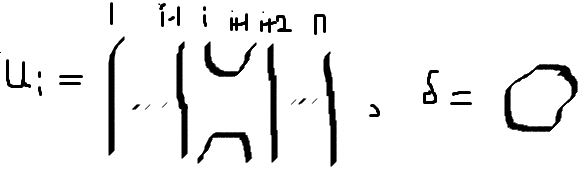 
\caption{Diagrams run from top-to-bottom. The identity element is given by $n$ vertical strands.} \label{tlgens}\end{figure}
\noindent $TL_{n}(\delta)$ is a free $\mathscr{A}$-module of rank $C_{n}$ where $C_{n}$ denotes the $n^{th}$ Catalan number. Following (\ref{ronin1})-(\ref{ronin2}), we construct a $\mathbb{C}$-linear map 
\begin{equation}\begin{split}
\hspace{25mm}\zeta:\mathbb{C}[B_{n}]&\to TL_{n}(\delta) \\
\hspace{25mm}\sigma_{i}&\mapsto a+bU_{i} \\
\hspace{25mm}\sigma_{i}^{-1}&\mapsto c+dU_{i}
\end{split}\quad , \quad \mathscr{A}=\mathbb{C}[a^{\pm1},b^{\pm1},c^{\pm1},d^{\pm1}]\end{equation}
\begin{proposition}
$\zeta$ defines an algebra homomorphism if and only if $c=a^{-1}, d=b^{-1}$ and $\delta=-(ab^{-1}+a^{-1}b)$. 
\label{yuzu}\end{proposition}
\noindent We henceforth assume that $c, d$ and $\delta$ are as in Proposition \ref{yuzu}. Since our representation 
\begin{equation}\rho:\mathbb{C}[B_{n}]\stackrel{\zeta}{\to} TL_{n}(\delta)\to U(V^{q^{n}}_{x})\label{tlfactrep}\end{equation}
is unitary, the conditions in Proposition \ref{luadap} (adapted from \cite[p.237]{lureps}) must hold.
\begin{proposition}Given $\rho$ as in (\ref{tlfactrep}), we have $U_{i}^{\dagger}=U_{i}$ and $|a|=|b|=1$.\label{luadap}\end{proposition}
\noindent From (\ref{factheck1}) we know that $\rho$ must also factor through $H_{n}(Q,k)$. Since we are considering a fusion rule of the form $q^{\otimes2}=y\oplus z$, we have that 
\begin{equation}\rho:\mathbb{C}[B_{n}]\to H_{n}(r_{1},r_{2}) \to U(V^{q^{n}}_{x})\label{ihfactrep}\end{equation}
It is easy to check that (\ref{tlfactrep}) is compatible with (\ref{ihfactrep}). Following Proposition \ref{luadap},
\begin{align*}U_{i}^{\dagger}=U_{i}&\iff \left[b^{-1}\left(\rho(\sigma_{i})-a\right)\right]^{\dagger}=b^{-1}\left(\rho(\sigma_{i})-a\right) \\
&\iff b\left(\rho(\sigma_{i})\right)^{\dagger}-a^{*}b=b^{*}\rho(\sigma_{i})-b^{*}a \\
&\iff \left(\rho(\sigma_{i})+a^{*}b^{2}\right)\left(\rho(\sigma_{i})-a\right)=0
\end{align*}
whence
\begin{equation}\rho:\mathbb{C}[B_{n}]\to H_{n}(-a^{*}b^{2},a) \to TL_{n}(\delta)\to U(V^{q^{n}}_{x})\end{equation}
Specifically, we have the following commutative diagram of linear homomorphisms: 
\vspace{-4mm}
\begin{equation*}
\begin{minipage}{.45\textwidth}
\begin{figure}[H]
\adjustbox{scale=1.075,center}{
\begin{tikzcd}
\mathbb{C}[B_{n}] \arrow[r,"\phi"] \arrow[d,"\zeta"] & H_{n}(-a^{-1}b^{2},a) \arrow[dl, "\eta"]\\
TL_{n}(\delta) \arrow[d, "\rho' "] \\
U(V^{q^{n}}_{x})
\end{tikzcd}}
\end{figure}
\end{minipage} \hfill
\begin{minipage}{.3\textwidth}
\vspace{5mm}
\begin{align*}
&a , b \in U(1)\\
\text{with} \quad \quad &\delta = -(ab^{-1}+a^{-1}b)\\
&\rho = \rho' \circ \zeta  
\end{align*}
\end{minipage}
\end{equation*}
\noindent where $\phi(\sigma_{i})=T_{i}$ such that $\ker\phi$ is generated by $(\sigma_{i}+a^{-1}b^{2})(\sigma_{i}-a)$, and $\eta(T_{i})=a+bU_{i}$. We may thus resolve crossings using skein relations
\begin{equation}\begin{split}
\Px&=\ a \ \ \Idm +  b \ \ \Ccm \\
\Nx&=a^{-1} \Idm + b^{-1} \Ccm 
\end{split}\quad , \quad \Loopy=-(ab^{-1}+a^{-1}b)\label{homflyfrskein}\end{equation}
Note that the resolution of the crossings in (\ref{homflyfrskein}) implies 
\begin{equation}\Px-b^{2}\ \Nx=(a-a^{-1}b^{2})\ \Idm \label{whirligig}\end{equation}
which is simply the Hecke (skein) relation for $H_{n}(-a^{-1}b^{2},a)$.
To recover the boxed result in Section \ref{quadcase}, we consider relations (\ref{homflyfrskein}) in the setting of a ribbon category. Note that
\begin{equation}\vartheta_{q}d_{q}=\raisebox{-2.1ex}{\includegraphics[width=0.11\textwidth]{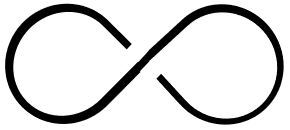}} =a \ \Loopy \ \Loopy + b \ \Loopy=-a^{2}b^{-1} \ \Loopy\label{quadrib}\end{equation}
whence $\vartheta_{q}=-a^{2}b^{-1}$. But we also have that $R^{qq}_{0}=a^{-1}b^{2}$ whence (\ref{fstwisteq}) tells us that $b=\pm a^{-1}$ for $\varkappa_{q}=\pm1$. Another way of seeing this is by applying (\ref{rotx}) to (\ref{homflyfrskein}).

\begin{remark}
The skein relations (\ref{homflyfrskein}) correspond to the \textit{framed} HOMFLY-PT polynomial. In order to see this, consider the well-known Lickorish-Millet presentation \cite{lickmill} of the HOMFLY-PT skein relations,
\begin{equation}l \ \Px + l^{-1}\  \Nx + m \ \Idm = 0 \quad , \quad \Loopy=1\end{equation}
Then setting
\begin{equation}l=\pm ib^{-1} \quad , \quad m=\mp i(ab^{-1}-a^{-1}b)\label{monrroe}\end{equation}
recovers (\ref{whirligig}). Finally, with $l$ and $m$ as in (\ref{monrroe}),\vspace{-2mm}

\begin{equation*}\begin{split}
& l \ \raisebox{-2.1ex}{\includegraphics[width=0.11\textwidth]{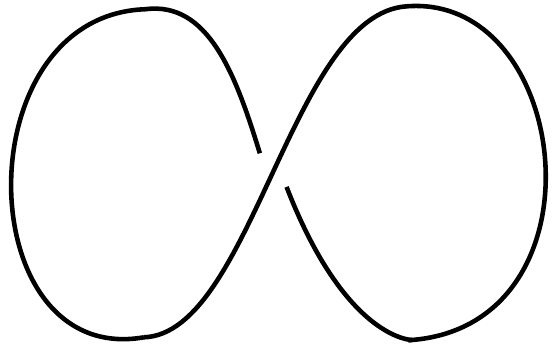}} + l^{-1}\ \raisebox{-2.1ex}{\includegraphics[width=0.11\textwidth]{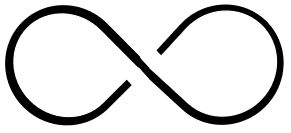}}+ m \  \Loopy \ \Loopy = 0  \\
\stackrel{(\ref{homflyfrskein})}{\implies} & -a^{2}b^{-1}l \ \Loopy -a^{-2}bl^{-1}\ \Loopy + m \ \Loopy \ \Loopy = 0 \\
\implies & -(ab^{-1}+a^{-1}b) \ \Loopy= \Loopy \ \Loopy
\end{split}\end{equation*}
whence we have the rescaled loop value $\loopy \ = -(ab^{-1}+a^{-1}b)$. Let $\tilde{H}$ denote the framed HOMFLY-PT polynomial, $L$ be a link and $D$ a corresponding link diagram. Then the (unframed) HOMFLY-PT polynomial $H$ is simply
\begin{equation}H(L)=(-a^{-2}b)^{w(D)}\tilde{H}(D)\end{equation}
where $w$ denotes the writhe. The HOMFLY-PT invariant can be derived by applying a normalised Markov trace to the Iwahori-Hecke algebra; this trace is characterised by its action on the basis elements of the HOMFLY-PT skein algebra of the annulus\footnote{Given $H_{n}(r_{1},r_{2})$, the HOMFLY-PT skein algebra $\mathcal{H}_{n}(r_{1},r_{2})$ is obtained by joining the ends of the strands (where the $i^{th}$ top is respectively connected to the $i^{th}$ bottom). We have $\mathcal{H}_{n}\cong\faktor{H_{n}}{[\cdot,\cdot]}$ (i.e. the quotient by the ideal generated by the commutator) and $\dim_{\mathbb{C}}\left(\mathcal{H}_{n}\right)$ is given by the $n^{th}$ partition number.} \cite{kasselturaev, bigelow}.
\end{remark}

We omit an analogous discussion for the fusion rule $q\otimes q=\bm{1}\oplus x \oplus y$, and solely remark that since there must exist $p_{1},p_{2},p_{3},p_{4}$ nonzero such that
\begin{equation}p_{1} \ \Ccm + p_{2} \ \Idm + p_{3} \ \Px + p_{4} \ \Nx = 0 \end{equation}
this indicates that the representation should factor through the Temperley-Lieb algebra, which in turn motivates the construction of a linear map
\begin{equation}\begin{split}
\zeta' : \mathbb{C}[B_{n}] &\to TL_{n}(\delta) \\
\sigma_{i}+a\sigma_{i}^{-1} &\mapsto b + cU_{i}
\end{split}\end{equation}
where $a,b,c\in\mathbb{C}^{\times}$ and $\delta\in\mathbb{C}[a^{\pm1},b^{\pm1},c^{\pm1}]$ are such that $\zeta'$ defines a homomorphism. It is also clear that the representation should factor through $H_{n}(Q,3)$. 

\newpage
\section{Invariants coming from $\varkappa_{q}=-1$}
\label{addendum}

Given a skein relation associated to the fusion rule (\ref{ourguysec3}) for $\varkappa_{q}=-1$, we can define a polynomial-valued function $\Lambda_{\C,q}$ that acts on any link diagram $D$. Since invariants $\Lambda_{\C,q}$ coming from $\varkappa_{q}=-1$ are derived from a setting where an isotopy of the form $\px\stackrel{\varphi}{\longrightarrow}\raisebox{-4.5mm}{\includegraphics[width=0.065\textwidth]{fsrot}}=-\nx$ introduces a difference in sign, it is natural to ask whether the invariant $\Lambda_{\C,q}$ carries such a sensitivity. 

\begin{problem}\label{jooksing}Let $D$ and $D'$ be link diagrams that are equivalent under framed isotopy, and let $k:=\frac{1}{2}|\mathcal{W}(D)-\mathcal{W}(D')|$ where $\mathcal{W}$ denotes the \textit{local writhe}\footnote{The local writhe $\mathcal{W}$ of a link diagram $D$ is defined as the sum over the signs of all crossings in $D$, where $\px$ and $\nx$ respectively contribute $+1$ and $-1$.}. Is it always true that 
\begin{equation}\Lambda_{\C,q}(D')=(-1)^{k}\Lambda_{\C,q}(D)\label{decisiveassault}\end{equation}
given $\Lambda_{\C,q}$ for some $q$ with $\varkappa_q=-1$?\end{problem}

\subsection{Understanding the invariant (\ref{kbtwinskein})}
\label{mindfuck}
The skein relation (\ref{kbtwinskein}) must be applied \textit{locally} i.e. the link must be without orientation, and the relation should be applied to crossings precisely as they appear.\\

\noindent For instance, while a twist of the form \raisebox{-2mm}{\includegraphics[width=0.04\textwidth]{kinkhorz1}} would  be resolved as a positive crossing if it were oriented, we only take into account the local form of the crossing (which is negative here). Letting  $[\cdot]_{\beta}$ denote the invariant (\ref{kbtwinskein}),
\begin{equation} \left[ \raisebox{-6mm}{\includegraphics[width=0.05\textwidth]{kinkhorz3}}\right]_{\beta}=\beta^{3} \ \raisebox{-7mm}{\includegraphics[width=0.019\textwidth]{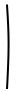}} \quad , \quad \left[\raisebox{-3mm}{\includegraphics[width=0.06\textwidth]{kinkhorz1}}\right]_{\beta}=-\beta^{3}\ \raisebox{-2mm}{\includegraphics[width=0.06\textwidth]{kinkhorz2}}\end{equation}
whence we see that the categorical distinction between horizontal and vertical twists (Remark \ref{kinkhorz}) carries over to the invariant (\ref{kbtwinskein}). Of course, this observation is subsumed by the following:

\begin{proposition}
\label{omglol}
The answer to Problem \ref{jooksing} is positive when $q^{\otimes2}=\bm{1}\oplus x$.
\begin{proof}
In this case, we know that $\Lambda:=\Lambda_{\C,q}$ is given by (\ref{kbtwinskein}). In the following, we assume that all diagrams are projected onto the plane. It suffices to consider an isotopy $D\to\check{D}$ which is one of: (i) a Reidemeister-$0$ move (i.e.\ an ambient isotopy) such that only one crossing has its sign flipped under the deformation (so $k=1$); (ii) a Reidemeister-II move; (iii) a Reidemeister-III move.
\begin{enumerate}[label=(\roman*)]
\item Given a link diagram $D$ containing some crossing \raisebox{-3mm}{\includegraphics[width=0.05\textwidth]{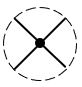}}, we let $D_0$ and $D_\infty$  respectively denote the same diagram but with the crossing smoothed to \raisebox{-3mm}{\includegraphics[width=0.05\textwidth]{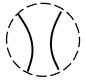}} and \raisebox{-3mm}{\includegraphics[width=0.05\textwidth]{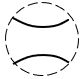}}. We consider $\Lambda(D)$ and $\Lambda(\check{D})$, first applying the skein relation to the crossing whose sign was flipped: suppose (without loss of generality) that the crossing in $D$ was $\px$\ . Then
\vspace{-2mm}
\begin{align*}
\Lambda(D)&=\beta\cdot\Lambda(D_0)-\beta^{-1}\cdot\Lambda(D_\infty) \\
\Lambda(\check{D})&=\beta^{-1}\cdot\Lambda(\check{D}_0)-\beta\cdot\Lambda(\check{D}_\infty)
\end{align*}
\vspace{-3mm}
Label the boundary points of the crossing in $D$ as \raisebox{-4mm}{\includegraphics[width=0.06\textwidth]{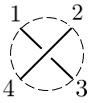}}, and so in $\check{D}$ said crossing is either \raisebox{-4mm}{\includegraphics[width=0.06\textwidth]{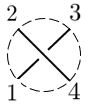}} or \raisebox{-4mm}{\includegraphics[width=0.06\textwidth]{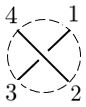}}\ . Then it is easy to see that the smoothings $D\to D_0$ and $\check{D}\to \check{D}_\infty$ are locally identical, whence the diagrams $D_0$ and $\check{D}_\infty$ are also equivalent under ambient isotopy. It follows that $\Lambda(D_0)=\Lambda(\check{D}_\infty)$. Similarly, $\Lambda(\check{D}_0)=\Lambda(D_\infty)$. Thus, $\Lambda(\check{D})=-\Lambda(D)$. 
\vspace{5mm}
\item Two crossings have their signs flipped under a Reidemeister-II move (so \mbox{$k=2$}). Thus, (\ref{decisiveassault}) respects the invariance of $\Lambda$ under Reidemeister-II i.e. \mbox{$\Lambda(\check{D})=\Lambda(D)$.}
\vspace{2mm}
\item Either zero or two crossings have their signs flipped under a Reidemeister-III move (so $k=0$ or $2$). It follows that (\ref{decisiveassault}) respects the invariance of $\Lambda$ under Reidemeister-III i.e. $\Lambda(\check{D})=\Lambda(D)$.
\end{enumerate}
\vspace{1.5mm}
Since $D\to D'$ in Problem \ref{jooksing} is a composition of moves (i)-(iii), the result follows.
\end{proof}
\end{proposition}

\noindent Thus, $\left[D\right]_{\beta}$ is invariant under framed isotopy \textit{up to a sign} which depends on $\mathcal{W}(D)$ e.g.
\begin{equation} \left[ \raisebox{-4mm}{\includegraphics[width=0.12\textwidth]{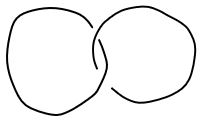}}\right]_{\beta}=-\ \left[ \raisebox{-8mm}{\includegraphics[width=0.11\textwidth]{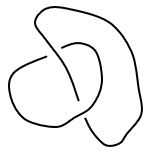}}\right]_{\beta}=\left[ \raisebox{-6.5mm}{\includegraphics[width=0.065\textwidth]{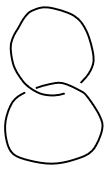}}\right]_{\beta} \label{huzzah}\end{equation}

\noindent From Figure \ref{bokchoi}, we see that (\ref{huzzah}) is equal to $\beta^{2}d^{2}-d-d+\beta^{-2}d^{2}$ (where $d:=\beta^2+\beta^{-2}$ is the value of the loop). 

\begin{figure}[H]\centering \includegraphics[width=0.39\textwidth]{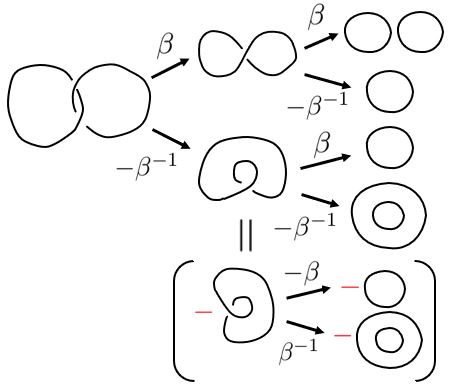} \caption{When we apply the skein relation (\ref{kbtwinskein}), we must take care to account for any minus signs accumulated if we choose to rotate crossings mid-evaluation e.g.\ if the parenthesised route is taken above (for which the resulting minus signs are highlighted in red).} \label{bokchoi}\end{figure}


In summary, we have seen that while (\ref{kbtwinskein}) is invariant under Reidemeister-II and III moves, it retains sensitivity to pivotality at some level: when a link diagram is isotoped in a way that locally rotates crossings, this introduces ``internal zig-zags'' that are detected by the invariant.

\begin{remark}\hspace{2mm}
\begin{enumerate}[label=(\roman*)]
\item In light of the above, referring to invariants $\Lambda_{\C,q}$ associated to $q$ antisymmetrically self-dual as ``framed link invariants" could be considered an abuse of terminology. We therefore propose the term \textit{pivotal framed invariant} for (\ref{kbtwinskein}). Similarly, if some invariant associated to $\varkappa_q=-1$ is sensitive to the local writhe in the same way as (\ref{kbtwinskein}), then it should also be identified as a pivotal framed invariant.\footnote{E.g. if the answer to Problem \ref{jooksing} is positive, then we propose that all invariants associated to $\varkappa_q=-1$ should be termed as such.} A pivotal framed invariant $\Lambda_{\C,q}$ may be normalised as in (\ref{woofwatch}), thus removing its sensitivity to framing. Nonetheless, the pivotal discrimination remains: we therefore propose that the resulting invariant should similarly be called an (oriented) \textit{pivotal invariant}. 
\vspace{2mm}

\item Although the topological motivation for studying (\ref{kbtwinskein}) and other $\Lambda_{\C,q}$ for $\varkappa_{q}=-1$ may be unclear, we suggest the following counterpoint: in the categorical setting, these invariants arise naturally; they give the evaluation of a link diagram (in $\End(\bm{1})$ whose components are all labelled by $q$) up to a possible factor of $-1$ coming from zig-zag morphisms. Whether these invariants have more interesting applications outside the categorical context remains to be seen.

\vspace{2mm}
\item One can obtain the skein relation
\begin{equation}\beta\OPx{}{}-\beta^{-1} \ \ONx=(\beta^{2}-\beta^{-2})\OIdm{}{} \label{purringmachine}\end{equation}
from (\ref{kbtwinskein}), noting that we have further endowed it with pessimistic orientation. This skein relation defines a framed oriented link invariant. Furthermore, the resulting polynomial of a framed oriented link $L$ coincides with the Kauffman bracket polynomial of the framed link $\tilde{L}$ (which is $L$ without orientation). In order to see this, take an oriented link diagram $D$ for $L$. We can isotope $D$ to obtain a diagram $D'$ whose crossings locally appear with pessimistic orientation. Now observe that applying the oriented skein relation (\ref{purringmachine}) to $D'$ is the same as applying the skein relation (\ref{whirligig}) (with $b=a^{-1}$ and $a=\beta$) to the link diagram $\tilde{D}$ (which is the diagram $D'$ without orientation).
\end{enumerate}
\end{remark}

\end{document}